\documentclass[11pt]{article}
\usepackage{amsmath}
\usepackage{amssymb}
\usepackage{amsthm}
\usepackage{bbm}
\usepackage{mathrsfs}
\usepackage{bm}
\usepackage[dvipdfmx, usenames]{color}
\usepackage[numbers]{natbib}  
\usepackage{hyperref}
\theoremstyle{plain} 
\newtheorem{thm}{Theorem}[section]
\newtheorem{lem}[thm]{Lemma}
\newtheorem{prop}[thm]{Proposition}
\newtheorem{cor}[thm]{Corollary}
\theoremstyle{plain}
\newtheorem{defn}[thm]{Definition}
\theoremstyle{remark}
\newtheorem{rem}{Remark}[section]

\DeclareMathOperator{\diver}{div}
\DeclareMathOperator{\push}{\! _\# \! }
\DeclareMathOperator*{\argmin}{argmin}

\DeclareMathOperator{\supp}{spt}

\DeclareMathOperator{\loc}{loc}

\newcommand{\stint}[3]{\int_{#1}^{#2}\!\!\!\int_{#3}}
\numberwithin{equation}{section}

\title{
Formulation of Chimera Gradient Flows for Chemotaxis Systems with Indirect Signal Production and Degenerate Diffusion}

\author{Yoshifumi MIMURA\thanks{mimura.yoshifumi@nihon-u.ac.jp}
\\ 
Department of Mathematics, 
College of Humanities \& Sciences, 
 Nihon University
 \\ 
 3-25-40 Sakurajosui Setagaya-ku Tokyo 156-8550, Japan
}

\date{}

\begin{document}

\maketitle

\begin{abstract}
A parabolic system of three unknown functions, not expressible as gradient flows, is treated as three coupled gradient flows. For each unknown function, the minimizing movement scheme is used to construct a time-discrete approximate solution. Unlike standard minimizing movement scheme for gradient flows, the relative compactness of the time-discrete approximate solution with respect to the time step is not inherently guaranteed. However, the existence of a Lyapunov functional ensures this relative compactness, leading to the existence of time-global solutions.
\end{abstract}

\noindent
{\sc keywords}:{\ Keller-Segel; Wasserstein distance; 
Minimizing movement scheme; Global existence; degenerate diffusion}

\noindent
{Mathmatics Subjet Classification: 35K65, 35K40, 47J30, 35Q92, 35B33}.

\section{Introduction}
We consider the following system.
\begin{equation} 
\label{P}
\begin{cases}
\begin{split} 
\partial_t u &= \Delta u^m - \nabla \cdot (u\nabla v),
 & \text{$x \in \mathbb{R}^d,\ t>0$}, \\
\varepsilon_1 \partial_t v &= \kappa_1 \Delta v - \gamma_1 v +  w, 
& \text{$x \in \mathbb{R}^d,\ t>0$}, \\ 
\varepsilon_2 \partial_t w &= \kappa_2 \Delta w -  \gamma_2 w  +  u, 
& \text{$x \in \mathbb{R}^d,\ t>0$}, \\
u(x,0)& = u_0 \geq 0, \ 
\varepsilon_1 v(x,0)= \varepsilon_1 v_0 \geq 0, \ 
\varepsilon_2 w(x, 0)=\varepsilon_2 w_0\geq 0, 
  &  \text{$x \in \mathbb{R}^d$}, 
\end{split}
\end{cases}
\end{equation}
where $\varepsilon_i, \kappa_i>0$ and $\gamma_i \geq 0$ for $i=1,2$. 
In this paper, we consider the case where
\begin{equation*} 
\begin{split} 
m\geq 2-\frac{4}{d}, \quad d\geq 5. 
\end{split}
\end{equation*}
Assuming the integrability of $\nabla u^m$ and $u\nabla v$, 
it follows from the first equation of system \eqref{P} that 
the conservation law of the mass 
\begin{equation*}
\begin{split}
M:=\int_{\mathbb{R}^d} u(x)\,dx=\int_{\mathbb{R}^d} u_0(x)\,dx
\end{split}
\end{equation*}
is satisfied.

The objectives of this paper are threefold:
\begin{itemize}

\item[1.]
to provide sufficient conditions for the existence of global-in-time weak solutions to problem \eqref{P},

\item[2.]
to demonstrate an application of the minimizing movement approach to a system of evolution equations that is not a gradient flow,

\item[3.]
to propose a method that allows for handling Lyapunov functionals without concern for the smoothness of the solution.

\end{itemize}

Let us briefly explain the system \eqref{P}. 
In conditions where 
\begin{equation} \label{ks}
\begin{split} 
\text{$\kappa_1 \neq 0, \varepsilon_2=\kappa_2=0$
\quad \text{ or }\quad 
$\kappa_2 \neq 0, \varepsilon_1=\kappa_1=0$}
\end{split}
\end{equation}
as specified in equation \eqref{P}, the system is well-known as the Keller-Segel system. The existence of a critical mass threshold $M_c$
has been predicted and studied: if $\|u_0\|_{L^1}<M_c$, 
solutions exist globally in time, whereas for any 
$M > M_c$, there exists a solution with $\|u_0\|_{L^1} =M$ 
that blows up in finite time. 
The Keller-Segel system, under the conditions of \eqref{ks}, 
is further classified into fully parabolic and parabolic-elliptic types 
depending on whether 
\begin{equation*} 
\begin{split} 
\varepsilon_1+\varepsilon_2\neq 0 
\quad \text{ or }\quad \varepsilon_1+\varepsilon_2=0, 
\end{split}
\end{equation*}
respectively. 
Research concerning the aforementioned threshold $M_c$ 
has been reported for both cases. 
See for instance, \cite{b1998, b-c-c, b2015, b-l, g-z, ym1, ym2, n-s-y} 
for time global existence, 
\cite{b-z, h-i-y, h-v1, i-l-m, l-m, o-s-w, o-w} for the existence of blow-up solutions, 
\cite{b-c-l, b-d-p, c-c, s-t} for the threshold. 
Additionally, for a review of the Keller-Segel system up to recent years, see \cite{a-t} 
and the references therein.
Furthermore, the method of constructing time-global solutions 
by viewing the Keller-Segel system as a gradient flow can be considered a precursor to the approaches in this paper. For the parabolic-elliptic case with $m=1$ and $d=2$, 
it is formulated as a gradient flow in \cite{b-c-c}, 
and for the fully parabolic case with $m=1$ and $d=2$ in 
\cite{b2015}, 
and for the fully parabolic case with $m>1$ in \cite{b-l, ym1, ym2, ym3}. 
Additionally, while not specific to the so-called minimal models, see \cite{k-w} for an application of the gradient flow approach to the multi-species Keller-Segel system.

To view the Keller-Segel system as a gradient flow, the concept of Wasserstein distance is required. The Wasserstein distance is a notion of distance applicable to probability measures, and was utilized by Jordan-Kinderlehrer-Otto \cite{j-k-o} in constructing solutions to the Fokker--Planck equation. This method is referred to as the JKO-Scheme or the Minimizing Movement Scheme (MMS). the MMS is often used in a more general context than the original work of Jordan-Kinderlehrer-Otto, and it involves discussions in abstract metric spaces as seen in \cite{a-g-s}. Both the JKO-Scheme and 
the MMS involve discretizing the time interval, variationally constructing discrete solutions at each time step, and then obtaining the existence of solutions in the limit as the time discretization approaches zero. Due to their variational nature, by focusing on the minimal point,  solutions to partial differential equations can be derived via the Euler-Lagrange equations, and by focusing on the minimum values, the existence of curves of maximal slope, which do not require the concept of derivatives, can be established as alternatives to solutions of partial differential equations. In any case, the resulting partial differential equations are of the form known as gradient flows.

However, the concept of a gradient flow is inherently unstable. For instance, even a slight perturbation to the first or second equation of the Keller-Segel system could disrupt its structure as a gradient flow. In mathematical modeling, inaccuracies and subsequent model revisions are unavoidable issues, and the MMS, being specialized for gradient flows, has inherent limitations in terms of versatility. A notable example is the problem \eqref{P} where coefficients are not set to zero; unlike the Keller-Segel system, 
\eqref{P} under these conditions no longer fits within the framework of gradient flows. Hence, this paper proposes a slight modification to MMS to variationally construct solutions to problem \eqref{P}. This minor modification uniquely does not require the knowledge of Lyapunov functionals or their lower bounds during the construction phase of approximate solutions. Moreover, it not only addresses equations that are non-gradient in flow, but also significantly diverges from the MMS, which constructs solutions under the assumption that the Lyapunov functional 
is already known.

The existence of a threshold similar to that in the Keller-Segel system 
has been confirmed for problem \eqref{P}, 
specifically reported when $m=1$ and $d=4$. 
Fujie-Senba \cite{f-s} have proved
that in a bounded domain, under spherical symmetry with the boundary conditions 
\begin{equation*}
\frac{\partial u}{\partial \bm{\nu}}
=\frac{\partial v}{\partial \bm{\nu}}=\frac{\partial w}{\partial \bm{\nu}}=0,  
\quad x \in \partial \Omega,\ t>0,  
\end{equation*}
or 
the boundary conditions 
\begin{equation*}
\frac{\partial u}{\partial \bm{\nu}}- u\frac{\partial v}{\partial \bm{\nu}}=v
=w=0,  
\quad x \in \partial \Omega,\ t>0
\end{equation*}
without spherical symmetry, 
if $\|u_0\|_{L^1}<(8\pi)^2$, 
the solution remains global in time. 
Furthermore, Fujie-Senba \cite{f-s2}  
have proved the existence of solutions that blow up in finite time 
when $M$ exceeds $(8\pi)^2$ and 
is not a natural number multiple of $(8\pi)^2$, 
satisfying $\|u_0\|_{L^1} =M$. 
Additionally, Hosono-Ogawa \cite{h-o} deal with problems that are reduced 
to \eqref{P} through variable transformations, 
demonstrating that global-in-time solutions exist under the conditions 
$m=1$, $\varepsilon_1=\varepsilon_2=0$ and $d=4$ 
when $\|u_0\|_{L^1} < (8\pi)^2$. 
Also, Hosono-Lauren\c cot \cite{h-l} showed that 
when $m=1$, $\varepsilon_1=\varepsilon_2=1$ and $d=4$, 
time-global solutions exist if $\|u_0\|_{L^1}<(8\pi)^2$. 

Fujie-Senba \cite{f-s}, Hosono-Ogawa \cite{h-o} 
and Hosono-Lauren\c cot \cite{h-l} 
derive the Lyapunov functional or the modified Lyapunov functional  
through direct calculation from the existence of classical local solutions. 
Roughly speaking, the existence of time-global solutions 
requires the boundedness from below of the functional, 
for which Fujie-Senba and Hosono et al. apply Adams-type inequality and Brezis-Merle inequality, respectively. 
Contrary to their analysis, our problem includes a degenerate diffusion term in the first equation, which precludes the expectation of classical solutions. Consequently, deriving and utilizing the Lyapunov functional is challenging. Therefore, instead of direct differentiation, we derived the Lyapunov functional from its variational properties and obtained the existence of time-global solutions and the energy inequality as demonstrated below.

Before describing the main theorem of this paper, let us first clarify the concept of weak solutions that will be addressed herein.

\begin{defn}[time-global weak solutions]\label{weak solutions}
A triple $(u,v,w)$ of non-negative functions is defined as a 
\textbf{time-global weak solution} of \eqref{P} if it satisfies the following conditions: 

\begin{itemize}
\item[\textbf{\rm(i)}] For any $T>0$, the functions meet these regularity conditions:
\begin{itemize}
\item $u \in L^{\infty}(0,T; (L^1 \cap L^m)(\mathbb{R}^d))$,
\item $v \in L^{\infty}(0,T; H^1(\mathbb{R}^d))$,
\item $w \in L^{\infty}(0,T; L^2(\mathbb{R}^d))$,
\item $\|u(t)\|_{L^1}=\|u_0\|_{L^1}$ for all $t \in [0,T]$, \\
\item $\displaystyle \sup_{t \in [0,T]}\int_{\mathbb{R}^d}|x|^2\,u\,dx<\infty$. 
\end{itemize}

\item[\textbf{\rm(ii)}] For any $T>0$, the functions demonstrate additional regularity:
\begin{itemize}
    \item $u \in L^{2}(0,T;L^2(\mathbb{R}^d))$,
    \item $u^m \in L^{1}(0,T;W^{1,\frac{d}{d-1}}(\mathbb{R}^d))$,
    \item $v \in L^{2}(0,T;W^{3,2}(\mathbb{R}^d))$,
    \item $w \in L^{2}(0,T;W^{2,2}(\mathbb{R}^d))$.
\end{itemize}

\item[\textbf{\rm(iii)}] 
As $t$ approaches $0$ from above, the following initial conditions are met:
\begin{itemize}
    \item $\lim\limits_{t \downarrow 0}\mathcal{W}_2(u(t), u_0) = 0$,
    \item $\lim\limits_{t \downarrow 0}\|v(t) - v_0\|_{H^1} = 0$,
    \item $\lim\limits_{t \downarrow 0}\|w(t) - w_0\|_{L^2} = 0$,
\end{itemize}
where $\mathcal{W}_2$ denotes the Wasserstein distance defined in Section 2.

\item[\textbf{\rm(iv)}] The triplet $(u,v,w)$ satisfies the following system 
of equations
\begin{equation*}
\begin{cases}\vspace{2mm}
\displaystyle 
\stint{0}{\infty}{\mathbb{R}^d}
[u\partial_t \varphi- \langle \nabla u^m - u \nabla v, \nabla \varphi\rangle  ]\,dxdt=0, 
\quad \text{ for all }\varphi \in C^{\infty}_c(\mathbb{R}^d \times (0,\infty)), 
\\ \vspace{2mm}
\displaystyle 
\varepsilon_1 \partial_t v = 
\kappa_1 \Delta v - \gamma_1 v + w, 
\quad \text{ a.e. in }\mathbb{R}^d \times (0, \infty), 
\\
\vspace{2mm}
\displaystyle 
\varepsilon_2 \partial w= 
\kappa_2 \Delta w - \gamma_2 w + u, 
\quad \text{ a.e. in }\mathbb{R}^d \times (0, \infty). 
\end{cases}
\end{equation*}
\end{itemize}
\end{defn}
In this paper, 
we establish 
the proof of the existence of time-global weak solutions 
by constructing approximate solutions using variational methods and obtaining 
the solution as a limit of these approximations. 
To address the technical issues related to the convergence of 
the approximate solutions, we make the following assumptions: 
\begin{equation}\label{uniform L^2}
\left (1< m < 2 \ \ \text{ and }\ \ \frac{4m+ 2(m-1)d}{d(2-m)} \geq 2\right ) 
\ \ \text{ or } \ \ m \geq 2.
\end{equation}
When $m>2-\frac{4}{d}$, 
 the existence of time-global weak solutions can be established 
 without any restrictions on $\|u_0\|_{L^1}$. 
This is due to the fact that under this condition, 
the Lyapunov functional is always bounded from below.

\begin{thm}[global existence in sub-critical exponent]\label{subcritical}
Assume that $\kappa_1\kappa_2\varepsilon_1\varepsilon_2 \neq 0$ and 
$m>2-\frac{4}{d}$ along with condition \eqref{uniform L^2}. 
Then, for any non-negative functions $u_0$ satisfying 
$u_0 \in (L^1\cap L^m)(\mathbb{R}^d)$ and $|x|^2 u_0 \in L^1(\mathbb{R}^d)$,  
$v_0 \in W^{2,2}(\mathbb{R}^d)$, and $w_0 \in L^2(\mathbb{R}^d)$, 
there exists a time-global weak solution to \eqref{P}. 
\end{thm}

When $m=2-\frac{4}{d}$, 
a mass threshold appears, and the existence of time-global weak solutions 
is conditional upon $\|u_0\|_{L^1} <M_*$. 
Furthermore, since condition \eqref{uniform L^2} is satisfied 
for $d \geq 6$, we can summarize without explicit mention of 
\eqref{uniform L^2} as follows: 
\begin{thm}[global existence in critical exponent] \label{critical}
Assume that $\kappa_1\kappa_2\varepsilon_1\varepsilon_2 \neq 0$ 
and 
$m=2-\frac{4}{d}$,  and $d \geq 6$. 
Then, for any non-negative functions $u_0$ satisfying 
$u_0 \in (L^1\cap L^m)(\mathbb{R}^d)$ and $|x|^2 u_0 \in L^1(\mathbb{R}^d)$, 
$v_0 \in W^{2,2}(\mathbb{R}^d)$, and $w_0 \in L^2(\mathbb{R}^d)$, 
there exists a time-global weak solution to problem \eqref{P}, 
provided that 
\begin{equation*} 
\|u_0\|_{L^1} < M_*, 
\end{equation*}
where $M_*$ is defined by 
\begin{equation} \label{critical mass}
M_*:=
\left ( \frac{2d}{d-4}\frac{\kappa_1 \kappa_2}{C_*^2} \right)^{\frac{d}{4}}
\end{equation}
and 
\begin{equation} \label{dividing constant}
\begin{split} 
C_*^2&: =
\sup_{(u,v) 
\in (L^1 \cap L^m)(\mathbb{R}^d) \times \dot{H}^2(\mathbb{R}^d)} 
\frac{
\displaystyle \int_{\mathbb{R}^d} 
u(-\Delta)^{-2}u\,dx}
{\|u\|_{L^1}^{\frac{4-(2-m)d}{d(m-1)}}\|u\|_{L^m}^{\frac{m(d-4)}{d(m-1)}}}. 
\end{split}
\end{equation}
\end{thm}

In particular, since $M_*$ is the threshold for the lower boundedness of 
the Lyapunov functional, 
it can be expected to serve as the threshold 
for the existence of finite-time blow-up solutions in \eqref{P}.

\begin{thm}[energy inequality]\label{energy inequality}
Assume that $\kappa_1\kappa_2\varepsilon_1\varepsilon_2 \neq 0$. 
Define the functional 
$\mathcal{L}$ as follows:
\begin{equation}\label{Lyapunov}
\begin{split} 
\mathcal{L}(u,v,w)& :=\frac{1}{m-1} 
\int_{\mathbb{R}^d} u^m\,dx - \int_{\mathbb{R}^d} uv\,dx 
  \\
& \hspace{7mm}
+ 
\frac{\kappa_1 \kappa_2}{2}\int_{\mathbb{R}^d}|\Delta v|^2\,dx 
+ \frac{\gamma_1 \kappa_2 + \gamma_2 \kappa_1}{2}
\int_{\mathbb{R}^d}|\nabla v|^2\,dx 
+ \frac{\gamma_1 \gamma_2}{2}
\int_{\mathbb{R}^d}v^2\,dx \\
& \hspace{15mm} + 
\displaystyle 
\frac{\varepsilon_2}{2\varepsilon_1}
\int_{\mathbb{R}^d} |\kappa_1 \Delta v - \gamma_1 v + w|^2\,dx, 
\end{split}
\end{equation}
Then,  the following inequality holds for the solutions obtained by 
Theorem \ref{subcritical} and Theorem \ref{critical}:
\begin{multline*} 
\mathcal{L}(u_0, v_0, w_0) 
- \mathcal{L}(u(T), v(T), w(T)) \\
\geq 
\stint{0}{T}{\mathbb{R}^d} 
\frac{| \nabla u^m -  u \nabla v|^2}{u}\,dxdt 
 +\frac{\varepsilon_1 \kappa_2 + \varepsilon_2 \kappa_1}{\varepsilon_1^2}
\stint{0}{T}{\mathbb{R}^d}
|\nabla 
\left \{ 
\kappa_1 \Delta v(t) -\gamma_1v(t)
+ w(t) \right \} |^2 \,dxdt \\
+ \frac{\gamma_1 \varepsilon_2 + \gamma_2 \varepsilon_1}{\varepsilon_1^2}
\stint{0}{T}{\mathbb{R}^d}
|\kappa_1 \Delta v(t) -\gamma_1v(t)
+ w(t)|^2 \,dxdt
\end{multline*}
for every $T \in [0,\infty)$. 
\end{thm}
It is anticipated that 
the existence of time-global weak solutions can also be proven 
for cases where $d=4,5$ and $m=2-\frac{4}{d}$. 
However, although we have succeeded in constructing approximate solutions 
and demonstrating the boundedness from below of $\mathcal{L}$ 
under the restricted condition for $M$, 
we were unable to verify what appears to be 
the condition $u \in L^2(\mathbb{R}^d \times (0,T))$, 
which seems 
necessary for the convergence of these approximate solutions. 
To establish the versatility of the methods 
used in this paper, 
it is desirable to clarify this issue in future research.

The structure of this paper is as follows: Section 2 introduces the Wasserstein distance and constructs approximate solutions for problem \eqref{P}. 
Section 3 derives the Lyapunov functional $\mathcal{L}$ for \eqref{P}, and Section 4 examines its boundedness from below. Notably, 
if $m>2-\frac{4}{d}$ 
or $m=2-\frac{4}{d}$ and $M < M_*$, $\mathcal{L}$ is bounded below. 
Under this condition, Section 5 derives uniform estimates for the approximate solutions that are independent of the time discretization step. The uniform boundedness obtained in Section 5 leads to the compactness of the approximate solutions, as shown in Section 6. This compactness implies the convergence of the approximate solutions to the solution of problem \eqref{P}, as demonstrated in Section 7. Section 8 proves the energy inequality.

\section{Time-discretized approximate solutions}

In this section, we partition the time interval $[0, \infty)$ and construct variational approximations to the solutions of \eqref{P} at each partition point. 
Before defining the approximate solutions, 
we revisit the concept of the (quadratic) Wasserstein distance, 
denoted as $\mathcal{W}_2$. 
Consider the space $\mathscr{P}(\mathbb{R}^d)$ of 
probability measures on $\mathbb{R}^d$. We define 
\begin{equation*} 
\begin{split} 
\mathscr{P}_2(\mathbb{R}^d)& := 
\left \{\mu \in \mathscr{P}(\mathbb{R}^d)\ : \ 
\int_{\mathbb{R}^d}|x|^2d\mu(x)<+\infty \right \}, \\
M\mathscr{P}_2(\mathbb{R}^d)& := 
\left \{ \mu \ : \ \frac{\mu}{M} \in \mathscr{P}_2(\mathbb{R}^d), M>0 \right \}. 
\end{split}
\end{equation*}

\begin{defn}[Wasserstein distance]
For $\mu, \nu \in \mathscr{P}_2(\mathbb{R}^d)$, 
the \textbf{Wasserstein distance} $\mathcal{W}_2$ is defined as 
\begin{equation*} 
\begin{split} 
\mathcal{W}_2(\mu, \nu):
=\inf_{p \in \Gamma(\mu, \nu)}
\left ( \iint_{\mathbb{R}^d \times \mathbb{R}^d}
|x-y|^2\,dp(x,y) \right )^{\frac{1}{2}}, 
\end{split}
\end{equation*}
where $\Gamma(\mu, \nu)$ represents the set of measures 
$p \in \mathscr{P}(\mathbb{R}^d \times \mathbb{R}^d)$ that satisfy
\begin{equation*}
\iint_{\mathbb{R}^d \times \mathbb{R}^d} b(x)\,dp(x,y)
=\int_{\mathbb{R}^d}b(x)\,d\mu(x), \ 
\iint_{\mathbb{R}^d \times \mathbb{R}^d} b(y)\,dp(x,y)
=\int_{\mathbb{R}^d}b(y)\,d\nu(y), 
\end{equation*}
for any continuous and bounded function 
$b \in C_b(\mathbb{R}^d)$. 
For $\mu, \nu \in M\mathscr{P}_2(\mathbb{R}^d)$, 
$\mathcal{W}_2(\mu, \nu)$ is additionally defined by 
\[
\mathcal{W}_2(\mu, \nu) 
:=\sqrt{M} \mathcal{W}_2\left (\frac{\mu}{M}, \frac{\nu}{M} \right ). 
\]
\end{defn}
It is well-established that a measure 
$\mu \in \mathscr{P}_2(\mathbb{R}^d)$, if absolutely continuous 
with respect to the Lebesgue measure $\mathscr{L}^d$ and 
denoted by $\mu=u\mathscr{L}^d$ with density $u \in L^1(\mathbb{R}^d)$, 
allows for the existence of a unique optimal transport $p_*$ 
and an optimal transport map $\bm{t}_{\mu}^{\nu}$, characterized by 
\begin{equation*} 
\begin{split} 
\mathcal{W}_2(\mu, \nu):
=
\left ( \iint_{\mathbb{R}^d \times \mathbb{R}^d}
|x-y|^2\,dp_*(x,y) \right )^{\frac{1}{2}}
=\left ( 
\int_{\mathbb{R}^d} |x-\bm{t}_{\mu}^{\nu}(x)|^2\,d\mu(x) \right )^{\frac{1}{2}}. 
\end{split}
\end{equation*}
Moreover, $\bm{t}_{\mu}^{\nu}$ coincides almost everywhere 
with the gradient of a convex function 
(refer to \cite[\S 6.2.3]{a-g-s}, 
\cite[Theorem 2.3]{a-s}, \cite[Theorem 2.12]{v2} for instance). 
For more detailed information on Wasserstein distances, consider consulting sources such as Chapter 7 of \cite{a-g-s}, Chapter 6 of \cite{v1}, and Chapter 5 of \cite{s}.
In this paper, we specifically address cases 
where both $\mu$ and $\nu$ possess densities 
$\rho_1$ and $\rho_2$, respectively, 
simplifying the notation for the Wasserstein distance to 
$\mathcal{W}_2(\rho_1, \rho_2)$ instead of 
$\mathcal{W}_2(\rho_1 \mathscr{L}^d, \rho_2 \mathscr{L}^d)$ and 
$\rho_1, \rho_2 \in \mathscr{P}_2(\mathbb{R}^d)$ instead of 
$\mu, \nu \in \mathscr{P}_2(\mathbb{R}^d)$.  
Let $p \in \Gamma_o(\rho_1, \rho_2)$ denote that $p$ is the optimal transport 
between $\rho_1$ and $\rho_2$.

We define three functionals $\mathcal{E}$, $\mathcal{F}$, and $\mathcal{G}$ 
as follows:
\begin{equation*}
\begin{split}
\mathcal{E}(u,v) &:= \frac{1}{m-1}\int_{\mathbb{R}^d} u^m\, dx 
- \int_{\mathbb{R}^d} uv\, dx, \\
\mathcal{F}(v,w) &:= \frac{1}{2} \int_{\mathbb{R}^d} 
(\kappa_1 |\nabla v|^2 + \gamma_1 v^2)\,dx - \int_{\mathbb{R}^d} vw\,dx, \\
\mathcal{G}(w, u) &:= \frac{1}{2} \int_{\mathbb{R}^d} 
(\kappa_2 |\nabla w|^2 + \gamma_2 w^2)\,dx - \int_{\mathbb{R}^d} uw\,dx.
\end{split}
\end{equation*}
Then, the system \eqref{P} can formally be expressed as
\begin{equation*}
\begin{cases}
\ \ \, \partial_t u = -\nabla_u \mathcal{E}(u,v), \\
\varepsilon_1 \partial_t v = -\nabla_v \mathcal{F}(v,w), \\
\varepsilon_2 \partial_t w = -\nabla_w \mathcal{G}(w,u),
\end{cases}
\end{equation*}
where $\nabla_u$ represents the gradient of the functional $u \mapsto \mathcal{E}(u,v)$ in the Wasserstein space. 
Namely, when $\mathcal{E}(u):=\int_{\mathbb{R}^d}f(u),dx$, we have
$-\nabla_u \mathcal{E}(u)
= \nabla \cdot (u\nabla f'(u))$. 
Similarly, $\nabla_v$ and $\nabla_w$ represent the gradients of the functionals $v \mapsto \mathcal{F}(v,w)$ and $w \mapsto \mathcal{G}(w, u)$ in $L^2$-space, respectively. 
Let us refer to this system of equations as the 
\textbf{Chimera Gradient Flow}.
As will be discussed later, it is possible to construct a \lq \lq formal\rq \rq 
approximate solution for the Chimera Gradient Flow using techniques similar to the minimizing movement scheme. While convergence of the approximate solution is not generally guaranteed by theory, in the current problem \eqref{P}, the existence of a Lyapunov functional can be asserted, which allows for a discussion on the convergence of the approximate solutions.

Now, we divide the time interval $[0,\infty)$ into sub-intervals with length $\tau>0$.  
Let $(u_{\tau}^0, v_{\tau}^0, w_{\tau}^0) := (\rho_{\tau}*u_0, v_0, w_0)$ 
be an approximated initial data, 
where 
${\rho_{\tau}}$ is a family of mollifiers satisfying
\begin{equation*}
\rho_{\tau} = \frac{1}{\sqrt{\tau}} \rho\left(\frac{x}{\sqrt{\tau}^{1/d}}\right), \quad \int_{\mathbb{R}^d} \rho(x) , dx = 1, \quad \rho \in C_c^{\infty}(\mathbb{R}^d), \quad \rho \geq 0.
\end{equation*}
Note that $\mathcal{W}_2^2(u_0, u_{\tau}^0) \to 0$ as $\tau \to 0$ by 
\cite[Lemma 7.1.10]{a-g-s} and that 
$\mathcal{L}(u_{\tau}^0,v_{\tau}^0,w_{\tau}^0) 
\to 
\mathcal{L}(u_0, v_0, w_0)$ as $\tau \to 0$, where 
$\mathcal{L}$ is the functional defined in Theorem \ref{energy inequality}. 
Therefore, there exists a $\tau_*>0$ such that for $\tau \in (0,\tau_*)$, 
\begin{equation} \label{initial energy}
\mathcal{L}(u_{\tau}^0,v_{\tau}^0,w_{\tau}^0) 
\leq \mathcal{L}(u_0, v_0, w_0)+1
\end{equation}
Hereafter, assume $\tau \in (0,\tau_*)$ 
and it will be understood without stating each time that \eqref{initial energy} holds. 
For $k = 1, 2, 3, \cdots$, we recursively define 
$(u_{\tau}^k, v_{\tau}^k, w_{\tau}^k)$ by
\begin{equation} \label{mm}
\begin{cases}
\begin{split} 
& w_{\tau}^k \in 
\argmin_{w \in H^1(\mathbb{R}^d)}
\left \{ 
\mathcal{G}(w, u_{\tau}^{k-1})
+ \frac{\varepsilon_2}{2\tau}\|w - w_{\tau}^{k-1}\|_{L^2}^2
\right \},  \\
& v_{\tau}^k \in 
\argmin_{v \in H^1(\mathbb{R}^d)}
\left \{ 
\mathcal{F}(v, w_{\tau}^k)
+ \frac{\varepsilon_1}{2\tau} \|v - v_{\tau}^{k-1}\|_{L^2}^2 
\right \},  \\
& u_{\tau}^k \in 
\argmin_{u \in M\mathscr{P}_2(\mathbb{R}^d)\cap L^m(\mathbb{R}^d)}
\left \{ 
\mathcal{E}(u,v_{\tau}^k)+ 
\frac{1}{2\tau}\mathcal{W}_2^2(u,u_{\tau}^{k-1}) 
\right \}, 
\end{split}
\end{cases}
\end{equation}
where 
$\argmin_{X}\{\cdots \}$ denotes the set of minimizers of the functional $\{\cdots\}$ over $X$. Note that $k$ represents a parameter, not an exponent.

\begin{prop}\label{propEL}
For any $k \in \mathbb{N}$, 
the triple $(u_{\tau}^k, v_{\tau}^k, w_{\tau}^k)$ is defined, where $u_{\tau}^k \in L^2(\mathbb{R}^d)$ with $(u_{\tau}^k)^m \in W^{1,1}(\mathbb{R}^d)$, $v_{\tau}^k \in W^{4,2}(\mathbb{R}^d)$, $w_{\tau}^k \in W^{2,2}(\mathbb{R}^d)$, 
and $u_{\tau}^k, v_{\tau}^k, w_{\tau}^k \geq 0$. 
Furthermore, they satisfy the following 
Euler-Lagrange equations: 
\begin{equation}\label{EL}
\begin{cases}\vspace{2mm}
\displaystyle 
 \ \ \int_{\mathbb{R}^d \times \mathbb{R}^d}
\langle y-x, \bm{\xi}(y) \rangle \,Mdp_k(x,y)+  
 \tau \int_{\mathbb{R}^d} \langle \nabla (u_{\tau}^k)^m
-  u_{\tau}^k \nabla v_{\tau}^{k}, \bm{\xi} \rangle\, dx  = 
0, \\ \vspace{2mm}
\displaystyle 
\varepsilon_1 \int_{\mathbb{R}^d}(v_{\tau}^k - v_{\tau}^{k-1})\zeta \,dx 
= \tau \int_{\mathbb{R}^d}( \kappa_1\Delta  v_{\tau}^k
- \gamma_1 v_{\tau}^k +  w_{\tau}^{k})\zeta\, dx, \\
\vspace{2mm}
\displaystyle 
\varepsilon_2 \int_{\mathbb{R}^d}(w_{\tau}^k - w_{\tau}^{k-1})\eta \,dx 
=  \tau \int_{\mathbb{R}^d} 
(\kappa_2 \Delta  w_{\tau}^k 
- \gamma_2 w_{\tau}^k + u_{\tau}^{k-1})\eta\, dx, 
\end{cases}
\end{equation}
for all 
$\bm{\xi} \in C^{\infty}_c(\mathbb{R}^d;\mathbb{R}^d)$, 
for all $\zeta, \eta \in C^{\infty}_c(\mathbb{R}^d)$, 
where 
$p_k \in \Gamma_o(u_{\tau}^{k-1}, u_{\tau}^{k})$ is the optimal transport plan between $u_{\tau}^{k-1}$ and $u_{\tau}^k$, 
which satisfies 
\[
\mathcal{W}_2(u_{\tau}^k/M, u_{\tau}^{k-1}/M)
= \left ( \int_{\mathbb{R}^d \times \mathbb{R}^d} |x-y|^2\,dp_k(x,y) \right )^{\frac{1}{2}}. 
\]
\end{prop}

The remainder of this section provides a proof of Proposition \ref{EL} by dividing it into several lemmas.

\begin{lem}
Assuming $u_{\tau}^{k-1}, v_{\tau}^{k-1}, w_{\tau}^{k-1} \in L^2(\mathbb{R}^d)$, 
the pair $(v_{\tau}^k, w_{\tau}^k)$ is 
uniquely well-determined as defined in \eqref{mm}, 
with both components being non-negative. 
Additionally,  
the second and third equations of \eqref{EL} 
are satisfied, and $v_{\tau}^k$ and $w_{\tau}^k$ respectively qualify as 
elements of $W^{4,2}(\mathbb{R}^d)$ and $W^{2,2}(\mathbb{R}^d)$. 
\end{lem}

\begin{proof}

The proof for $v_{\tau}^k$ can be similarly established; 
therefore, for $w_{\tau}^k$, 
the proof is provided divided into four parts: 
existence, uniqueness, non-negativity, and the Euler-Lagrange equations.

\vspace{2mm}

\noindent
\textbf{(i) existence of $w_{\tau}^k$ }
\vspace{2mm}

The negative term in $\mathcal{G}$ is estimated 
using the Cauchy-Schwarz and Young's inequalities, 
yielding
\begin{equation*} 
\begin{split} 
\int_{\mathbb{R}^d}wu_{\tau}^{k-1}\,dx 
& = \int_{\mathbb{R}^d}(w-w_{\tau}^{k-1})u_{\tau}^{k-1}\,dx+   
\int_{\mathbb{R}^d}w_{\tau}^{k-1}u_{\tau}^{k-1}\,dx   \\
& \leq \|w -w_{\tau}^{k-1}\|_{L^2}\|u_{\tau}^{k-1}\|_{L^2}  
+ \int_{\mathbb{R}^d}w_{\tau}^{k-1}u_{\tau}^{k-1}\,dx \\
& \leq \frac{\varepsilon_2}{4\tau}\|w-w_{\tau}^{k-1}\|_{L^2}^2 
+ \frac{\tau}{\varepsilon_2}\|u_{\tau}^{k-1}\|_{L^2}^2 
+ \int_{\mathbb{R}^d}w_{\tau}^{k-1}u_{\tau}^{k-1}\,dx. 
\end{split}
\end{equation*}
From Young's inequality: 
\begin{equation*} 
\begin{split} 
\|w-w_{\tau}^{k-1}\|_{L^2}^2 
\geq \frac{1}{2}\|w\|_{L^2}^2 -\|w_{\tau}^{k-1}\|_{L^2}^2,  
\end{split}
\end{equation*}
we derive 
\begin{equation} \label{LEG}
\begin{split} 
\mathcal{G}(w, u_{\tau}^{k-1})+\frac{\varepsilon_2 \|w-w_{\tau}^{k-1}\|_{L^2}^2}{2\tau}
& \geq 
\frac{1}{2}\int_{\mathbb{R}^d}(\kappa_2|\nabla w|^2 + \gamma_2 w^2)\,dx
+ 
\frac{\varepsilon_2 \|w\|_{L^2}^2}{8\tau} \\
& \hspace{5mm}
 -\frac{\varepsilon_2}{4\tau}\|w_{\tau}^{k-1}\|_{L^2}^2
- \frac{\tau}{\varepsilon_2}\|u_{\tau}^{k-1}\|_{L^2}^2 
- \int_{\mathbb{R}^d}w_{\tau}^{k-1}u_{\tau}^{k-1}\,dx. 
\end{split}
\end{equation}
Hence, the minimizing sequence 
$\{w_n\}$ for the left-hand side is bounded in 
$H^1(\mathbb{R}^d)$, and we may assume 
it weakly converges to some $w_* \in H^1(\mathbb{R}^d)$. 
At this point, we have 
\begin{equation*} 
\begin{split} 
\mathcal{G}(w_*, u_{\tau}^{k-1}) \leq \liminf_{n \to \infty}\mathcal{G}(w_n, u_{\tau}^{k-1}), 
\quad \|w_*-w_{\tau}^{k-1}\|_{L^2}^2 
\leq \liminf_{n \to \infty}\|w_n-w_{\tau}^{k-1}\|_{L^2}^2  
\end{split}
\end{equation*}
which confirms that $w_*$ is a minimizer of the left-hand side in \eqref{LEG}. 
Therefore, we can define $w_{\tau}^k:=w_*$. 

\vspace{2mm}

\noindent 
\textbf{(ii) uniqueness}

\vspace{2mm}

The uniqueness of discrete solution follows from 
the strict convexity of 
$w \mapsto \mathcal{G}(w, u_{\tau}^{k-1})$ and  
$w\mapsto \|w-w_{\tau}^{k-1}\|_{L^2}^2$.

\vspace{2mm}

\noindent
\textbf{(iii) non-negativity}

\vspace{2mm}
 
The non-negativity of $u_{\tau}^{k-1}$ is clear because 
$u_{\tau}^{k-1}$ belongs to $M\mathscr{P}_2(\Omega)$ for $k \geq 2$ and 
$u_0 \geq 0$. 
Consequently, it holds that  
\begin{equation}\label{morethanabs}
\mathcal{G}(w_{\tau}^k, u_{\tau}^{k-1}) \geq
\mathcal{G}(|w_{\tau}^k|, u_{\tau}^{k-1}).
\end{equation}
We prove $v_{\tau}^k \geq 0$ by 
induction in $k$. 
For $k=0$, we have $w_{\tau}^0=w_0 \geq 0$ by 
the assumption. 
By the triangle inequality, 
$ \bigl | |w_{\tau}^k|-|w_{\tau}^{k-1}|
\bigr | \leq  | w_{\tau}^k - w_{\tau}^{k-1} |$. 
Therefore, if $w_{\tau}^{k-1} \geq 0$, then we have
\begin{equation}\label{triangle}
\bigl \| |w_{\tau}^k|-w_{\tau}^{k-1}\bigr \|_{L^2} 
\leq \|w_{\tau}^k - w_{\tau}^{k-1}\|_{L^2}. 
\end{equation}
Combining \eqref{morethanabs} and \eqref{triangle}, we obtain 
\[\mathcal{G}(w_{\tau}^k, u_{\tau}^{k-1}) + \frac{\varepsilon_2}{2 \tau}
\|v_{\tau}^k - v_{\tau}^{k-1}\|_{L^2}^2 \geq 
\mathcal{G}(|w_{\tau}^k|, u_{\tau}^{k-1})+ \frac{\varepsilon_2}{2 \tau}  
\bigl \| |v_{\tau}^k|-v_{\tau}^{k-1}\bigr \|_{L^2}^2.  \]
Since $w_{\tau}^k$ is a unique minimizer of the functional 
\[ w \mapsto \mathcal{G}(w, u_{\tau}^{k-1}) + 
\frac{\varepsilon_2}{2 \tau}
\|w - w_{\tau}^{k-1}\|_{L^2}^2, \]
we have $|w_{\tau}^k|=w_{\tau}^k$. 

\vspace{2mm}

\noindent
\textbf{(iv) Euler-Lagrange equation} 

\vspace{2mm}

For any $\eta \in C^{\infty}_c(\mathbb{R}^d)$, 
\begin{equation*} 
\begin{split} 
\frac{d}{d\varepsilon}
\left (\mathcal{G}(w_{\tau}^k + \varepsilon \eta)
+\frac{\varepsilon_2 
\|w_{\tau}^k + \varepsilon \eta-w_{\tau}^{k-1}\|_{L^2}^2}{2\tau}\right )
\Biggm |_{\varepsilon=0}=0, 
\end{split}
\end{equation*}
thus confirming that the third equation of \eqref{EL} holds. 
From elliptic regularity, we have
$w_{\tau}^k \in W^{2,2}(\mathbb{R}^d)$. 
\end{proof}

\begin{lem}
Assume that 
$u_{\tau}^{k-1} \in M\mathscr{P}_2(\mathbb{R}^d) \cap L^m(\mathbb{R}^d)$, 
and $v_{\tau}^k \in W^{4,2}(\mathbb{R}^d)$. 
Then, $u_{\tau}^k$,
 defined in \eqref{mm}, is uniquely well-determined 
and non-negative. 
\end{lem}
\begin{proof}
The proof is provided in two parts: existence and uniqueness.

\vspace{2mm}

\noindent 
\textbf{(i) existence of $\bm{u_{\tau}^k}$}

\vspace{2mm}

Since $\frac{2(d-2)}{d} \leq m \leq 2$, the conjugate exponent 
$m'$ satisfies 
\[
2 \leq m' \leq \frac{2(d-2)}{d-4}. 
\]
Furthermore, given 
$v_{\tau}^k \in W^{4,2}(\mathbb{R}^d)$ embeds into 
$L^{m'}(\mathbb{R}^d)$, 
the H\"older and Young's inequalities yield
\begin{equation*} 
\begin{split} 
\int_{\mathbb{R}^d} uv_{\tau}^k\, dx 
\leq \|u\|_{L^m}\|v_{\tau}^k\|_{L^{m'}} 
\leq \frac{1}{m}\|u\|_{L^m}^m + \frac{1}{m'}\|v_{\tau}^k\|_{L^{m'}}^{m'}. 
\end{split}
\end{equation*}
From this, it follows that 
\begin{equation*} 
\begin{split} 
\mathcal{E}(u,v_{\tau}^k)+ 
\frac{1}{2\tau}\mathcal{W}_2^2(u,u_{\tau}^{k-1}) 
\geq \frac{1}{m(m-1)}\|u\|_{L^m}^m 
-\frac{1}{m'}\|v_{\tau}^k\|_{L^{m'}}^{m'}
+ \frac{1}{2\tau}\mathcal{W}_2^2(u,u_{\tau}^{k-1}), 
\end{split}
\end{equation*}
and the minimizing sequence 
$\{u_n\} \subset M\mathscr{P}_2(\mathbb{R}^d) \cap L^m(\mathbb{R}^d)$ 
in the left-hand side 
forms a bounded sequence in $L^m(\mathbb{R}^d)$. 
Moreover, for any $K>0$, 
the Cauchy-Schwarz inequality implies 
\begin{equation*} 
\begin{split} 
\int_{\{ u_n >K\}}u_n\,dx 
& = \int_{\{ u_n >K\}}(u_n)^{\frac{m}{2}} (u_n)^{1-\frac{m}{2}}\,dx \\
& \leq 
\left (\int_{\{ u_n >K\}}u_n^m\,dx \right )^{\frac{1}{2}}
\left (\int_{\{ u_n >K\}}u_n^{1-(m-1)}\,dx \right )^{\frac{1}{2}} \\
& \leq 
\frac{1}{K^{m-1}} 
\left (\int_{\mathbb{R}^d}u_n^m\,dx \right )^{\frac{1}{2}}
\left (\int_{\mathbb{R}^d}u_n\,dx \right )^{\frac{1}{2}}
= \frac{\sqrt{M}}{K^{m-1}}\|u_n\|_{L^m}^{\frac{m}{2}}, 
\end{split}
\end{equation*}
hence, the sequence $\{u_n\}$ is equi-integrable. 
The boundedness of $\|u_n\|_{L^m}$ and the Dunford-Pettis theorem 
imply that a subsequence of $\{u_n\}$, still denoted by $\{u_n\}$, 
weakly converges to some $u_* \in (L^1 \cap L^m)(\mathbb{R}^d)$. 
Particularly, 
from $\|u_n \|_{L^1}= M$, 
it follows $\|u_*\|_{L^1}=M$. At this point, 
\begin{equation*} 
\mathcal{E}(u_*, v_{\tau}^k) \leq 
\liminf_{n\to \infty}\mathcal{E}(u_n, v_{\tau}^k)
\end{equation*}
and from \cite[Lemma 7.1.4]{a-g-s} 
\begin{equation*} 
\mathcal{W}(u_*, u_{\tau}^{k-1}) 
\leq \liminf_{n \to \infty}
\mathcal{W}(u_n, u_{\tau}^{k-1}). 
\end{equation*} 
By the triangle inequality, we have 
\begin{equation*} 
\begin{split} 
\left ( \int_{\mathbb{R}^d}|x|^2u_*(x)\,dx \right )^{\frac{1}{2}}
= \mathcal{W}(u_*, \delta_0)
& \leq \mathcal{W}(u_*, u_{\tau}^{k-1}) + \mathcal{W}(u_{\tau}^{k-1}, \delta_0) \\
& = \liminf_{n \to \infty}\mathcal{W}(u_n, u_{\tau}^{k-1})
+
\left ( \int_{\mathbb{R}^d}|x|^2u_{\tau}^{k-1}(x)\,dx \right )^{\frac{1}{2}}, 
\end{split}
\end{equation*}
hence, $u_* \in M\mathscr{P}_2(\mathbb{R}^d)$. 
Thus, 
\[
\mathcal{E}(u_*, v_{\tau}^k)+ \frac{1}{2\tau}\mathcal{W}(u_*, u_{\tau}^{k-1}) 
= \lim_{n\to \infty}
\left ( \mathcal{E}(u_n, v_{\tau}^k) 
+ \frac{1}{2\tau}\mathcal{W}(u_n, u_{\tau}^{k-1}) \right ). 
\]
Therefore, $u_{\tau}^k:=u_*$ can be defined.

\vspace{2mm}

\noindent 
\textbf{(ii) uniqueness}  

\vspace{2mm}

It is clear that $u \mapsto \mathcal{E}(u, v_{\tau}^k)$ is strictly convex, 
so we demonstrate that $u \mapsto \mathcal{W}_2^2(u, u_{\tau}^{k-1})$ 
is weakly convex following \cite[{ Proposition A.1}]{o1}. 
Let $p_1 \in \Gamma_o(u_1, u_3)$ and $\ p_2 \in \Gamma_o(u_2, u_3)$. 
In this case, it is straightforward to verify that 
$
p_s:=(1-s)p_1+sp_2 \in \Gamma((1-s)u_1+ su_2, u_3)
$. 
Consequently, we have 
\begin{equation*} 
\begin{split} 
\mathcal{W}_2^2((1-s)u_1 + su_2, u_3) 
& \leq \int_{\mathbb{R}^d \times \mathbb{R}^d}|x-y|^2\,dp_s
\\
& = (1-s)\int_{\mathbb{R}^d \times \mathbb{R}^d}|x-y|^2\,dp_1
+ s\int_{\mathbb{R}^d \times \mathbb{R}^d}|x-y|^2\,dp_2 \\
& = 
(1-s)\mathcal{W}_2^2(u_1, u_3) + s\mathcal{W}_2^2(u_2, u_3)
\end{split}
\end{equation*}
The convexity demonstrated here implies the uniqueness of $u_{\tau}^k$. 
\end{proof}

To derive the Euler-Lagrange equations for $u_{\tau}^k$, 
 it is important to recall the concept of the push-forward and its properties. 
\begin{defn}[push-forward \text{\cite[\S 5.2]{a-g-s}}]
Let $\mu, \nu \in \mathscr{P}(\mathbb{R}^d)$. 
Consider a $\mu$-measurable map $\bm{t}:\mathbb{R}^d \to \mathbb{R}^d$. 
If for every $f \in C_b(\mathbb{R}^d)$, 
the equality   
\[ \int_{\mathbb{R}^d}f(y)\, d\nu(y) =
\int_{\mathbb{R}^d}f(\bm{t}(x))\,d\mu(x)\] 
holds, then $\nu$ is termed the \textbf{push-forward} of $\mu$ through $\bm{t}$, 
denoted by $\nu= \bm{t}\push \mu$. Specifically, if 
$\mu=\rho_1\mathscr{L}^d$ and 
$\nu=\rho_2\mathscr{L}^d$ with $\rho_1, \rho_2 \in L^p(\mathbb{R}^d)$ and 
$1 \leq p <\infty$, the same notation remains consistent provided that 
\[
\int_{\mathbb{R}^d}f(y)\rho_2(y)\,dy= 
\int_{\mathbb{R}^d} f(\bm{t}(x))\rho_1(x)\,dx 
 \]
for every $f \in L^{p'}(\mathbb{R}^d)$. 
\end{defn}

If $\nu=\bm{t}\push \mu$ with 
$d\mu=u\,d\mathscr{L}^d$ and $d\nu=v\,d\mathscr{L}^d$, 
then one can deduce from change of variables that  
\begin{equation}\label{push forward for densities}
 v(\bm{t}(x))|\det{(D\bm{t}(x))}|=u(x)
 \end{equation}
holds for almost every $x$ in $\mathbb{R}^d$.

\begin{lem} \label{push-forward integrability}
Let $1 \leq p < \infty$ and  
$\bm{\xi} \in C^{\infty}_c(\mathbb{R}^d; \mathbb{R}^d)$. 
If $v_{n}$ converges to $v$ in $L^p(\mathbb{R}^d)$, 
then for $\delta>0$ small enough and for every $t \in [0,\delta]$,  
$v_{n}(\bm{id}+t\bm{\xi})$ converges to $v(\bm{id}+t\bm{\xi})$ 
in $L^p(\mathbb{R}^d)$. 
In addition, it holds that 
\[ \| v_{n}(\bm{id}+t\bm{\xi}) - v(\bm{id}+t\bm{\xi}) \|_{L^p(\mathbb{R}^d)} \leq 
C_{\delta} \| v_{n} - v \|_{L^p(\mathbb{R}^d)} \text{ for all } t \in [0,\delta], \]
where $C_{\delta}$ is a positive constant depending on $\delta$ 
and $\bm{\xi}$.
\end{lem}

\begin{proof}Let $\bm{r}_t(x):=x+t\bm{\xi}(x)$. 
Note that for $\delta$ small enough, 
$\bm{r}_t$ is a $C^1$ diffeomorphism and 
$\det{D\bm{r}_t}>0$ in $t \in [0,\delta]$ 
since 
$\bm{\xi} \in C^{\infty}_c(\mathbb{R}^d; \mathbb{R}^d)$.  
By the change of variables $y = \bm{r}_t(x)$, we have 
\begin{equation*}
\begin{split}
 \int_{\mathbb{R}^d} | v_n(\bm{r}_t(x))-v(\bm{r}_t(x)) |^p\,dx 
&= \int_{\mathbb{R}^d} |v_n(y)-v(y)|^p \det (D\bm{r}_t^{-1}(y))\, dy 
 \\
& \leq 
\sup_{(y,t) \in \mathbb{R}^d \times [0,\delta]}
\Bigl(\det{( D\bm{r}_t^{-1}(y))}\Bigr)\|v_n-v\|_{L^p(\mathbb{R}^d)}^p. 
\end{split}
\end{equation*}
\end{proof}

With these preparations in place, the Euler-Lagrange equations for $u_{\tau}^k$ 
are derived.

\begin{lem}
Let the pair $(u_{\tau}^k, v_{\tau}^k)$ be 
a solution of \eqref{mm}. Then,  
the first equation in \eqref{EL} holds.  
\end{lem}

\begin{proof}
The proof is provided in four steps. 

\vspace{2mm}

\noindent 
\textbf{Step 1 : G\^ateaux derivative of the internal energy}

\vspace{2mm}

Define $U_{s}:=(\bm{id} + s \bm{\xi})\push u_{\tau}^k$. 
Then, from the relation \eqref{push forward for densities}, we have 
\[U_{s}(x + s \bm{\xi}(x))
|\det{(\bm{id} + s D\bm{\xi}(x)})|=u_{\tau}^k(x).  \]
Consequently, we have 
\begin{equation*} 
\begin{split} 
\frac{1}{s}
\int_{\mathbb{R}^d}
U_{s}^m - (u_{\tau}^k)^m\,dx
&= \frac{1}{s}
\left( \int_{\mathbb{R}^d} U_{s}^{m-1}(\bm{id} 
+ s \bm{\xi}(x))u_{\tau}^k(x)\,dx 
- \int_{\mathbb{R}^d} (u_{\tau}^k)^m\,dx 
\right) \\
& = 
\int_{\mathbb{R}^d} 
\frac{|\det{(\bm{id} + s D\bm{\xi}(x)})|^{1-m}-1}
{s}(u_{\tau}^k)^m\,dx.  
\end{split}
\end{equation*}
Since $(|\det{(\bm{id} + sD\bm{\xi}(x)})|^{1-m}-1)/s$ 
uniformly converges to $(1-m)\diver{\bm{\xi}}$ 
because of the uniform boundedness of $D \bm{\xi}$, we obtain  
\begin{equation*}
\lim_{s \to 0}\int_{\mathbb{R}^d}
\frac{U_{s}^m - (u_{\tau}^k)^m}{s}\,dx
= (1-m) \int_{\mathbb{R}^d} (u_{\tau}^k)^m\diver{\bm{\xi}} \,dx.  
\end{equation*}

\noindent 
\textbf{Step 2 : G\^ateaux derivative of the potential energy}

\vspace{2mm}

Let $\{v_n\} \subset C^{\infty}_c(\mathbb{R}^d)$ 
be a sequence converging to $v$ in $W^{2,2}(\mathbb{R}^d)$. 
Define 
\[ 
I_n(s):=\int_{\mathbb{R}^d} 
U_s v_n\,dx=\int_{\mathbb{R}^d} v_n(x + s \bm{\xi}(x))u(x)\,dx
\text{\ \ \  and \ \ }I(s):=\int_{\mathbb{R}^d}U_sv\,dx.
\]
Then, $I_n(s)$ converges uniformly to $I(s)$ in a small interval $[0,\delta]$, 
and 
$I_n(s)$ is differentiable with 
\[
I_n'(s) = \int_{\mathbb{R}^d} 
\langle \nabla v_n(x + s \bm{\xi}(x)), \bm{\xi}(x)\rangle u(x)\,dx.  
 \]
On the other hand, 
\begin{multline*}
\left | \int_{\mathbb{R}^d}
  \langle \nabla v(x + s \bm{\xi}(x)) - \nabla v_n(x + s \bm{\xi}(x)), 
\bm{\xi}(x)\rangle u(x)\,dx\right | \\
 \leq \|\bm{\xi}\|_{L^{\infty}} \|u\|_{L^m} 
\left (\int_{\supp{\bm{\xi}}} |\nabla v(x + s \bm{\xi}(x)) 
- \nabla v_n(x + s \bm{\xi}(x))|^{m'}\,dx \right)^{1/m'}. 
\end{multline*}
By Lemma \ref{push-forward integrability}, the right-hand side converges uniformly 
to $0$ in $[0,\delta]$. 
Hence, $I(s)$ is differentiable at $s=0$, and we have
\[ 
I'(0) 
=\int_{\mathbb{R}^d} \langle \nabla v(x), \bm{\xi}(x) \rangle u(x) \, dx. 
\] 

\noindent 
\textbf{Step 3 : the right G\^ateaux derivative of the Wasserstein distance}

\vspace{2mm}

Let $p_k \in \Gamma_o(u_{\tau}^{k-1}/M, u_{\tau}^k/M)$. 
Define $p_s \in \Gamma(U_s/M, u_{\tau}^k/M)$ by 
\[ 
\int_{\mathbb{R}^d \times \mathbb{R}^d}
f(x,y)\,dp_s(x,y) = 
\int_{\mathbb{R}^d \times \mathbb{R}^d} 
f(x,y+ s\bm{\xi}(y))\,dp_k(x,y) 
\]
for any $f \in C_b(\mathbb{R}^d \times \mathbb{R}^d)$. 
Then by definition of the Wasserstein distance, 
\[\mathcal{W}_2^2(U_s, u_{\tau}^k) 
\leq 
\int_{\mathbb{R}^d \times \mathbb{R}^d} |x-y|^2\,Mdp_s(x,y)
=
\int_{\mathbb{R}^d \times \mathbb{R}^d} |x-y-s\bm{\xi}(y)|^2\,Mdp_k(x,y). \]
Therefore we have 
\begin{equation*}
\begin{split}
\mathcal{W}_2^2(U_s, u_{\tau}^{k-1})- \mathcal{W}_2^2(u_{\tau}^k,u_{\tau}^{k-1}) 
& \leq 
\int_{\mathbb{R}^d \times \mathbb{R}^d} 
 (|x-y-s\bm{\xi}(y)|^2
- 
 |x-y|^2)\,Mdp_k(x,y) \\
 & \leq 
  -2s\int_{\mathbb{R}^d \times \mathbb{R}^d} 
 \langle x-y, \bm{\xi}(y)\rangle \,Mdp_k(x,y) 
 + s^2M\|\bm{\xi}\|_{L^{\infty}}. 
\end{split}
\end{equation*}
Dividing by $s>0$ and passing to the limit as $s \downarrow 0$, 
we obtain 
\begin{equation*}
\limsup_{s \downarrow 0} 
\frac{\mathcal{W}_2^2(U_s, u_{\tau}^{k-1}) 
- \mathcal{W}_2^2(u_{\tau}^k, u_{\tau}^{k-1})}{s}
\leq 2\int_{\mathbb{R}^d \times \mathbb{R}^d} 
\langle y-x, \bm{\xi}(y) \rangle \,Mdp_k(x,y). 
\end{equation*}

\noindent 
\textbf{Step 4 : Euler-Lagrange equations}

\vspace{2mm}

Given the minimality of $u_{\tau}^k$, we have 
\begin{equation*}  
\mathcal{E}(U_s, v_{\tau}^k)+\frac{1}{2\tau}\mathcal{W}_2^2(U_s, u_{\tau}^{k-1}) 
- \mathcal{E}(u_{\tau}^k, v_{\tau}^k)
+\frac{1}{2\tau}\mathcal{W}_2^2(u_{\tau}^k, u_{\tau}^{k-1})  \geq 0
\end{equation*}
Dividing both sides by $s>0$ and taking the limit as $s \downarrow 0$, 
we obtain 
\begin{equation*} 
\begin{split} 
\frac{1}{\tau}\int_{\mathbb{R}^d \times \mathbb{R}^d} 
\langle y-x, \bm{\xi}(y) \rangle \,Mdp_k(x,y)
-\int_{\mathbb{R}^d} (u_{\tau}^k)^m \diver \bm{\xi}\,dx 
-\int_{\mathbb{R}^d}\langle \nabla v_{\tau}^k ,\bm{\xi}\rangle  u_{\tau}^k\,dx\geq 0
\end{split}
\end{equation*}
By considering $-\bm{\xi}$ instead of $\bm{\xi}$, we obtain 
\begin{equation*} 
\begin{split} 
\frac{1}{\tau}\int_{\mathbb{R}^d \times \mathbb{R}^d} 
\langle y-x, \bm{\xi}(y) \rangle \,Mdp_k(x,y)
-\int_{\mathbb{R}^d} (u_{\tau}^k)^m \diver \bm{\xi}\,dx 
-\int_{\mathbb{R}^d}\langle \nabla v_{\tau}^k ,\bm{\xi}\rangle  u_{\tau}^k\,dx=0. 
\end{split}
\end{equation*}

Below, we demonstrate that $\nabla (u_{\tau}^k)^m \in L^1(\mathbb{R}^d)$. 
By the Cauchy-Schwarz inequality,  we have 
\begin{equation*} 
\begin{split} 
\left | \int_{\mathbb{R}^d} (u_{\tau}^k)^m \diver \bm{\xi}\,dx 
+ \int_{\mathbb{R}^d}\langle \nabla v_{\tau}^k ,\bm{\xi}\rangle  u_{\tau}^k\,dx \right |
& \leq \left |\frac{1}{\tau} 
\int_{\mathbb{R}^d \times \mathbb{R}^d} 
\langle y-x, \bm{\xi}(y) \rangle \,Mdp_k(x,y) \right | \\
& \leq 
\frac{1}{\tau}
\left (\int_{\mathbb{R}^d \times \mathbb{R}^d}
|y-x|^2\,Mdp_k \right )^{\frac{1}{2}}
\left (\int_{\mathbb{R}^d}
|\bm{\xi}(y)|^2\,Mdp_k
 \right )^{\frac{1}{2}} \\
 & \leq \frac{\mathcal{W}_2(u_{\tau}^{k}, u_{\tau}^{k-1})}{\tau}
 \left (\int_{\mathbb{R}^d}
|\bm{\xi}(y)|^2u_{\tau}^{k}\,dx
 \right )^{\frac{1}{2}}  \\
 & \leq \frac{\mathcal{W}_2(u_{\tau}^{k}, u_{\tau}^{k-1})}{\tau}
 \sqrt{M}\|\bm{\xi}\|_{L^{\infty}}, 
\end{split}
\end{equation*}
so, we have 
\begin{equation*} 
\begin{split} 
\left | \int_{\mathbb{R}^d} (u_{\tau}^{k})^m \diver \bm{\xi}\,dx \right |
& \leq 
\left | \int_{\mathbb{R}^d}
\langle \nabla v_{\tau}^k ,\bm{\xi}\rangle  u^k\,dx \right | 
+ \frac{\mathcal{W}_2(u_{\tau}^{k}, u_{\tau}^{k-1})}{\tau}\sqrt{M}
\|\bm{\xi}\|_{L^{\infty}} \\
& \leq \|\bm{\xi}\|_{L^{\infty}} 
\left ( 
\int_{\mathbb{R}^d}u^k|\nabla v_{\tau}^k|\,dx 
+ \sqrt{M} \frac{\mathcal{W}_2(u_{\tau}^{k}, u_{\tau}^{k-1})}{\tau}
\right ). 
\end{split}
\end{equation*}
By the Hahn-Banach theorem and the Riesz-Markov representation theorem, 
there exists an $\mathbb{R}^d$-valued measure 
$\bm{\mu}^k=(\mu^k_1, \mu^k_2,\ldots , \mu^k_d)$ such 
that 
\begin{equation*}
- \int_{\mathbb{R}^d}(u_{\tau}^k)^m \diver{\bm{\xi}} \,dx =
\sum_{j=1}^d
\int_{\mathbb{R}^d}
\xi_j\, d\mu^k_j, 
\end{equation*}
where $\bm{\xi}=(\xi_1,\xi_2,\ldots,\xi_d)$. 
Building on this result, we revisit the aforementioned inequality:
\begin{equation*} 
\begin{split} 
\left | \sum_{j=1}^d
\int_{\mathbb{R}^d}
\xi_j\, d\mu^k_j, 
- \int_{\mathbb{R}^d}\langle \nabla v_{\tau}^k ,\bm{\xi}\rangle  u_{\tau}^k\,dx \right |
\leq \frac{\mathcal{W}_2(u_{\tau}^k, u_{\tau}^{k-1})}{\tau}
 \left (\int_{\mathbb{R}^d}
|\bm{\xi}(y)|^2u_{\tau}^{k}\,dx
 \right )^{\frac{1}{2}}  
\end{split}
\end{equation*}
and apply the Riesz representation theorem to find:
\begin{equation*} 
\begin{cases}\vspace{2mm} 
\displaystyle 
\sum_{j=1}^d
\int_{\mathbb{R}^d}
\xi_j\, d\mu^k_j, 
- \int_{\mathbb{R}^d}\langle \nabla v_{\tau}^k ,\bm{\xi}\rangle  u_{\tau}^k\,dx
= \int_{\mathbb{R}^d}\langle R_{\tau}^k, \bm{\xi}\rangle u_{\tau}^k\,dx, 
\\ 
\displaystyle 
\left ( \int_{\mathbb{R}^d}|R_{\tau}^k|^2u_{\tau}^k\,dx\right )^{\frac{1}{2}}
 < \frac{\mathcal{W}_2(u_{\tau}^k, u_{\tau}^{k-1})}{\tau},
 \end{cases} 
\end{equation*}
confirming the existence of an $\mathbb{R}^d$-valued function $R_{\tau}^k$. 
Summarizing the above, we can express the equation as:
\begin{equation*} 
\begin{split} 
- \int_{\mathbb{R}^d}(u_{\tau}^k)^m \diver{\bm{\xi}} \,dx =
\sum_{j=1}^d
\int_{\mathbb{R}^d}
\xi_j\, d\mu^k_j
= \int_{\mathbb{R}^d}\langle u_{\tau}^kR_k+ u^k\nabla v_{\tau}^k , \bm{\xi}\rangle \,dx. 
\end{split}
\end{equation*}
From the fact that 
$u_{\tau}^kR_k+ u_{\tau}^k\nabla v_{\tau}^k \in L^1(\mathbb{R}^d)$, 
it follows that
$
\nabla u_{\tau}^k \in L^1(\mathbb{R}^d) \text{ and }
\bm{\mu}^k= \nabla u_{\tau}^k \mathscr{L}^d. 
$
Consequently, we obtain 
\begin{equation*} 
\begin{split} 
\frac{1}{\tau}\int_{\mathbb{R}^d \times \mathbb{R}^d} 
\langle y-x, \bm{\xi}(y) \rangle \,Mdp_k(x,y)
+ 
\int_{\mathbb{R}^d}\langle \nabla (u_{\tau} ^k)^m 
- u_{\tau}^k\nabla v_{\tau}^k ,\bm{\xi}\rangle \,dx=0. 
\end{split}
\end{equation*}
\end{proof}

The following lemma is an estimation formula 
associated with the minimizing movement scheme and 
can be proven in an abstract setting. 
In this paper, 
it can  be directly derived from the Euler-Lagrange equations \eqref{EL}. 
\begin{lem}[slope estimate {\cite[Lemma 3.1.3]{a-g-s}}]\label{Lem slope estimate}
Let the pair $(u_{\tau}^k, v_{\tau}^k)$ be 
a solution of \eqref{mm}. Then, the following estimate holds. 
\begin{equation} \label{slope estimate}
\begin{split} 
\left (\int_{\mathbb{R}^d} 
\frac{| \nabla (u_{\tau}^k)^m -  u_{\tau}^k \nabla v_{\tau}^k|^2}{u_{\tau}^k}\,dx 
\right)^{\frac{1}{2}}
\leq 
\frac{\mathcal{W}_2(u_{\tau}^k, u_{\tau}^{k-1})}{\tau}. 
\end{split}
\end{equation}
\end{lem}

\begin{proof}
From the first equation in \eqref{EL} 
and the Cauchy-Schwarz inequality we have 
\begin{equation*}
\begin{split}
\tau \int_{\mathbb{R}^d} 
\langle \nabla (u_{\tau}^k)^{m} 
- u_{\tau}^k \nabla v_{\tau}^k, \bm{\xi} \rangle \,dx 
& = 
 \int_{\mathbb{R}^d \times \mathbb{R}^d}
\langle y-x, \bm{\xi}(y) \rangle \,Mdp_k \\
& \leq 
\left (
\int_{\mathbb{R}^d \times \mathbb{R}^d} |x-y|^2 \, Mdp_k 
\right )^{\frac{1}{2}} 
\left (
 \int_{\mathbb{R}^d \times \mathbb{R}^d}
 |\bm{\xi}(y)|^2\,Mdp_k
\right )^{\frac{1}{2}}
\\
& \leq \mathcal{W}_2(u_{\tau}^k, u_{\tau}^{k-1})
\left(
\int_{\mathbb{R}^d}|\bm{\xi}(y)|^2\,u_{\tau}^k(y)\,dy
 \right)^{\frac{1}{2}}. 
\end{split}
\end{equation*}
As a consequence of the Riesz representation theorem, 
there exists 
a function $R_{\tau}^k$ in $L^2(\mathbb{R}^d; u_{\tau}^k\mathscr{L}^d)$ 
such that 
\begin{equation*} 
\begin{cases}\vspace{2mm}
\displaystyle 
\int_{\mathbb{R}^d} 
\langle \nabla (u_{\tau}^k)^{m} 
- u_{\tau}^k \nabla v_{\tau}^k, \bm{\xi} \rangle \,dx 
= \int_{\mathbb{R}^d} 
\langle 
R_{\tau}^k
, \bm{\xi} \rangle \,u_{\tau}^k dx, \\
\displaystyle 
\|R_{\tau}^k\|_{L^2(\mathbb{R}^d; u_{\tau}^k\mathscr{L}^d)} 
\leq 
\frac{\mathcal{W}_2(u_{\tau}^k, u_{\tau}^{k-1})}{\tau}. 
\end{cases}
\end{equation*}
This confirms that the specified inequality is valid. 
\end{proof}
By setting $u_{\tau}^0:=\rho_{\tau}*u_0$, 
we can ensure that $u_{\tau}^0 \in L^2(\mathbb{R}^d)$ even if 
$u_0 \not \in L^2(\mathbb{R}^d)$. 
This regularity is then transferred to 
$w_{\tau}^1 \in W^{2,2}(\mathbb{R}^d)$ and $v_{\tau}^1 \in W^{4,2}(\mathbb{R}^d)$ 
through the second and third equations of \eqref{EL} and elliptic estimates. 
To maintain this chain of regularity, we hope to obtain 
$u_{\tau}^1 \in L^2(\mathbb{R}^d)$ from 
\eqref{EL} and $v_{\tau}^1 \in W^{4,2}(\mathbb{R}^d)$. 
The following lemma demonstrates that 
this expectation is indeed realized. 
\begin{lem}\label{pointwise estimate}
Let $u_{\tau}^k$ be a solution of the minimizing problem \eqref{mm}. 
Then 
$u_{\tau}^k \in L^2(\mathbb{R}^d)$. 
\end{lem}

\begin{proof}
Initially, we establish that 
$
u_{\tau}^k \in L^{\frac{md}{d-1}}(\mathbb{R}^d)$. 
By the Cauchy-Schwarz inequality and the 
slope estimate \eqref{slope estimate}, 
we obtain 
\begin{equation*} 
\begin{split} 
\int_{\mathbb{R}^d} 
| \nabla (u_{\tau}^k)^m - u_{\tau}^k \nabla v_{\tau}^k|\,dx 
& \leq 
\left ( 
\int_{\mathbb{R}^d} 
u_{\tau}^k\,dx 
\right )^{\frac{1}{2}} 
\left ( 
\int_{\mathbb{R}^d} 
\frac{| \nabla (u_{\tau}^k)^m - u_{\tau}^k \nabla v_{\tau}^k|^2}{u_{\tau}^k}\,dx 
\right )^{\frac{1}{2}} \\
& \leq 
\sqrt{M}\ \frac{\mathcal{W}_2(u_{\tau}^k, u_{\tau}^{k-1})}{\tau}. 
\end{split}
\end{equation*}
Using the Euler-Lagrange equations \eqref{EL} and 
elliptic estimates, 
we conclude that
$u_{\tau}^{k-1} \in L^2(\mathbb{R}^d)$ implies $w_{\tau}^k \in W^{2,2}(\mathbb{R}^d)$, 
which in turn ensures  $v_{\tau}^k \in W^{4,2}(\mathbb{R}^d)$. 
So, it follows from the Sobolev inequality that 
\begin{equation*} 
\begin{split} 
\int_{\mathbb{R}^d} |u_{\tau}^k \nabla v_{\tau}^k|\,dx 
& \leq \|u_{\tau}^k\|_{L^{\frac{2d}{d+4}}} \|\nabla v_{\tau}^k \|_{\frac{2d}{d-4}} \\
& \leq 
C_{3,2} \|u_{\tau}^k\|_{L^{\frac{2d}{d+4}}} 
 \| v_{\tau}^k \|_{W^{3,2}}. 
\end{split}
\end{equation*}
Since $m \geq 2-4/d > 2d/(d+4)$ we have 
\[
\|\nabla (u_{\tau}^k)^m \|_{L^1} 
\leq 
\sqrt{M}\ \frac{\mathcal{W}_2(u_{\tau}^k, u_{\tau}^{k-1})}{\tau} 
+ \|u\nabla v_{\tau}^k\|_{L^1}<\infty. 
\]
Consequently, by the Sobolev inequality we obtain 
\begin{equation*} 
\begin{split} 
\|u_{\tau}^k \|_{L^{\frac{md}{d-1}}}^m 
= \|(u_{\tau}^k)^m\|_{\frac{d}{d-1}} 
\leq C_{1,1} \| (u_{\tau}^k)^m \|_{W^{1,1}} <\infty.  
\end{split}
\end{equation*}
Therefore, we have $u_{\tau}^k \in L^{\frac{md}{d-1}}(\mathbb{R}^d)$. 
Note that 
since $d>4$,   
\begin{equation*} 
\begin{split} 
\frac{md}{d-1} \geq 
\frac{2(d-2)}{d}\cdot \frac{d}{d-1}= 
\frac{2(d-2)}{d-1}> \frac{2d}{d+2}
\end{split}
\end{equation*}
holds.

Next, let us suppose that $u_{\tau}^k \in (L^1 \cap L^q)(\mathbb{R}^d)$ 
and $2d/(d+2) < q < 2$. 
Let $p ={qd}/{2(d-q)}$. 
Then we have 
$1 < p < q$. 
By the Cauchy-Schwarz  inequality and H\"older's inequality, 
for any $\bm{\xi} \in C_c^{\infty}(\mathbb{R}^d;\mathbb{R}^d)$ we have  
\begin{equation*}
\begin{split}
\int_{\mathbb{R}^d} 
\langle \nabla (u_{\tau}^k)^{m} 
- u_{\tau}^k \nabla v_{\tau}^k, \bm{\xi} \rangle \,dx 
& \leq 
\left(\int_{\mathbb{R}^d}
\frac{|\nabla (u_{\tau}^k)^{m} 
-  u_{\tau}^k \nabla v_{\tau}^k|^2}
{u_{\tau}^k}\,dx\right)^{\frac{1}{2}}
\left(
\int_{\mathbb{R}^d}|\bm{\xi}|^2\,u_{\tau}^k\,dx
 \right)^{\frac{1}{2}}\\
& 
\leq 
\left(\int_{\mathbb{R}^d}
\frac{|\nabla (u_{\tau}^k)^{m} 
- u_{\tau}^k \nabla v_{\tau}^k|^2}
{u_{\tau}^k}\,dx\right)^{\frac{1}{2}}
\|u_{\tau}^k\|_{L^p}^{\frac{1}{2}}
 \|\bm{\xi}\|_{L^{\frac{2p}{p-1}}}. 
\end{split}
\end{equation*}
By duality and the slope estimate \eqref{slope estimate}, we obtain 
\begin{equation}\label{a priori} 
\begin{split}
\| \nabla (u_{\tau}^k)^{m} - u_{\tau}^k
 \nabla v_{\tau}^k\|_{L^{\frac{2p}{p+1}}} 
& \leq 
\left(\int_{\mathbb{R}^d}
\frac{|\nabla (u_{\tau}^k)^{m} 
- u_{\tau}^k \nabla v_{\tau}^k|^2}
{u_{\tau}^k}\,dx\right)^{\frac{1}{2}}
\|u_{\tau}^k\|_{L^p}^{\frac{1}{2}} \\
& \leq 
\frac{\mathcal{W}_2(u_{\tau}^k, u_{\tau}^{k-1})}{\tau}
\|u_{\tau}^k\|_{L^p}^{\frac{1}{2}}. 
\end{split}
\end{equation}
Since $2p/(p+1) = 2qd/\{2d + q(d-2)\}< d/(d-1)$, 
the Sobolev inequality 
with $C_{1,\frac{2qd}{2d+q(d-2)}}$ 
and 
the inequality \eqref{a priori} and H\"older's inequality 
lead to 
\begin{equation*}
\begin{split}
\|(u_{\tau}^k)^{m}\|_{L^{\frac{2qd}{2d+q(d-4)}}} 
& \leq C_{1,\frac{2qd}{2d+q(d-2)}} 
\|\nabla (u_{\tau}^k)^{m}\|_{L^{\frac{2qd}{2d+q(d-2)}}} 
\\
& \leq 
C_{1,\frac{2qd}{2d+q(d-2)}}
\left ( \frac{\mathcal{W}_2(u_{\tau}^k,u_{\tau}^{k-1})}{\tau} \|u_{\tau}^k\|_{L^p}^{\frac{1}{2}} 
+ \|u_{\tau}^k\|_{L^q} 
\|\nabla v_{\tau}^k\|_{L^\frac{2d}{d-2}} \right ), 
\end{split}
\end{equation*}
which implies that 
$u_{\tau}^k \in {L^{\frac{2m qd}{2d+q(d-4)}}}(\mathbb{R}^d)$ 
since $u \in L^1(\mathbb{R}^d) \cap L^q(\mathbb{R}^d) \subset L^p(\mathbb{R}^d)$. Consequently, 
we see that 
\begin{equation*} 
\begin{split} 
\frac{2d}{d+2} 
< q < 2 
\text{ and } u_{\tau}^k \in L^1(\mathbb{R}^d) \cap L^q(\mathbb{R}^d)
\ \Longrightarrow  \ 
u_{\tau}^k \in L^{\frac{2mqd}{2d+ d(d-4)}}
\end{split}
\end{equation*}
Now, let us define 
\begin{equation*} 
\begin{cases}\vspace{2mm}
\displaystyle q_1:=\frac{md}{d-1} >\frac{2d}{d+2}, \\
q_{j+1} :=\displaystyle \frac{2m q_jd}{2d+q_j(d-4)}. 
\end{cases}
\end{equation*}
Since if $q_j<2$ and $q_j < {2(m-1)d}/{(d-4)}$ 
we have $u_{\tau}^k \in L^{q_{j+1}}(\mathbb{R}^d)$ and $q_{j+1}>q_j$, 
we see that 
$u_{\tau}^k \in L^{\min \{2, \frac{2(m-1)d}{d-4}\}}(\mathbb{R}^d)$. 
When $m\geq 2-4/d$ we have $2(m-1)d/(d-4) \geq 2$. 
Therefore, $u_{\tau}^k \in L^2(\mathbb{R}^d)$. 
\end{proof}

In this paper, we initially derived the Euler-Lagrange equations for each unknown function and subsequently obtained the regularity of the discrete solutions. Conversely, as seen in \cite{m-m-s, b-l, b2015, k-w}, there is also an approach where the regularity of the discrete solutions is discussed before deriving the Euler-Lagrange equations. Knowing the sufficient regularity of discrete solutions in advance can facilitate the derivation of the Euler-Lagrange equations. Therefore, it is advisable to select the method according to the problem at hand.

\section{Lyapunov functional for discrete solutions}
In this section, we derive the Lyapunov functional for problem \eqref{P} using the variational properties of discrete solutions and the Euler-Lagrange equations. For convenience, we introduce the following two operators for the discrete solution 
$z_{\tau}^k \in \{v_{\tau}^k, w_{\tau}^k\}$. 
\begin{equation}\label{DP} 
\begin{split} 
\partial_t z_{\tau}^k := \frac{z_{\tau}^k-z_{\tau}^{k-1}}{\tau}, 
\quad 
\overline{\partial_t}z_{\tau}^k:= \frac{z_{\tau}^k+ z_{\tau}^{k-1}}{\tau}. 
\end{split}
\end{equation}
From the definitions, 
it is clear that these operators are linear, and the following relations hold:
\begin{equation*} 
\begin{cases}\vspace{2mm}
\displaystyle 
z_{\tau}^k 
= \frac{\tau}{2}(\partial_t + \overline{\partial_t})z_{\tau}^k, 
\quad 
 z_{\tau}^{k-1} = \frac{\tau}{2}(\overline{\partial_t}- \partial_t)z_{\tau}^k,
\\ \vspace{2mm}
\displaystyle 
 \langle \partial_t z_{\tau}^k, \overline{\partial_t} z_{\tau}^k\rangle_{L^2}
= \frac{\|z_{\tau}^k\|_{L^2}^2- \|z_{\tau}^{k-1}\|_{L^2}^2}{\tau^2}, \\
\displaystyle 
\nabla \partial_t z_{\tau}^k = \partial_t \nabla  z_{\tau}^k, 
\quad 
\nabla \overline{\partial_t} z_{\tau}^k 
= \overline{\partial_t} \nabla  z_{\tau}^k, 
 \quad 
 \Delta \partial_t  z_{\tau}^k = \partial_t \Delta z_{\tau}^k, 
\quad 
\Delta \overline{\partial_t}  z_{\tau}^k = \overline{\partial_t} \Delta z_{\tau}^k. 
\end{cases}
\end{equation*}

\begin{prop}\label{ED}
Let 
$(u_{\tau}^k, v_{\tau}^k, w_{\tau}^k)$ be a solution in \eqref{mm} and  
$\mathcal{L}$ be the functional defined in \ref{Lyapunov}. 
Then, 
for any $N \in \mathbb{N}$, 
the following inequality holds: 
\begin{equation*} 
\begin{split} 
\frac{1}{2\tau}\sum_{k=1}^N \mathcal{W}_2^2(u_{\tau}^k, u_{\tau}^{k-1}) 
+ \sum_{k=1}^N \mathcal{D}_{\tau}^k 
\leq 
\mathcal{L}(u_{\tau}^0,v_{\tau}^0,w_{\tau}^0)
- \mathcal{L}(u_{\tau}^N,v_{\tau}^N,w_{\tau}^N), 
\end{split}
\end{equation*}
where 
\begin{equation*} 
\begin{split} 
\mathcal{D}_{\tau}^k
& :=   \frac{ \varepsilon_1 \varepsilon_2 \tau^2}{2}
\|\partial_t^2 v_{\tau}^k\|_{L^2}^2
+ \frac{ \kappa_1 \kappa_2 \tau^2}{2} 
\|\partial_t \Delta v_{\tau}^k\|_{L^2}^2 
\\
& \hspace{7mm} + \frac{\tau^2}{2}
\left [
\frac{2(\varepsilon_1 \kappa_2 + \varepsilon_2 \kappa_1)}{\tau}
+ (\gamma_1 \kappa_2 + \gamma_2 \kappa_1)
\right ] \|\partial_t \nabla v_{\tau}^k \|_{L^2}^2 
\\ 
& \hspace{7mm} 
+ \frac{\tau^2}{2}
\left [
\frac{2(\gamma_1 \varepsilon_2 + \gamma_2 \varepsilon_1)}{\tau}
+ \gamma_1 \gamma_2
\right ]\|\partial_t v_{\tau}^k \|_{L^2}^2. 
\end{split}
\end{equation*}
\end{prop}

\begin{proof}
By definition of $u_{\tau}^k$, 
the following inequality holds: 
\begin{equation*} 
\begin{split} 
\mathcal{E}(u_{\tau}^k,v_{\tau}^k) + \frac{1}{2\tau}\mathcal{W}_2^2(u_{\tau}^k, u_{\tau}^{k-1})
\leq 
\mathcal{E}(u_{\tau}^{k-1},v_{\tau}^k) 
+ \frac{1}{2\tau}\mathcal{W}_2^2(u_{\tau}^{k-1}, u_{\tau}^{k-1}). 
\end{split}
\end{equation*}
Considering that $u_{\tau}^{k-1} \in L^2(\mathbb{R}^d)$, we have
\begin{equation} \label{estimate of wsd}
\begin{split} 
\frac{1}{2\tau}\mathcal{W}_2^2(u_{\tau}^k, u_{\tau}^{k-1})
& \leq \mathcal{E}(u_{\tau}^{k-1},v_{\tau}^k) - \mathcal{E}(u_{\tau}^k,v_{\tau}^k) \\
&=\mathcal{E}(u_{\tau}^{k-1}, v_{\tau}^{k-1}) - \mathcal{E}(u_{\tau}^k, v_{\tau}^k)
- \langle u_{\tau}^{k-1}, v_{\tau}^k - v_{\tau}^{k-1} \rangle_{L^2} \\
& = 
\mathcal{E}(u_{\tau}^{k-1}, v_{\tau}^{k-1}) - \mathcal{E}(u_{\tau}^k, v_{\tau}^k)
- \tau \langle u_{\tau}^{k-1}, \partial_t v_{\tau}^k \rangle_{L^2}. 
\end{split}
\end{equation}
Let us rewrite the last term on the right-hand side using the Euler-Lagrange equations \eqref{EL}.
The second and third equations in \eqref{EL} can be expressed 
using the operator 
\eqref{DP} as follows:
\begin{equation} \label{RWEL}
\begin{cases}
\begin{split} 
w_{\tau}^k 
&= 
\left (\varepsilon_1 \partial_t -\kappa_1 \Delta + \gamma_1 \right )
v_{\tau}^k,  \\
 u_{\tau}^{k-1} 
&= (\varepsilon_2 \partial_t 
-\kappa_2 \Delta + \gamma_2)w_{\tau}^{k} \\ 
& = 
\left (\varepsilon_2 \partial_t -\kappa_2 \Delta + \gamma_2 \right )
\left (\varepsilon_1 \partial_t -\kappa_1 \Delta + \gamma_1 \right )
v_{\tau}^k 
\\ 
&= \varepsilon_1 \varepsilon_2 \partial_t^2 v_{\tau}^k 
- \{\varepsilon_1(\kappa_2 \Delta - \gamma_2) 
+ \varepsilon_2(\kappa_1 \Delta -\gamma_1) \}\partial_t v_{\tau}^k  \\
& \hspace{5mm} 
+ (\kappa_1 \Delta - \gamma_1)(\kappa_2 \Delta -\gamma_2) v_{\tau}^k. 
\end{split}
\end{cases}
\end{equation}
Define $v_{\tau}^{-1}$ as follows: 
\begin{equation*} 
v_{\tau}^{-1}
:=v_{\tau}^0-\tau
\left (
\frac{\kappa_1\Delta v_{\tau}^0 -\gamma_1 v_{\tau}^0+w_{\tau}^0 }{\varepsilon_1}
\right )
\end{equation*}
This definition allows us to assume that 
\eqref{RWEL} holds for $k \geq 1$. 
Consequently, we have 
\begin{equation*} 
\begin{split} 
\langle u_{\tau}^{k-1}, \partial_t v_{\tau}^k \rangle_{L^2} 
& = \varepsilon_1 \varepsilon_2
\langle \partial_t^2 v_{\tau}^k, \partial_t v_{\tau}^k \rangle_{L^2}
- 
\langle 
(\varepsilon_1\kappa_2+ \varepsilon_2\kappa_1 )
\Delta  -\gamma_1\varepsilon_2 -\gamma_2\varepsilon_1 )
\partial_t v_{\tau}^k, \partial_t v_{\tau}^k \rangle_{L^2} \\
& \hspace{5mm}
+ 
\langle 
(\kappa_1 \Delta - \gamma_1)(\kappa_2 \Delta -\gamma_2)
 v_{\tau}^k, \partial_t v_{\tau}^k \rangle_{L^2} \\
 & = I_1 + I_2 + I_3, 
\end{split}
\end{equation*}
where 
\begin{align*}
I_1&:= \varepsilon_1 \varepsilon_2
\langle \partial_t^2 v_{\tau}^k, \partial_t v_{\tau}^k \rangle_{L^2} \\
&\  = \varepsilon_1 \varepsilon_2\frac{\tau}{2}
\langle 
\partial_t^2 v_{\tau}^k, (\partial_t + \overline{\partial_t})\partial_t v_{\tau}^k
\rangle_{L^2} 
 = \varepsilon_1 \varepsilon_2 \frac{\tau}{2} 
\|\partial_t^2 v_{\tau}^k\|_{L^2}^2
+ \varepsilon_1 \varepsilon_2
\frac{\|\partial_t v_{\tau}^k\|_{L^2}^2-\|\partial_t v_{\tau}^{k-1}\|_{L^2}^2}{2\tau}, \\
I_2&:=- 
\langle 
(\varepsilon_1\kappa_2+ \varepsilon_2\kappa_1 )
\Delta  -\gamma_1\varepsilon_2 -\gamma_2\varepsilon_1 )
\partial_t v_{\tau}^k, \partial_t v_{\tau}^k \rangle_{L^2}
\\
& \  =-(\varepsilon_1\kappa_2+ \varepsilon_2\kappa_1 )
\langle \Delta \partial_t v_{\tau}^k, \partial_t v_{\tau}^k \rangle_{L^2} 
+ (\gamma_1\varepsilon_2 + \gamma_2\varepsilon_1  )
\|\partial_t v_{\tau}^k \|_{L^2}^2 \\
 & \ = 
(\varepsilon_1\kappa_2+ \varepsilon_2\kappa_1 )
\|\nabla \partial_t v_{\tau}^k \|_{L^2}^2 
+ (\gamma_1\varepsilon_2 + \gamma_2 \varepsilon_1)
\|\partial_t v_{\tau}^k \|_{L^2}^2, \\
I_3& :=\langle 
(\kappa_1 \Delta - \gamma_1)(\kappa_2 \Delta -\gamma_2)
 v_{\tau}^k, \partial_t v_{\tau}^k \rangle_{L^2} 
 = 
 \langle 
(\kappa_1 \Delta - \gamma_1)
 v_{\tau}^k, (\kappa_2 \Delta -\gamma_2) \partial_t v_{\tau}^k \rangle_{L^2} 
 \\
&\  = \kappa_1 \kappa_2 
\frac{\tau}{2}
\langle (\partial_t+\overline{\partial_t})
\Delta v_{\tau}^k, \partial_t \Delta v_{\tau}^k \rangle_{L^2}
 + 
\frac{\tau(\gamma_1\kappa_2+ \gamma_2\kappa_1)}{2}
 \langle 
 (\partial_t+\overline{\partial_t})
 \nabla v_{\tau}^k, \partial_t \nabla v_{\tau}^k \rangle_{L^2}  \\
&\  \hspace{10mm}
+ \frac{\gamma_1 \gamma_2 \tau }{2} 
\langle (\partial_t+\overline{\partial_t})v_{\tau}^k, \partial_t v_{\tau}^k \rangle_{L^2} 
 \\
&\  = \kappa_1 \kappa_2 
\frac{\tau}{2}\|\partial_t \Delta v_{\tau}^k\|_{L^2}^2 
+ \kappa_1 \kappa_2 
\frac{\|\Delta v_{\tau}^k\|_{L^2}^2-\|\Delta v_{\tau}^{k-1}\|_{L^2}^2}{2\tau} \\
&\  \hspace{5mm} + 
\frac{\tau(\gamma_1\kappa_2+ \gamma_2\kappa_1)}{2}
\|\partial_t \nabla v_{\tau}^k\|_{L^2}^2
+ 
\frac{(\gamma_1\kappa_2+ \gamma_2\kappa_1)}{2\tau}
\{
\|\nabla v_{\tau}^k\|_{L^2}^2-\|\nabla v_{\tau}^{k-1}\|_{L^2}^2
\}
\\
&\  \hspace{10mm} 
+ 
\frac{\gamma_1 \gamma_2 \tau}{2}\|\partial_t v_{\tau}^k\|_{L^2}^2
+ \frac{\gamma_1 \gamma_2}{2\tau}
\{
\|v_{\tau}^k\|_{L^2}^2- \|v_{\tau}^{k-1}\|_{L^2}^2
\}. 
\end{align*}
From the above, 
it can be observed that the right-hand side of  
\eqref{estimate of wsd} can be divided into two parts: 
the norm values of the terms parameterized by $k$, 
and the differences in the norm values 
of the terms parameterized by $k$ and $k-1$. 
The former is denoted as $\mathcal{D}_{\tau}^k$, 
and the latter is denoted as 
$\mathcal{L}(u_{\tau}^{k-1},v_{\tau}^{k-1},w_{\tau}^{k-1})
- \mathcal{L}(u_{\tau}^k,v_{\tau}^k,w_{\tau}^k)$. 
Then, we can see that the following inequality holds. 
\[
\frac{1}{2\tau} \mathcal{W}_2^2(u_{\tau}^k, u_{\tau}^{k-1}) + \mathcal{D}_{\tau}^k 
\leq 
\mathcal{L}(u_{\tau}^{k-1}, v_{\tau}^{k-1}, w_{\tau}^{k-1})
-\mathcal{L}(u_{\tau}^k, v_{\tau}^k, w_{\tau}^k). 
\]
By summing both sides from $k=1$ to $N$, 
we obtain the desired inequality. 
\end{proof}

\begin{rem}
As demonstrated in the proof of Proposition \ref{ED}, the derivation of $\mathcal{L}$ does not utilize the properties of the first term of $\mathcal{E}$. Consequently, if 
$\mathcal{E}$ takes the form
\[
\mathcal{E}(u,v)=\mathcal{E}_0(u) -\int_{\mathbb{R}^d}uv\,dx, 
\]
a similar derivation is possible. Notably, it is not necessary for $u$ to be a gradient flow in the Wasserstein space. 
Therefore, by replacing $\mathcal{E}_0$ or considering the gradient $\nabla_u$ 
as a gradient in a space other than the Wasserstein space, 
it is possible to consider a chimera gradient flow 
endowed with $\mathcal{L}$ as a Lyapunov functional 
different from \eqref{P}. 
\end{rem}

\section{Lower bounds of $\mathcal{L}$}
In this section, 
we examine the lower boundedness of $\mathcal{L}$  
on 
\begin{equation} \label{effective domain}
X_M:=
\{
(u,v,w) \in (L^1\cap L^m)(\mathbb{R}^d) \times W^{2,2}(\mathbb{R}^d) 
\times L^2(\mathbb{R}^d)
\ |\  
u, v, w\geq 0, \ \|u\|_{L^1}=M
\}, 
\end{equation} 
and investigate its coercivity, expressed by the inequality
\begin{equation} \label{coercivity of L}
\begin{split} 
\mathcal{L}(u,v,w) 
& \geq \alpha \|u\|_{L^m}^m + \beta \|\Delta v\|_{L^2}^2  \\
&\hspace{5mm} + \frac{\gamma_1 \kappa_2 + \gamma_2 \kappa_1}{2}
\|\nabla v\|_{L^2}^2
+ \frac{\gamma_1 \gamma_2}{2}
\|v\|_{L^2}^2 
+ \frac{\varepsilon_2}{2\varepsilon_1}
\|\kappa_1 \Delta v - \gamma_1 v + w\|_{L^2}^2
\geq 0
\end{split}
\end{equation}
for some positive constants $\alpha$, $\beta$, and 
for all elements $(u,v,w) \in X_M$. 
By carefully estimating the lower bound of the negative term in $\mathcal{L}$, 
we derive the following propositions. 

\begin{prop}[lower bounds of $\mathcal{L}$ in subcritical case]
\label{lower bounds in subcritical}
Assume $m> 2-\frac{4}{d}$ and $M>0$. 
Then, the functional $\mathcal{L}$ is bounded from below on 
$X_M$ and satisfies \eqref{coercivity of L}. 
\end{prop}

\begin{prop}[lower bounds of $\mathcal{L}$ in critical case]
\label{lower bounds in critical}
Assume $m=2-\frac{4}{d}$ and $M>0$. 
Then, the functional $\mathcal{L}$ is bounded from below on $X_M$ 
if and only if $M\leq M_*$, 
where $M_*$ is defined in \eqref{critical mass}. 
In particular, when $M<M_*$, 
\eqref{coercivity of L} holds. 
\end{prop}
\begin{proof}[Proof of Proposition \ref{lower bounds in subcritical}]
When $u \in L^1(\mathbb{R}^d) \cap L^m(\mathbb{R}^d)$ and $v \in W^{2,2}(\mathbb{R}^d)$, the negative term of $\mathcal{L}$ can be estimated as follows.
\begin{equation} \label{FE}
\begin{split} 
\int_{\mathbb{R}^d} uv\,dx 
& \leq \|u\|_{L^{\frac{2d}{d+4}}}\|v\|_{L^{\frac{2d}{d-4}}} \quad 
(\text{H\"older's inequality})\\
& \leq \|u\|_{L^1}^{1-\theta} \|u\|_{L^m}^{\theta}
\|v\|_{L^{\frac{2d}{d-4}}}
\quad (\text{The interpolation inequality}) \\
& \leq C_{1} \|u\|_{L^{\frac{2d}{d+4}}}\|\nabla v\|_{L^{\frac{2d}{d-2}}}
\quad  (\text{The Sobolev inequality}) 
\\
& \leq 
C_{1}C_{2} \|u\|_{L^1}^{1-\theta} \|u\|_{L^m}^{\theta}
 \|D^2 v\|_{L^2}
 \quad (\text{The Sobolev embedding theorem})
 \\
 & \leq 
 C_{1}C_{2}C_3
  \|u\|_{L^1}^{1-\theta} \|u\|_{L^m}^{\theta}
 \|\Delta v\|_{L^2}
 \quad (\text{Corollary 9.10 in \cite{g-t}}) \\
 & \leq 
 \frac{(C_1C_2C_3)^2}{2\kappa_1 \kappa_2}
\|u\|_{L^1}^{2(1-\theta)} \|u\|_{L^m}^{2 \theta} 
+ \frac{\kappa_1 \kappa_2 }{2}\|\Delta v\|_{L^2}^2 
 \quad (\text{Young's inequality})
\end{split}
\end{equation}
where 
\begin{equation} \label{1m-exp}
\begin{split} 
\theta= \frac{m(d-4)}{2d(m-1)}, 
\quad 
1-\theta 
= \frac{4m-(2-m)d}{2d(m-1)}. 
\end{split}
\end{equation}
Consequently, the functional $\mathcal{L}$ satisfies
\begin{equation*} 
\begin{split} 
\mathcal{L}(u,v,w) 
\geq 
\frac{1}{m-1}\|u\|_{L^m}^m 
- \frac{(C_1C_2C_3)^2}{2\kappa_1 \kappa_2}
\|u\|_{L^1}^{2(1-\theta)} \|u\|_{L^m}^{2 \theta}. 
\end{split}
\end{equation*}
From this, it follows that when 
$m > 2 \theta$, i.e.,  
\[
m > 2-\frac{4}{d}, 
\]
$\mathcal{L}$ is bounded below regardless of the value of $\|u\|_{L^1}$ 
and \eqref{coercivity of L} holds. 
\end{proof}

\begin{proof}[Proof of Proposition \ref{lower bounds in critical}]
The proof is divided into three parts.

\vspace{2mm}

\noindent
\textbf{(i) lower bounds and coercivity}

\vspace{2mm}

To obtain a better, lossless estimate in \eqref{FE}, we define
\begin{equation} \label{CM}
\Tilde{C}_*:= 
\sup_{(u,v)  \in (L^1 \cap L^m)(\mathbb{R}^d) \times \dot{H}^2(\mathbb{R}^d)}
\frac{\displaystyle \int_{\mathbb{R}^d} uv\,dx}
{\|u\|_{L^1}^{1-\theta}\|u\|_{L^m}^{\theta}\|\Delta v\|_{L^2}},  
\end{equation}
where 
$\dot{H}^2(\mathbb{R}^d)$ is the closure of compactly supported 
smooth functions in the seminorm $\| \Delta \cdot \|_{L^2}$ 
and $\theta$ is given by \eqref{1m-exp}. 
From the definition of $\Tilde{C}_*$ and Young's inequality, 
for any $\delta$ such that  
$0 \leq \delta < \kappa_1\kappa_2$, 
the following inequality holds
\begin{equation*} 
\begin{split} 
\int_{\mathbb{R}^d} uv\,dx 
\leq \frac{\Tilde{C}_*^2}{2(\kappa_1 \kappa_2-\delta)}
M^{\frac{4}{d}} \|u\|_{L^m}^{m} 
+ \frac{\kappa_1 \kappa_2 -\delta }{2}\|\Delta v\|_{L^2}^2. 
\end{split}
\end{equation*}
Consequently, we have 
\begin{equation*} 
\begin{split} 
& \frac{1}{m-1} \int_{\mathbb{R}^d} u^m\,dx - \int_{\mathbb{R}^d}uv\,dx 
+ \frac{\kappa_1 \kappa_2}{2}\int_{\mathbb{R}^d} |\Delta v|^2\,dx 
\\
& \geq 
\left (
\frac{1}{m-1}
-\frac{\Tilde{C}_*^2}{2(\kappa_1 \kappa_2 -\delta)}M^{\frac{4}{d}}
 \right )\|u\|_{L^m}^m + \frac{\delta}{2} \|\Delta v\|_{L^2}^2  \\
 &= 
  \frac{\Tilde{C}_*^2}{2(\kappa_1 \kappa_2 -\delta)} 
 \left (\frac{2d}{d-4}\frac{(\kappa_1 \kappa_2-\delta)}{\Tilde{C}_*^2} - M^{\frac{4}{d}}
 \right )\|u\|_{L^m}^m + \frac{\delta}{2} \|\Delta v\|_{L^2}^2
\end{split}
\end{equation*}
Setting $\delta=0$ shows that the right-hand side is lower bounded when 
$M \leq 
\left ( \frac{2d}{d-4}\frac{\kappa_1\kappa_2}{\Tilde{C}_*^2} \right )^{\frac{d}{4}}$, 
and as shown in the next, 
$\Tilde{C}_*^2=C_*^2$, defined in \eqref{dividing constant}, holds, 
indicating that this implies $M\leq M_*$. 
When $M<M_*$, since the coefficient of $\|u\|_{L^m}^m$ 
can be made positive by appropriately choosing $\delta$, 
\eqref{coercivity of L} holds. 

\vspace{2mm}

\noindent
\textbf{(ii) $\mathbf{\Tilde{C}_*^2= C_*^2}$}

\vspace{2mm}

By definition, we have 
\begin{equation*} 
\begin{split} 
\Tilde{C}_*^2
& =\sup_{(u,v) 
\in (L^1 \cap L^m)(\mathbb{R}^d) \times \dot{H}^2(\mathbb{R}^d)}
\frac{\displaystyle 
\left ( \int_{\mathbb{R}^d} uv\,dx \right )^2}
{\|u\|_{L^1}^{2(1-\theta)}\|u\|_{L^m}^{2\theta}\|\Delta v\|_{L^2}^2} \\
& = \sup_{u \in (L^1 \cap L^m)(\mathbb{R}^d)}
\frac{1}{\|u\|_{L^1}^{2(1-\theta)}\|u\|_{L^m}^{2\theta}}
\left (
\inf_{v \in \dot{H}^2(\mathbb{R}^d)}
\frac{\|\Delta v\|_{L^2}^2 }
{\left ( \int_{\mathbb{R}^d}uv\,dx\right )^2} \right )^{-1}. 
\end{split}
\end{equation*}
Define 
\begin{equation*} 
\begin{split} 
A_u(v):=\frac{\|\Delta v\|_{L^2}^2}
{\left ( \int_{\mathbb{R}^d}uv\,dx\right )^2}, 
\end{split}
\end{equation*}
which is bounded below, 
and let $\{v_n\}$ be a minimizing sequence for $A_u$. 
Since $\{v_n/\|\Delta v_n\|_{L^2}\}$ is also a minimizing sequence, 
$\{ v_n/\|\Delta v_n\|_{L^2}\}$ forms a bounded sequence 
in $\dot{H}^2(\mathbb{R}^d)
\hookrightarrow L^{\frac{2d}{d-4}}(\mathbb{R}^d)$. 
Thus, we can extract a subsequence $\{v_{n(j)}\}$ such that 
\begin{equation*} 
\begin{split} 
&\Delta v_{n(j)} \rightharpoonup \Delta v_* \text{ weakly in }L^2(\mathbb{R}^d), \\
&v_{n(j)} \rightharpoonup v_* \text{ weakly in }L^{\frac{2d}{d-4}}(\mathbb{R}^d). 
\end{split}
\end{equation*}
Since 
$u \in (L^1 \cap L^m)(\mathbb{R}^d) \subset L^{\frac{2d}{d+4}}(\mathbb{R}^d)$, 
the dual space of $L^{\frac{2d}{d-4}}(\mathbb{R}^d)$, it follows that
\[
\int_{\mathbb{R}^d} uv_{n(j)} \,dx \to \int_{\mathbb{R}^d} uv_* \,dx\ 
(j \to \infty). 
\]
Moreover, 
\[
\|\Delta v_*\| \leq \liminf_{j \to \infty} \|\Delta v_{n(j)}\|_{L^2}
\]
implies that
\[
A_u(v_*) \leq \liminf_{j \to \infty} A_u(v_{n(j)})
\]
confirming the existence of a minimizer $v_*$. 
For any $\varphi \in \dot{H}^2(\mathbb{R}^d)$, 
the following holds
\begin{equation*} 
\begin{split} 
\frac{d}{d\varepsilon}
\left [ 
{\left ( \int_{\mathbb{R}^d}u(v_*+\varepsilon \varphi)\,dx\right )^2}
A_u(v_*+\varepsilon \varphi)
\right ]_{\varepsilon=0}
=
\frac{d}{d\varepsilon}
{\|\Delta v_*+\varepsilon \Delta \varphi\|_{L^2}^2} 
\biggm |_{\varepsilon=0}
\end{split}
\end{equation*}
yielding
\begin{equation*} 
\begin{split} 
A_u(v_*)\int_{\mathbb{R}^d}uv_*\,dx
\int_{\mathbb{R}^d}u\varphi\,dx
=
\int_{\mathbb{R}^d} \Delta v_* \Delta \varphi\,dx
= \int_{\mathbb{R}^d} \varphi(-\Delta)^2v_* \,dx, 
\end{split}
\end{equation*}
that is, 
\begin{equation} \label{relation}
\begin{split} 
\left (A_u(v_*)  \int_{\mathbb{R}^d}uv_*\,dx
\right ) u
=
(-\Delta)^2 v_*. 
\end{split}
\end{equation}
Continuing from the earlier equations, 
by solving for $v_*$ using $(-\Delta)^{-2}$ and then
multiplying both sides by $u$ and integrating, 
we derive that 
\begin{equation*} 
\begin{split} 
\frac{1}{A_u(v_*)}
= \int_{\mathbb{R}^d} 
u(-\Delta)^{-2}u\,dx
=  \int_{\mathbb{R}^d \times \mathbb{R}^d \times \mathbb{R}^d} 
\mathcal{K}(x-z)\mathcal{K}(z-y)u(x)u(y)\,dxdydz
\end{split}
\end{equation*}
where, $\mathcal{K}$ is the fundamental solution of 
$-\Delta$, given by 
\[
\mathcal{K}(x)= \frac{1}{(d-2)\omega_d |x|^{d-2}}, 
\]
and $\omega_d :={2\pi^{{d}/{2}}}/{\Gamma({d}/{2})}$ 
is the surface area of the unit sphere $\mathbb{S}^{d-1}$ in $\mathbb{R}^d$. 
Therefore $\Tilde{C}_*^2=C_*^2$.

\vspace{2mm}

\noindent
\textbf{(iii) unboundedness of $\mathcal{L}$ from below}

\vspace{2mm}

From the analysis provided, we obtain an expression for $C_*$ 
given by  
\[
C_*= \sup_{u \in (L^1 \cap L^m)(\mathbb{R}^d)}\frac{\int_{\mathbb{R}^d}uv_*\,dx}
{\|u\|_{L^1}^{\frac{2}{d}}\|u\|_{L^m}^{\frac{m}{2}}\|\Delta v_*\|_{L^2}}, 
\]
where $v_*$ is defined as specified in equation \eqref{relation}. 
Therefore, for any $\delta>0$, there exists 
a pair 
$(U_*, V_*) \in (L^1\cap L^m)(\mathbb{R}^d) 
\times 
\dot{H}^2(\mathbb{R}^d)$ such that 
\begin{equation*} 
\begin{split} 
\int_{\mathbb{R}^d}U_*V_*\,dx
& > (C_* -\delta)M^{\frac{2}{d}}\|U_*\|_{L^m}^{\frac{m}{2}}\|\Delta V_*\|_{L^2} \\
& = \frac{\sqrt{\alpha}}{\sqrt{\kappa_1\kappa_2}}
(C_* -\delta)M^{\frac{2}{d}}\|U_*\|_{L^m}^{\frac{m}{2}}
\cdot 
\frac{\sqrt{\kappa_1\kappa_2}}{\sqrt{\alpha}}
\|\Delta V_*\|_{L^2} \\
& = 
\frac{\alpha}{2\kappa_1\kappa_2}
(C_* -\delta)^2M^{\frac{4}{d}}\|U_*\|_{L^m}^{m}
+ 
\frac{\kappa_1\kappa_2}{2\alpha}
\|\Delta V_*\|_{L^2}^2, 
\end{split}
\end{equation*}
where 
$(U_*, V_*)$ satisfies 
\begin{equation} \label{relation2}
\begin{split} 
\frac{\|\Delta V_*\|_{L^2}^2}
{
\int_{\mathbb{R}^d}U_*V_*\,dx
} 
U_*
=
(-\Delta)^2 V_*. 
\end{split}
\end{equation}
and we set 
\[
\alpha:= \frac{\kappa_1 \kappa_2 \|\Delta V_*\|_{L^2}}
{C_* M^{\frac{2}{d}}\|U_*\|_{L^m}^{\frac{m}{2}}} 
\]
to utilize the equality condition of Young's inequality.
So, we have 
\begin{equation*} 
\begin{split} 
\int_{\mathbb{R}^d}U_*\left ( \frac{V_*}{\alpha} \right )\,dx
>
\frac{
(C_* -\delta)^2M^{\frac{4}{d}}}{2\kappa_1\kappa_2}\|U_*\|_{L^m}^{m}
+ 
\frac{\kappa_1\kappa_2}{2}
\left \|
\Delta \left ( \frac{V_*}{\alpha} \right )
\right \|_{L^2}^2. 
\end{split}
\end{equation*}
Define the functional $\mathcal{L}_0$ by 
\begin{equation*} 
\mathcal{L}_0(u,v,w) := 
\int_{\mathbb{R}^d} \frac{u^m}{m-1}\,dx - \int_{\mathbb{R}^d} uv\,dx 
+ 
\frac{\kappa_1 \kappa_2}{2}\int_{\mathbb{R}^d}|\Delta v|^2\,dx 
 + 
\displaystyle 
\frac{\varepsilon_2}{2\varepsilon_1}
\int_{\mathbb{R}^d} |\kappa_1 \Delta v  + w|^2\,dx. 
\end{equation*}
Then, it holds that 
\begin{equation*} 
\begin{split} 
\mathcal{L}_0(U_*, \alpha^{-1}V_*, -\kappa_1 \alpha^{-1}\Delta V_*) <
\left (\frac{d}{d-4} - 
\frac{
(C_* -\delta)^2M^{\frac{4}{d}}}{2\kappa_1\kappa_2}
\right )\|U_*\|_{L^m}^{m}. 
\end{split}
\end{equation*}
When $M >M_*$, 
we can choose $\delta$ such that the right-hand side becomes negative. 
Particularly, 
since we may assume $U_*, V_* \geq 0$, 
it follows from the relation \eqref{relation2} that 
$-\kappa_1 \alpha^{-1}\Delta V_* \geq 0$.  
Now, define 
\[ (U_{\lambda}(x), V_{\lambda}(x), W_{\lambda}(x)):=
(\lambda^dU_*(\lambda x), \lambda^{d-4}
\alpha^{-1}V_*(\lambda x), -\lambda^{d-2}\kappa_1\alpha^{-1} \Delta V_*(\lambda x)). 
 \]
Then, we can check the following.   
\begin{equation*} 
\begin{cases}\vspace{2mm} 
\|U_{\lambda}\|_{L^1}= \|U_*\|_{L^1}, \\
\|U_{\lambda}\|_{L^m}^m = \lambda^{d-4}\|U_*\|_{L^m}^m, 
\end{cases}
\ 
\begin{cases}\vspace{2mm}
\|V_{\lambda}\|_{L^2}^2= \lambda^{d-8}\|\alpha^{-1}V_*\|_{L^2}^2,  \\ \vspace{2mm}
\|\nabla V_{\lambda}\|_{L^2}^2 =  \lambda^{d-6}\|\alpha^{-1}\nabla V_*\|_{L^2}^2, \\
\|\Delta V_{\lambda}\|_{L^2}^2 = \lambda^{d-4}\|\alpha^{-1}\Delta V_*\|_{L^2}^2, 
\end{cases}
\
\|W_{\lambda}\|_{L^2}^2 = \lambda^{d-4} \|\kappa_1\alpha^{-1} \Delta V_*\|_{L^2}^2. 
\end{equation*}
Consequently, we obtain 
\begin{equation*} 
\mathcal{L}_0(U_{\lambda}(x), V_{\lambda}(x), W_{\lambda}(x)) = 
\lambda^{d-4} \mathcal{L}_0(U_*, \alpha^{-1}V_*, -\kappa_1 \alpha^{-1}\Delta V_*) 
\to -\infty \quad (\lambda \to \infty), 
\end{equation*}
from which 
\begin{equation*} 
\begin{split} 
\mathcal{L}(U_{\lambda}(x), V_{\lambda}(x), W_{\lambda}(x)) 
\to -\infty \quad (\lambda \to \infty).  
\end{split}
\end{equation*}
follows. 
\end{proof}

\section{Uniform bounds}

In this section, we introduce the piecewise constant interpolation of discrete solutions and derive uniform estimates for them. These uniform estimates play a crucial role in leading to the weak compactness of the discrete solutions.

\begin{defn}[piecewise constant interpolation of discrete solutions]
\label{piecewise constant interpolation}
We define the \textbf{piecewise constant interpolation of discrete solutions} by 
\begin{equation*} 
\begin{cases}\vspace{2mm}
\overline{u}_{\tau}(t):=u_{\tau}^k, & t \in ((k-1)\tau, k\tau], \\ \vspace{2mm}
\overline{v}_{\tau}(t):=v_{\tau}^k, & t \in ((k-1)\tau, k\tau], \\ 
\overline{w}_{\tau}(t):=w_{\tau}^k, & t \in ((k-1)\tau, k\tau]. 
\end{cases}
\qquad \begin{cases}\vspace{2mm}
\underline{u}_{\tau}(t):=u_{\tau}^{k-1}, & t \in ((k-1)\tau, k\tau], \\ \vspace{2mm}
\underline{v}_{\tau}(t):=v_{\tau}^{k-1}, & t \in ((k-1)\tau, k\tau], \\ 
\underline{w}_{\tau}(t):=w_{\tau}^{k-1}, & t \in ((k-1)\tau, k\tau]. 
\end{cases}
\end{equation*}
\end{defn}

\begin{prop}[uniform bounds]\label{uniform bounds}
Assume \eqref{uniform L^2} and 
\begin{equation} \label{lbc}
m>2-\frac{4}{d} \text{ or }\left ( m=2-\frac{4}{d} \text{ and }
M < M_* \right ). 
\end{equation}
Then, for any $T>0$ and 
for any $N \in \mathbb{N}$, 
the following inequalities hold:  
\begin{equation*} 
\begin{split} 
& \sup_{\substack{t\in [0,T] \\ \tau \in (0, \tau_*)}} 
\left \{ 
\|\overline{u}_{\tau}(t)\|_{L^m}^m, 
\|\overline{v}_{\tau}(t)\|_{W^{2,2}}^2, 
\|\overline{w}_{\tau}(t)\|_{H^1}^2, 
\int_{\mathbb{R}^d}|x|^2\overline{u}_{\tau}(t)\,dx
\right \}<\infty, 
\\
&
\sup_{\tau \in (0, \tau_*)}
\int_0^T
\left (
\|\overline{u}_{\tau}(t)\|_{L^2}^2
+\|\nabla \overline{u}_{\tau}^{m}(t)\|_{L^{\frac{d}{d-1}}}^{p_*}
 +\|\overline{v}_{\tau}(t)\|_{W^{3,2}}^2
+\|\overline{w}_{\tau}(t)\|_{W^{2,2}}^2
\right )
\,dt
<\infty, 
\\ 
 &
 \sup_{\tau \in (0,\tau_*)}\left \{ 
\sum_{k=1}^N \frac{\mathcal{W}_2^2(u_{\tau}^k, u_{\tau}^{k-1})}{\tau}, \ 
\sum_{k=1}^N \frac{\|v_{\tau}^k-v_{\tau}^{k-1}\|_{H^1}^2}{\tau}, \ 
\sum_{k=1}^N \frac{\|w_{\tau}^k-w_{\tau}^{k-1}\|_{L^2}^2}{\tau} \right \}<\infty, 
\end{split}
\end{equation*}
where 
\begin{equation} \label{p_*}
p_*:=
\begin{cases}\vspace{2mm}
\displaystyle\frac{2(d-2m)}{d(2-m)}, & \text{ if }m<2, \\
2, & \text{ if }m \geq 2. 
\end{cases}
\end{equation}
\end{prop}

\begin{rem}
When $d \geq 5$, $p_* >1$, and 
further, when $d \geq 7$, $p_* \geq 2$. 
\end{rem}

In the following, we will prove Proposition \ref{uniform bounds} 
by dividing it into several lemmas.

\begin{lem}\label{vL2}
Assume \eqref{lbc}. 
then, 
for any $T>0$ and for any $N \in \mathbb{N}$ such that 
$\tau N \leq T$
the following 
inequality holds. 
\begin{equation*} 
\begin{split} 
& \frac{ \varepsilon_1 \varepsilon_2 }{2\tau}
\sum_{k=1}^N
\|v_{\tau}^k - v_{\tau}^{k-1}\|_{L^2}^2 
\leq T \mathcal{L}(u_{\tau}^0, v_{\tau}^0, w_{\tau}^0). 
\end{split}
\end{equation*}
\end{lem}

\begin{proof}
According to Propositions \ref{lower bounds in subcritical} 
and \ref{lower bounds in critical}, 
and based on \eqref{coercivity of L}, we establish that 
\[
\frac{\varepsilon_2}{2}
 \| \kappa_1\Delta v_{\tau}^k-\gamma_1 v_{\tau}^k + w_{\tau}^k \|_{L^2}^2
\leq \varepsilon_1 \mathcal{L}(u_{\tau}^k,v_{\tau}^k,w_{\tau}^k). 
\]
Furthermore, employing \eqref{EL} and Proposition \ref{ED}, 
we deduce that  
\begin{equation*} 
\begin{split} 
\frac{\varepsilon_2}{2} \|\varepsilon_1\partial_t v_{\tau}^k \|_{L^2}^2
= 
\frac{\varepsilon_2}{2}
 \| \kappa_1 \Delta v_{\tau}^k-\gamma_1 v_{\tau}^k + w_{\tau}^k \|_{L^2}^2
\leq \varepsilon_1 \mathcal{L}(u_{\tau}^k,v_{\tau}^k,w_{\tau}^k)
\leq \varepsilon_1 {L}(u_{\tau}^0,v_{\tau}^0,w_{\tau}^0). 
\end{split}
\end{equation*}
Thus, 
by multiplying both sides by $\tau$ and summing over $k$ from $1$ to $N$, 
we obtain 
\begin{equation*} 
\begin{split} 
\frac{\varepsilon_1^2 \varepsilon_2}{2} 
\sum_{k=1}^N \tau
\|\partial_t v_{\tau}^k \|_{L^2}^2
\leq  \varepsilon_1 \sum_{k=1}^N \tau {L}(u_{\tau}^0,v_{\tau}^0,w_{\tau}^0)
\leq  \varepsilon_1 T \mathcal{L}(u_{\tau}^0, v_{\tau}^0, w_{\tau}^0). 
\end{split}
\end{equation*}
\end{proof}
The following lemma follows immediately from 
Proposition \ref{ED} and Lemma \ref{vL2}.
\begin{lem}\label{Hcontinuity}
Assume \eqref{lbc}. 
Then, for any $T>0$ and for any $N \in \mathbb{N}$ such that 
$\tau N \leq T$
the following 
inequality holds.
\begin{equation*} 
\begin{split} 
\frac{1}{2\tau}
\left ( 
\sum_{k=1}^N \mathcal{W}_2^2(u_{\tau}^k, u_{\tau}^{k-1}) 
+ \sum_{k=1}^N d_{H^1}^2(v_{\tau}^k, v_{\tau}^{k-1}) \right )
\leq (T+1)
\mathcal{L}(u_{\tau}^0,v_{\tau}^0,w_{\tau}^0), 
\end{split}
\end{equation*}
where 
\[d_{H^1}^2(v_{\tau}^k, v_{\tau}^{k-1}):=
(\varepsilon_1 \kappa_2 + \varepsilon_2 \kappa_1)
\|\nabla (v_{\tau}^k - v_{\tau}^{k-1})\|_{L^2}^2 + 
\left (\gamma_1 \varepsilon_2 + \gamma_2 \varepsilon_1+
\varepsilon_1\varepsilon_2 \right )
\|v_{\tau}^k - v_{\tau}^{k-1}\|_{L^2}^2. 
\]
\end{lem}
Lemma \ref{Hcontinuity} ensures the uniform boundedness of 
the second moment of $\overline{u}_{\tau}$ and 
$H^1$-norm of $\overline{v}_{\tau}(t)$. 
\begin{lem}\label{2moment}
Assume \eqref{lbc}. Then, 
for any $T>0$ 
and for any $t \in [0, T]$ 
the following 
inequality holds. 
\begin{equation*} 
\begin{split} 
& \int_{\mathbb{R}^d} |x|^2\overline{u}_{\tau}(t)\,dx
+ 
(\varepsilon_1 \kappa_2 + \varepsilon_2 \kappa_1)
\|\nabla \overline{v}_{\tau}(t)\|_{L^2}^2 + 
\left (\gamma_1 \varepsilon_2 + \gamma_2 \varepsilon_1+
\varepsilon_1\varepsilon_2 \right )
\|\overline{v}_{\tau}(t)\|_{L^2}^2  \\
& \hspace{5mm}
\leq 4T(T+1)\mathcal{L}(u_{\tau}^0, v_{\tau}^0, w_{\tau}^0) \\
& \hspace{10mm}
+ 2  
\left (
\int_{\mathbb{R}^d} |x|^2\overline{u}_0(t)\,dx
+ 
(\varepsilon_1 \kappa_2 + \varepsilon_2 \kappa_1)
\|\nabla {v}_0\|_{L^2}^2 + 
\left (\gamma_1 \varepsilon_2 + \gamma_2 \varepsilon_1+
\varepsilon_1\varepsilon_2 \right )
\|{v}_0\|_{L^2}^2
\right ). 
\end{split}
\end{equation*}
\end{lem}
\begin{proof}
Let $z_{\tau}^k:= (u_{\tau}^k, v_{\tau}^k)$, 
$z_*:=(\delta_0, 0)$, 
and define $d_{WH^1}:=\sqrt{\mathcal{W}_2^2+ d_{H^1}^2}$. 
Using the triangle inequality and the Cauchy-Schwarz inequality, 
for any $\ell \in \{1, 2, \cdots, N\}$, we have
\begin{equation*} 
\begin{split} 
d_{WH^1}^2(z_{\tau}^{\ell}, z_*)
&\leq \left (
\sum_{k=1}^{\ell} d_{WH^1}(z_{\tau}^k, z_{\tau}^{k-1}) + d_{WH^1}(z_{\tau}^0, z_*)
\right )^2 \\
& \leq 
\left ( 
\left (\sum_{k=1}^{\ell}\tau \right )^{\frac{1}{2}}
\left (\sum_{k=1}^{\ell} 
\frac{d_{WH^1}^2(z_{\tau}^k, z_{\tau}^{k-1})}{\tau} \right )^{\frac{1}{2}}
+ d_{WH^1}(z_{\tau}^0, z_*)
\right )^2 \\
& \leq 
4T \sum_{k=1}^{N} 
\frac{d_{WH^1}^2(z_{\tau}^k, z_{\tau}^{k-1})}{2\tau} 
+ 2d_{WH^1}^2(z_{\tau}^0, z_*). 
\end{split} 
\end{equation*}
Therefore, the claim substantiated by Lemma \ref{Hcontinuity}.
\end{proof}
The following lemma follows immediately from \eqref{coercivity of L} and Lemma \ref{2moment}. 
\begin{lem}\label{UB}
Assume \eqref{lbc}. 
Then, 
for any $T>0$ 
there exists a constant $C_T$ such that 
for any $t \in [0,T]$, 
\begin{multline*}
\| \overline{u}_{\tau}(t)\|_{L^m}^m +\| \Delta \overline{v}_{\tau}(t) \|_{L^2}^2 \\
+ \frac{\kappa_2(\gamma_1+\varepsilon_1)  + \kappa_1(\varepsilon_2+ \gamma_2) }{2}
\|\nabla \overline{v}_{\tau}(t)\|_{L^2}^2
+ \frac{\gamma_1 \gamma_2+\varepsilon_1\varepsilon_2}{2}
\|\overline{v}_{\tau}(t)\|_{L^2}^2 
<C_T. 
\end{multline*}
\end{lem}

Consequently, it is evident that the first inequality of 
Proposition \ref{uniform bounds} holds, 
except for the estimate of $\|\overline{w}_{\tau}(t)\|_{H^1}$. 
To demonstrate the remaining inequalities, we begin with the proof of the following lemma.

\begin{lem} \label{uL2p}
Assume \eqref{uniform L^2} and \eqref{lbc}. 
Then, for any $T>0$, it holds that
\[
\sup_{\tau \in (0, \tau_*)}\int_0^T
\|\overline{u}_{\tau}(t)\|_{L^2}^2\,dt <\infty. 
\]
\end{lem}

\begin{cor} \label{uL2}
Assume that $m=2-4/d$
and $M<M_*$. 
If $d \geq 6$. 
then, for any $T>0$ it holds that 
\[
\sup_{\tau \in (0, \tau_*)}\int_0^T
\|\overline{u}_{\tau}(t)\|_{L^2}^2\,dt <\infty. 
\]
\end{cor}
\begin{proof}[Proof of Lemma \ref{uL2p}]
It is clear from Lemma \ref{2moment} when $m \geq 2$, 
so consider $1 < m<2$. 
Letting $p=\frac{d}{d-2}$ in \eqref{a priori}, we obtain 
\begin{equation}\label{L2estimate}
\|\nabla (u_{\tau}^k)^m\|_{L^{\frac{d}{d-1}}} 
\leq \frac{\mathcal{W}_2(u_{\tau}^k, u_{\tau}^{k-1})}{\tau}
\|u_{\tau}^k\|_{L^{\frac{d}{d-2}}}^{\frac{1}{2}} 
+ \|u_{\tau}^k\nabla v_{\tau}^k\|_{L^{\frac{d}{d-1}}}. 
\end{equation}
First, we show that the left-hand side can be lower bounded 
by a constant multiple of 
$\|u_{\tau}^k\|_{L^2}^{\frac{4m}{d(2-m)}}$. 
By the Sobolev inequality, we have 
\begin{equation*} 
\begin{split} 
\|u_{\tau}^k\|_{L^{\frac{md}{d-2}}}^m
= \|(u_{\tau}^k)^m\|_{L^{\frac{d}{d-2}}}
\leq C_1
\|\nabla (u_{\tau}^k)^{m}\|_{L^{\frac{d}{d-1}}}. 
\end{split}
\end{equation*}
Since $md/(d-2)\geq 2$, 
by the interpolation inequality,  we obtain 
\begin{equation*} 
\begin{split} 
\|u_{\tau}^k\|_{L^2} 
\leq \|u_{\tau}^k\|_{L^m}^{\frac{4-(2-m)d}{4}}
\|u_{\tau}^k\|_{L^{\frac{md}{d-2}}}^{\frac{(2-m)d}{4}}
\leq 
C_2
\|u_{\tau}^k\|_{L^{\frac{md}{d-2}}}^{\frac{(2-m)d}{4}}, 
\end{split}
\end{equation*}
where $C_2$ is a constant determined by Lemma \ref{UB}. 
Consequently, we obtain 
\begin{equation*} 
\begin{split} 
C_3 \|u_{\tau}^k\|_{L^2}^{\frac{4m}{d(2-m)}} 
\leq \|\nabla (u_{\tau}^k)^m\|_{L^{\frac{d}{d-1}}}. 
\end{split}
\end{equation*}
Next, consider estimating the norm appearing on the right-hand side of equation \eqref{L2estimate} from above using $\|u_{\tau}^k\|_{L^2}$. 
By the interpolation inequality and 
Lemma \ref{UB}, there exists a constant $C_4$ such that 
\begin{equation*} 
\begin{split} 
\|u_{\tau}^k\|_{\frac{d}{d-2}}
\leq \|u_{\tau}^k\|_{L^m}^{\frac{m(d-4)}{d(2-m)}} 
\|u_{\tau}^k\|_{L^2}^{\frac{4m-2(m-1)d}{d(2-m)}} 
\leq C_4\|u_{\tau}^k\|_{L^2}^{\frac{4m-2(m-1)d}{d(2-m)}}. 
\end{split}
\end{equation*}
By H\"older's inequality and 
the Sobolev inequality 
with the Sobolev constant $C_{2,2}$
we have  
\begin{equation} \label{ugradv}
\begin{split} 
\|u_{\tau}^k\nabla v_{\tau}^k\|_{L^{\frac{d}{d-1}}} 
\leq \|u_{\tau}^k\|_{L^2}\|\nabla v_{\tau}^k\|_{L^{\frac{2d}{d-2}}} 
\leq C_{2,2}\|u_{\tau}^k\|_{L^2}\|\Delta v_{\tau}^k\|_{L^2} 
\leq C_5\|u_{\tau}^k\|_{L^2}, 
\end{split}
\end{equation}
where we have used Lemma \ref{UB} to determine the positive constant $C_5$. 
Hence we have 
\begin{equation} \label{gradum}
\begin{split} 
C_3 \|u_{\tau}^k\|_{L^2}^{\frac{4m}{d(2-m)}} 
\leq \|\nabla (u_{\tau}^k)^m\|_{L^{\frac{d}{d-1}}} 
\leq 
C_4\frac{\mathcal{W}_2(u_{\tau}^k, u_{\tau}^{k-1})}{\tau}
\|u_{\tau}^k\|_{L^2}^{\frac{4m-2(m-1)d}{2d(2-m)}}
+ C_5\|u_{\tau}^k\|_{L^2}, 
\end{split}
\end{equation}
Since 
\begin{equation} \label{exponentsrelation}
\begin{split} 
m \geq 2-\frac{4}{d} \geq \frac{2d}{d+4}
\iff 
\frac{4m}{d(2-m)} >1, 
\quad \frac{4m-2(m-1)d}{2d(2-m)}<1 
\iff m<\frac{d}{2}, 
\end{split}
\end{equation}
we obtain 
\begin{equation*} 
\begin{split} 
C_3 \|u_{\tau}^k\|_{L^2}^{\frac{4m}{d(2-m)}-\frac{4m-2(m-1)d}{2d(2-m)}} 
\leq 
C_4\frac{\mathcal{W}_2(u_{\tau}^k, u_{\tau}^{k-1})}{\tau}
+ C_5\|u_{\tau}^k\|_{L^2}^{1-\frac{4m-2(m-1)d}{2d(2-m)}}. 
\end{split}
\end{equation*}
After squaring both sides and rearranging the exponents, we have 
\begin{equation*} 
\begin{split} 
C_3^2 \|u_{\tau}^k\|_{L^2}^{\frac{4m+2(m-1)d}{d(2-m)}}
\leq 2C_4^2\frac{\mathcal{W}_2^2(u_{\tau}^k, u_{\tau}^{k-1})}{\tau^2}
+2 C_5^2 \|u_{\tau}^k\|_{L^2}^{\frac{2(d-2m)}{d(2-m)}}. 
\end{split}
\end{equation*}
For $p=\frac{2m+(m-1)d}{d-2m}$ and its conjugate exponent 
$p'=\frac{2m+(m-1)}{m(d+4) -2d}$, 
by Young's inequality we have 
\begin{equation*} 
\begin{split} 
2 C_5^2 \|u_{\tau}^k\|_{L^2}^{\frac{2(d-2m)}{d(2-m)}}
& = 2C_5^2 \left (\frac{C_3^2}{2} \right )^{-\frac{1}{p}} 
\left (\frac{C_3^2}{2} \right )^{\frac{1}{p}}
 \|u_{\tau}^k\|_{L^2}^{\frac{2(d-2m)}{d(2-m)}}\\
& \leq \frac{1}{p'}\left ( 2C_5^2 \right )^{p'}
\left (\frac{C_3^2}{2} \right )^{-\frac{p'}{p}}
+ \frac{C_3^2}{2} 
\|u_{\tau}^k\|_{L^2}^{p\frac{2(d-2m)}{d(2-m)}} \\
&= C_6 + \frac{C_3^2}{2}
\|u_{\tau}^k\|_{L^2}^{\frac{4m-2(m-1)d}{2d(2-m)}}. 
\end{split}
\end{equation*}
Thus, we obtain 
\begin{equation} \label{uL2PW}
\begin{split} 
\frac{C_3^2}{2} \tau \|u_{\tau}^k\|_{L^2}^{\frac{4m+2(m-1)d}{d(2-m)}}
\leq 2C_4^2\frac{\mathcal{W}_2^2(u_{\tau}^k, u_{\tau}^{k-1})}{\tau}
+C_6\tau. 
\end{split}
\end{equation}
\[
\lceil T/\tau \rceil -1 <\frac{T}{\tau} \leq \lceil T/\tau \rceil 
\]
Using the celling function $\lceil \, \cdot \, \rceil$ and  
summing over $k=1$ to $\lceil T/\tau \rceil$, we have 
\begin{equation*} 
\begin{split} 
\frac{C_3^2}{2}
\int_0^T
\|\overline{u}_{\tau}(t)\|_{L^2}^{\frac{4m+2(m-1)d}{d(2-m)}}\,dt
& \leq 
\frac{C_3^2}{2}
\sum_{k=1}^{\lceil T/\tau \rceil}
\tau \|u_{\tau}^k\|_{L^2}^{\frac{4m+2(m-1)d}{d(2-m)}} \\
& <
2C_4^2\sum_{k=1}^{\lceil T/\tau \rceil}
\frac{\mathcal{W}_2^2(u_{\tau}^k, u_{\tau}^{k-1})}{\tau}
+ C_6(T+\tau_*). 
\end{split}
\end{equation*}
The right-hand side is bounded independently of $\tau$ 
as per Lemma \ref{Hcontinuity}. 
\end{proof}

\begin{cor}\label{umW1p}
Assume \eqref{uniform L^2} and \eqref{lbc}. 
Then, for any $T>0$ 
the following 
inequality holds. 
\begin{equation*} 
\begin{split} 
\sup_{\tau \in (0, \tau_*)}\int_0^T
\|\nabla \overline{u}_{\tau}^{m}(t)\|_{L^{\frac{d}{d-1}}}^{p_*}
\,dt
<\infty, 
\end{split}
\end{equation*}
where $p_*$ is specified in \eqref{p_*}
\end{cor}

\begin{proof} 
The proof is divided into two cases: 
$1<m<2$ and $m \geq 2$.

\vspace{2mm}

\noindent
\textbf{(i) the case \bm{$1<m<2$}}

\vspace{2mm}

Let $p= \frac{4m-2(m-1)d}{d(2-m)}$. 
According to \eqref{exponentsrelation}, 
we have $p<2$. From 
\eqref{gradum}, we derive
\begin{equation*} 
\begin{split} 
\| \nabla (u_{\tau}^k)^m \|_{L^\frac{d}{d-1}}
& \leq C_4
\frac{\mathcal{W}_2(u_{\tau}^k, u_{\tau}^{k-1})}{\tau}
\|u_{\tau}^k\|_{L^2}^{\frac{4m-2(m-1)d}{2d(2-m)}} 
+ C_5 \|u_{\tau}^k\|_{L^2} \\
& = C_4
\frac{\mathcal{W}_2(u_{\tau}^k, u_{\tau}^{k-1})}{\tau}
\|u_{\tau}^k\|_{L^2}^{\frac{p}{2}} 
 + C_5 \|u_{\tau}^k\|_{L^2} \\
 & \leq 
 \left (\frac{\mathcal{W}_2(u_{\tau}^k, u_{\tau}^{k-1})}{\tau} \right )^{\frac{2}{2-p}}
+
(C_4^{\frac{2}{p}}+C_5)\|u_{\tau}^k\|_{L^2}. 
\end{split}
\end{equation*}
Raising both sides to the power of $2-p=\frac{2(d-2m)}{d(2-m)}>1$, we get
\begin{equation*} 
\| \nabla (u_{\tau}^k)^m \|_{L^\frac{d}{d-1}}^{2-p}
 \leq 2^{1-p}\left [ 
 \left (\frac{\mathcal{W}_2(u_{\tau}^k, u_{\tau}^{k-1})}{\tau} \right )^{2}
 + (C_4^{\frac{2}{p}}+C_5)^{2-p}\|u_{\tau}^k\|_{L^2}^{2-p} \right ]. 
\end{equation*}
By considering 
$\frac{4m+2(m-1)d}{d(2-m)}
\geq 2-p$ and Young's inequality 
$\|u_{\tau}^k\|_{L^2}^{2-p} 
\leq \|u_{\tau}^k\|_{L^2}^{\frac{4m+2(m-1)d}{d(2-m)}}
+ 1$, 
\begin{multline*}
\int_0^T
\|\nabla (\overline{u}_{\tau}(t))^m \|_{L^\frac{d}{d-1}}^{\frac{2(d-2m)}{d(2-m)}}\,dt
\leq  
\sum_{k=1}^{\lceil T/\tau \rceil}
\tau \|\nabla (u_{\tau}^k)^m \|_{L^\frac{d}{d-1}}^{\frac{2(d-2m)}{d(2-m)}}
\\
 \leq  2^{1-p}
 \sum_{k=1}^{\lceil T/\tau \rceil}
 \frac{\mathcal{W}_2^2(u_{\tau}^k, u_{\tau}^{k-1})}{\tau}
 + 2^{1-p}(C_4^{\frac{2}{p}}+C_5)^{2-p}
 \sum_{k=1}^{\lceil T/\tau \rceil}\tau 
 \|u_{\tau}^k\|_{L^2}^{\frac{4m+2(m-1)d}{d(2-m)}} \\
 + 2^{1-p}(C_4^{\frac{2}{p}}+C_5)^{2-p}(T+\tau_*). 
\end{multline*}
According to Lemmas \ref{Hcontinuity} and \ref{uL2p}, 
the right-hand side is bounded independently of $\tau$. 

\vspace{2mm}

\noindent
\textbf{(ii) the case \bm{$m \geq 2$}}

\vspace{2mm}

The interpolation inequality yields that 
\begin{equation*} 
\|u\|_{L^{\frac{d}{d-2}}}^{\frac{1}{2}}
 \leq \|u\|_{L^1}^{\frac{m(d-2)-d}{2d(m-1)}} 
\|u\|_{L^m}^{\frac{2m}{2d(m-1)}} \leq C_7. 
\end{equation*}
Combining this with \eqref{L2estimate} and \eqref{ugradv}, 
we obtain 
\begin{equation}
\|\nabla (u_{\tau}^k)^m\|_{L^{\frac{d}{d-1}}} 
\leq C_7\frac{\mathcal{W}_2(u_{\tau}^k, u_{\tau}^{k-1})}{\tau}
+ C_5\|u_{\tau}^k\|_{L^2}. 
\end{equation}
Squaring both sides and multiplying by $\tau$, and then summing 
from $k=1$ to ${\lceil T/\tau \rceil}$ yields the result of the corollary. 
\end{proof}

\begin{lem}\label{UB4w}
Assume \eqref{uniform L^2} and \eqref{lbc}. 
Then, 
there exists a positive constant $C$ such that 
for any $\ell \in \mathbb{N}$, 
\begin{equation*} 
\begin{split} 
\frac{\varepsilon_2 \kappa_2}{2}\|\nabla w_{\tau}^{\ell}\|_{L^2}^2+ 
\frac{\varepsilon_2 \gamma_2}{2}\|w_{\tau}^{\ell}\|_{L^2}^2+ 
\frac{\varepsilon_2^2}{4\tau} \sum_{k=1}^{\ell}
\|w_{\tau}^{k}-w_{\tau}^{k-1}\|_{L^2}^2 
\leq 
C. 
\end{split}
\end{equation*}
\end{lem}
\begin{proof}
By the definitions of $w_{\tau}^k$ and applying H\"older's and Young's inequalities, 
we derive:
\begin{equation*} 
\begin{split} 
\mathcal{G}(w_{\tau}^k) 
-\mathcal{G}(w_{\tau}^{k-1}) 
+ \frac{\varepsilon_2}{2\tau}\|w_{\tau}^{k}-w_{\tau}^{k-1}\|_{L^2}^2 
& \leq 
 \int u_{\tau}^{k-1}(w_{\tau}^k - w_{\tau}^{k-1})\,dx 
  \leq 
 \|u_{\tau}^{k-1} \|_{L^2}\|w_{\tau}^k - w_{\tau}^{k-1}\|_{L^2} \\
 & \leq 
 \frac{\tau}{\varepsilon_2} \|u_{\tau}^{k-1} \|_{L^2}^2 
 + \frac{\varepsilon_2}{4\tau}\|w_{\tau}^k - w_{\tau}^{k-1}\|_{L^2}^2. 
\end{split}
\end{equation*}
Summing from $k=1$ to $\ell$ gives
\begin{equation*} 
\begin{split} 
\mathcal{G}(w_{\tau}^{\ell})+ 
\frac{\varepsilon_2}{4\tau} \sum_{k=1}^{\ell}
\|w_{\tau}^{k}-w_{\tau}^{k-1}\|_{L^2}^2 
& \leq 
\mathcal{G}(w_{\tau}^0)
+\frac{1}{\varepsilon_2}\sum_{k=1}^{\ell}  \tau \|u_{\tau}^{k-1} \|_{L^2}^2  \\
& \leq 
\mathcal{G}(w_0)
+\frac{1}{\varepsilon_2}\sum_{k=1}^{N}  \tau \|u_{\tau}^{k-1} \|_{L^2}^2 \\
& \leq 
\mathcal{G}(w_0)
+\frac{1}{\varepsilon_2} 
\left (
\tau \|u_{\tau}^0\|_{L^2}
+ \int_0^T \|\overline{u}_{\tau}(t)\|_{L^2}^2 \,dt
\right ). 
\end{split}
\end{equation*}
For $\tau \in (0, \tau_*)$, 
the term $\tau \|u_{\tau}^0\|_{L^2}^2$ on the right-hand side can be estimated 
using Young's convolution inequality as follows: 
\begin{equation} \label{initialL2}
\tau \|u_{\tau}^0\|_{L^2}^2 
= \tau \|\rho_{\tau}*u_0 \|_{L^2}^2
\leq \tau \|\rho_{\tau}\|_{L^2}^2 \|u_0\|_{L^1}^2 = \sqrt{\tau} \|u_0\|_{L^1}^2
\leq \sqrt{\tau_*} \|u_0\|_{L^1}^2.   
\end{equation}
Thus, it is bounded. 
The subsequent term is also bounded by Lemma \ref{uL2p}. 
\end{proof}

From Lemmas \ref{Hcontinuity} and \ref{UB4w}, 
it can be seen that the third inequality in  
Proposition \ref{uniform bounds} holds. 
Furthermore, as demonstrated by the following lemma, 
even when $\gamma_2 =0$, 
Lemma \ref{UB4w} asserts the uniform boundedness of 
$\|\overline{w}_{\tau}(t)\|_{L^2}$.

\begin{lem}\label{wL2}
Assume \eqref{uniform L^2} and \eqref{lbc}. 
Then, for any $T>0$ 
there exists a positive constant $C_T$ such that for any $t \in [0, T]$, 
\[ \varepsilon_2 \|\overline{w}_{\tau}(t)\|_{L^2}^2 <C_T  \]
\end{lem}

\begin{proof}
Assume that $\ell \in \mathbb{N}$ satisfies 
$\tau \ell \leq T$. 
The, the triangle inequality and the Cauchy-Schwarz inequality yield that 
\begin{equation*} 
\begin{split} 
\varepsilon_2 
\|w_{\tau}^{\ell}\|_{L^2}^2 
&\leq \varepsilon_2
\left ( \sum_{k=1}^{\ell}\|w_{\tau}^{k}- w_{\tau}^{k-1}\|_{L^2}
 +\|w_0\|_{L^2} \right )^2\\ 
& \leq 
\varepsilon_2 \left (
\left (\sum_{k =1}^{\ell} \tau \right )^{\frac{1}{2}} 
\left ( 
\sum_{k =1}^{\ell} \frac{\|w_{\tau}^{k}- w_{\tau}^{k-1}\|_{L^2}^2}{\tau} 
\right )^{\frac{1}{2}}
 + \varepsilon_2 \|w_0\|_{L^2}^2 \right )^2 \\
& \leq 
2\varepsilon_2T
\left ( 
\sum_{k =1}^N \frac{\|w_{\tau}^{k}- w_{\tau}^{k-1}\|_{L^2}^2}{\tau} 
\right ) + 2\varepsilon_2\|w_0\|_{L^2}^2.  
\end{split}
\end{equation*}
The right-hand side is uniformly bounded 
by Lemma \ref{UB4w}. 
\end{proof}

\begin{lem}\label{w22}
Assume \eqref{uniform L^2} and \eqref{lbc}. 
Then, for any $T>0$ it holds that 
\begin{equation*} 
\begin{split} 
\sup_{\tau \in (0, \tau_*)}\int_0^T \|\overline{w}_{\tau}(t)\|_{W^{2,2}}^2\,dt
<\infty. 
\end{split}
\end{equation*}
\end{lem}
\begin{proof}
From the Euler-Lagrange equation \eqref{EL}, we derive 
\begin{equation*} 
\begin{split} 
\kappa_2 \Delta w_{\tau}^k
 = \varepsilon_2\partial_t w_{\tau}^k -\gamma_2 w_{\tau}^k + u_{\tau}^{k-1}. 
\end{split}
\end{equation*}
Applying elliptic estimates yields
\begin{equation*} 
\begin{split} 
\sum_{k=1}^{\lceil T/\tau \rceil}
\tau \|\kappa_2 w_{\tau}^k\|_{W^{2,2}}^2
&  =C 
 \sum_{k=1}^{\lceil T/\tau \rceil}
 \tau 
 \| \varepsilon_2\partial_t w_{\tau}^k -\gamma_2 w_{\tau}^k + u_{\tau}^{k-1}\|_{L^2}^2 \\
 & \leq 
 3C \sum_{k=1}^{\lceil T/\tau \rceil}\tau\| \varepsilon_2\partial_t w_{\tau}^k\|_{L^2}^2
 + 3C \sum_{k=1}^{\lceil T/\tau \rceil} \tau \|\gamma_2 w_{\tau}^k\|_{L^2}^2 
 + 3C \sum_{k=1}^{\lceil T/\tau \rceil}\tau \|u_{\tau}^{k-1}\|_{L^2}^2, 
\end{split}
\end{equation*}
for some positive constant $C$. 
The right-hand side is uniformly bounded by Lemmas \ref{uL2p} and \ref{UB4w}, 
and \eqref{initialL2}.
\end{proof}

\begin{lem}
Assume \eqref{uniform L^2} and \eqref{lbc}. 
Then, it holds that 
\begin{equation*} 
\begin{split} 
\sup_{\tau \in (0, \tau_*)} \int_0^T \|v(t)\|_{W^{3,2}}^2\,dt
\leq C
\end{split}
\end{equation*}
\end{lem}
\begin{proof}
From the Euler-Lagrange equations \eqref{EL}, we derive
\begin{equation*} 
\begin{split} 
\kappa_1 \Delta v_{\tau}^k
 = \varepsilon_1\partial_t v_{\tau}^k -\gamma_1 v_{\tau}^k + w_{\tau}^k. 
\end{split}
\end{equation*}
Since the right-hand side belongs to $H^1(\mathbb{R}^d)$, elliptic estimates give
\[
\|\kappa_1 v_{\tau}^k\|_{W^{3,2}} \leq C 
\|\varepsilon_1 \partial_t v_{\tau}^k -\gamma_1 v_{\tau}^k + w_{\tau}^k\|_{H^1}
\]
for some positive constant $C$. Consequently, we have 
\begin{equation*} 
\begin{split} 
\sum_{k=1}^{\lceil T/\tau \rceil}\tau \|\kappa_1 v_{\tau}^k\|_{W^{3,2}}^2 \leq 3C 
\sum_{k=1}^{\lceil T/\tau \rceil}\tau \|\varepsilon_1 \partial_t v_{\tau}^k \|_{H^1}^2 
+3C\sum_{k=1}^{\lceil T/\tau \rceil}\tau \|\gamma_1v_{\tau}^k\|_{H^1}^2 
+ 3\sum_{k=1}^{\lceil T/\tau \rceil}\tau \|w_{\tau}^k\|_{H^1}^2. 
\end{split}
\end{equation*}
The right-hand side is uniformly bounded by 
Lemmas \ref{Hcontinuity} and \ref{w22}
\end{proof}
Consequently, all parts of Proposition \ref{uniform bounds} have been proven.

\section{Convergence}
In this section, we demonstrate that 
there exists a sequence of time steps for which 
piecewise constant interpolation of discrete solutions 
converges to the solution of problem \eqref{P}.

\begin{prop}\label{cpt}
Assume \eqref{uniform L^2} and \eqref{lbc}. 
Then, 
there exist a sequence $\{ \tau_n \} $ with $\tau_n \to 0\ (n \to \infty)$ 
and a triple $(u,v,w)$ of functions 
such that for any $t >0$, 
\begin{equation} \label{candidate}
\begin{cases}
\overline{u}_{\tau_n}(t),\ \underline{u}_{\tau_n}(t) 
\rightharpoonup u(t)
&\text{ weakly in }(L^1\cap L^m)(\mathbb{R}^d), \\
\overline{v}_{\tau_n}(t),\ \underline{v}_{\tau_n}(t)
 \rightharpoonup v(t) 
& \text{ weakly in }W^{2,2}(\mathbb{R}^d), \\
\overline{w}_{\tau_n}(t),\ \underline{w}_{\tau_n}(t)
 \rightharpoonup w(t) 
& \text{ weakly in }H^{1}(\mathbb{R}^d). 
\end{cases}
\end{equation}
\end{prop}

\begin{proof}[Proof of Proposition \ref{cpt}]

The proposition's proof is already presented under abstract assumptions 
in \cite[Proposition 3.3.1]{a-g-s}. 
For clarity and relevance, we provide a tailored proof for this paper. 
Define $\overline{z}_{\tau}(t):= 
(\overline{u}_{\tau}(t), \overline{v}_{\tau}(t), \overline{w}_{\tau}(t))$ and 
the distance $d$ as follows 
\begin{equation*} 
\begin{split} 
d(z_1, z_2):= 
\sqrt{ \mathcal{W}_2^2 (u_1, u_2) 
+ \|v_2-v_1\|_{H^1}^2 
+ \|w_2-w_1\|_{L^2}^2 
}
\end{split}
\end{equation*}
for $z_1:=(u_1, v_1, w_1), z_2:=(u_2, v_2, w_2) \in 
M\mathscr{P}_2(\mathbb{R}^d) \times H^1(\mathbb{R}^d) \times L^2(\mathbb{R}^d)$. 
Firstly, we establish that  
\begin{equation} \label{hh}
d(\overline{z}_{\tau}(t), \overline{z}_{\tau}(s)) 
\leq C \sqrt{|t-s|+\tau}
\end{equation}
for some positive constant $C$, 
Without loss of generality, we can assume $s<t$. 
By the triangle inequality, the Cauchy-Schwarz inequality, 
and the third inequality in Proposition \ref{uniform bounds}, we have  
\begin{equation*}
\begin{split} 
d(\overline{z}_{\tau}(t), \overline{z}_{\tau}(s)) 
& \leq \sum_{k=\lceil s/\tau \rceil +1}^{\lceil t/\tau \rceil}
d(z_{\tau}^k,z_{\tau}^{k-1})
\leq \sqrt{t-s+\tau} 
\left(
\sum_{\lceil s/\tau \rceil +1}
^{\lceil t/\tau \rceil}
\frac{d^2(z_{\tau}^k,z_{\tau}^{k-1})}{\tau}
\right)^{1/2} \\ 
& \leq C\sqrt{
t-s+\tau}, 
\end{split}
\end{equation*}
where $\lceil t/\tau \rceil$ denotes 
the smallest integer not less than $T/\tau$.

Next, we use the diagonal argument to show that for any 
$t \in (0,\infty) \cap \mathbb{Q}$, 
\eqref{candidate} holds. 
First, we take an arbitrary 
$t_1 \in (0,\infty) \cap \mathbb{Q}$. 
From the first inequality in Proposition \ref{uniform bounds}, 
there exist a sequence $\{\tau_n\}=\{\tau_n^1\}$ and $(u(t_1),v(t_1),w(t_1))$ 
that satisfy 
\eqref{candidate} for $t=t_1$. 
Similarly, there exist a subsequence $\{\tau_n^2\}$ of $\{\tau_n^1\}$ and 
functions $(u(t_1),v(t_1),w(t_1))$ 
that satisfy \eqref{candidate} for $t=t_1, t_2$. 
By repeating this argument, we can ensure that 
for any $t \in (0,\infty)\cap \mathbb{Q}$, 
there exist a sequence $\{\tau_n\}$ with $\tau_n:=\tau_n^n$ 
and functions $(u, v, w)$ that satisfy \eqref{candidate}. 

Using inequality \eqref{hh}, weak convergence \eqref{candidate}, 
the weak lower semicontinuity of the norm in Hilbert spaces, 
and the weak lower semicontinuity of 
the Wasserstein distance \cite[Lemma 7.1.4]{a-g-s}, 
for $z(t):=(u(t), v(t), z(t))$ and $t, s \in (0,\infty)\cap \mathbb{Q}$, we have
\begin{equation*} \label{density}
d(z(t), z(s))\leq 
\liminf_{n \to \infty}d(z_{\tau_n}(t), z_{\tau_n}(s)) 
\leq C\sqrt{|t-s|}. 
\end{equation*}
Since $M\mathscr{P}_2(\mathbb{R}^d) \times 
W^{2,2}(\mathbb{R}^d) \times H^1(\mathbb{R}^d)$ 
is complete with respect to the distance $d$, 
the function $z(t):=(u(t), v(t), w(t))$, which is defined on 
$(0,\infty) \cap \mathbb{Q}$, 
can be continuously extended to a function on $(0,\infty)$. 

To establish \eqref{candidate} for 
$t\in (0,\infty) \setminus\mathbb{Q}$, we use 
the density of rational points and continuity of $z(t)$. 
For any $\varepsilon>0$, 
there exist $\delta >0$ and $t_* \in (0,\infty)\cap \mathbb{Q}$ 
such that 
if $|t-t_*|<\delta$, then $d(z(t), z(t_*)) < \varepsilon$. 
From the first inequality in Proposition \ref{uniform bounds}, 
since $\{z_{\tau_n}(t)\}$ is a bounded sequence in 
$(L^1\cap L^m)(\mathbb{R}^d) 
\times W^{2,2}(\mathbb{R}^d)\times H^1(\mathbb{R}^d)$, 
there exists a subsequence $\{\tau_n'\}$ of $\{\tau_n\}$ for which 
convergence similar to \eqref{candidate} can be asserted. 
Let the weak limit be $\Tilde{z}(t)= (\Tilde{u}(t), \Tilde{v}(t), \Tilde{w}(t))$. 
Then, by the weak lower continuity of $d$ and 
\eqref{hh}, we have 
\begin{equation*} 
\begin{split} 
d(\Tilde{z}(t), z(t_*))
\leq \liminf_{n \to \infty}
d(z_{\tau_n'}(t), z_{\tau_n'}(t_*))\leq C\sqrt{|t-t_*|}
< C\sqrt{\varepsilon}. 
\end{split}
\end{equation*}
Thus, 
\begin{equation*} 
\begin{split} 
d(\Tilde{z}(t), z(t))
\leq d(\Tilde{z}(t), z(t_*))
+ d(z(t_*), z(t)) < C\sqrt{\varepsilon} +\varepsilon. 
\end{split}
\end{equation*}
Since $\varepsilon>0$ is arbitrary, 
it follows that $\Tilde{z}(t)=z(t)$. 
This holds for any weakly convergent subsequence 
$\{\overline{z}_{\tau_n'}(t)\}$, 
implying that result holds without taking subsequences. 
Hence, 
\eqref{candidate} is also valid $t \in (0,\infty)\setminus \mathbb{Q}$. 
\end{proof}

Let us consider any $T > 0$ and define the following. 
\begin{equation*} 
\begin{cases}\vspace{2mm}
\displaystyle 
\mathcal{B}_{\tau_n}(t)
:=
\|\overline{u}_{\tau_n}(t)\|_{L^2}
+\|\nabla \overline{u}_{\tau_n}^{m}(t)\|_{L^{\frac{d}{d-1}}}
+ \|\overline{v}_{\tau_n}(t)\|_{W^{3,2}}
+\|\overline{w}_{\tau_n}(t)\|_{W^{2,2}}, \\ \vspace{2mm}
 q_*:=\min \{2, p_*\}>1, \text{ where $p_*$  is specified in \eqref{p_*}}, \\
\displaystyle 
S_{\tau_n}(K):=\{t \in (0,T] \ | \ \mathcal{B}_{\tau_n}^{q_*}(t)> K\}. 
\end{cases}
\end{equation*}
Then, according to the second inequality in Proposition \ref{uniform bounds},  
$\|\mathcal{B}_{\tau_n}\|_{L^{q_*}(0,T)}$ is uniformly bounded. 
Consequently, 
it holds that 
\begin{equation*} 
\begin{split} 
\mathscr{L}^1\bigl ( S_{\tau_n}(K) \bigr )
\leq 
\int_{S_{\tau_n}(K)}
\frac{\mathcal{B}_{\tau_n}^{q_*}(t)}{K}\,dt
\leq \frac{1}{K}
\int_0^T \mathcal{B}_{\tau_n}^{q_*}(t)\,dt
\to 0 \quad (K \to \infty). 
\end{split}
\end{equation*}

\begin{lem}[pointwise-convergence]\label{pointwise}
Let $(\overline{u}_{\tau_n},\overline{v}_{\tau_n})$ 
and $(u, v)$ be that in Proposition \ref{cpt}. 
Assume
that $\sup_n \mathcal{B}_{\tau_n}(t_0)<\infty$. Then, 
\begin{equation*} 
\begin{cases}\vspace{2mm}
\displaystyle 
\lim_{n \to \infty}
\int_{\mathbb{R}^d} \langle \nabla \overline{u}_{\tau_n}^m(t_0) 
-\overline{u}_{\tau_n}(t_0)\nabla \overline{v}_{\tau_n}(t_0), \bm{\xi} \rangle\,dx
=\int_{\mathbb{R}^d} \langle \nabla u^m(t_0) -u(t_0)\nabla v(t_0), \bm{\xi} \rangle\,dx, \\
\vspace{2mm}
\displaystyle 
\lim_{n \to \infty} 
\int_{\mathbb{R}^d}( \kappa_1\Delta  \overline{v}_{\tau_n}(t_0)
- \gamma_1 \overline{v}_{\tau_n}(t_0) +  \overline{w}_{\tau_n}(t_0))\eta\, dx
= 
\int_{\mathbb{R}^d}( \kappa_1\Delta  v(t_0)
- \gamma_1 v(t_0) +  w(t_0))\eta\, dx, \\ \vspace{2mm}
\displaystyle 
\lim_{n \to \infty} 
\int_{\mathbb{R}^d}( \kappa_2\Delta  \overline{w}_{\tau_n}(t_0)
- \gamma_2 \overline{w}_{\tau_n}(t_0) +  \underline{u}_{\tau_n}(t_0))\eta\, dx
= 
\int_{\mathbb{R}^d}( \kappa_1\Delta  w(t_0)
- \gamma_1 w(t_0) +  u(t_0))\eta\, dx, \\ \vspace{2mm}
\displaystyle 
\lim_{n \to \infty} 
\int_{\mathbb{R}^d}
\langle \nabla ( \kappa_1\Delta  \overline{v}_{\tau_n}
- \gamma_1 \overline{v}_{\tau_n} +  \overline{w}_{\tau_n})(t_0), \bm{\xi}\rangle 
\, dx
= 
\int_{\mathbb{R}^d}
\langle \nabla ( \kappa_1\Delta  v
- \gamma_1 v +  w)(t_0), \bm{\xi} \rangle \, dx, \\
\displaystyle 
\lim_{n \to \infty}
\int_{\mathbb{R}^d}  \underline{u}_{\tau_n}(t_0)\eta \,dx
=\int_{\mathbb{R}^d} u(t_0) \eta\,dx
\end{cases}
\end{equation*}
holds for all $\bm{\xi} \in C^{\infty}_c(\mathbb{R}^d;\mathbb{R}^d)$ 
and for all $\eta \in C^{\infty}_c(\mathbb{R}^d)$. 
\end{lem}

\begin{proof}
By assumption 
$\sup_n \mathcal{B}_{\tau_n}(t_0)<\infty$, 
$\{ u_{\tau_n}^m(t_0) \}$ forms 
a bounded sequence in $W^{1,\frac{d}{d-1}}(\mathbb{R}^d)$. 
Considering the Rellich-Kondrachov theorem, 
there exists a subsequence $\{ \tau_n' \}\subset \{\tau_n\}$ such that 
\begin{equation} \label{limit in W1p}
\begin{split} 
\begin{cases}\vspace{2mm}
u_{\tau_n'}^m(t_0) \to u_* \text{ in strongly }L^1_{\loc}(\mathbb{R}^d), \\
\nabla u_{\tau_n'}^m(t_0) 
\rightharpoonup \nabla u_* \text{ weakly in }L^{\frac{d}{d-1}}(\mathbb{R}^d). 
\end{cases}
\end{split}
\end{equation}
Moreover, for any bounded subset $\Omega \subset \mathbb{R}^d$, 
there exists another subsequence $\{ \tau_n'' \} \subset \{\tau_n'\}$ 
such that 
$u_{\tau_n''}(x, t_0) \to u_*^{\frac{1}{m}}(x, t_0)$ a.e. $x \in \Omega$. 
Combining this with \eqref{limit in W1p}, 
we achieve convergence 
$u_{\tau_n''}(t_0) \to u_*^{\frac{1}{m}}(t_0)$ in $L^m(\Omega)$
(\cite[Proposition 1.33]{a-f-p}). 
Furthermore, since 
$u_{\tau_n''}(t_0) \rightharpoonup u(t_0)$ weakly in $L^m(\mathbb{R}^d)$, 
it follows that $u_*=u^m(t_0)$. 
As this holds for any chosen subsequence, 
the results are valid without extracting a particular subsequence. 
Therefore, we obtain 
\begin{equation*} 
\begin{split} 
\lim_{n \to \infty}
\int_{\mathbb{R}^d} \langle \nabla \overline{u}_{\tau_n}^m(t_0), \bm{\xi} \rangle\,dx
=\int_{\mathbb{R}^d} \langle \nabla u^m(t_0), \bm{\xi} \rangle\,dx. 
\end{split}
\end{equation*}
Given that $\sup_{n \in \mathbb{N}}\|\overline{u}_{\tau_n}(t_0)\|_{L^2}^2< \infty$ 
and 
$\overline{u}_{\tau_n}(t_0) 
\rightharpoonup u(t_0) \text{ weakly in }L^m(\mathbb{R}^d)$,  
it follows that 
$\overline{u}_{\tau_n}(t_0) 
\rightharpoonup u(t_0) \text{ weakly in }L^2(\mathbb{R}^d)$ as well. 
Therefore, we have 
\begin{equation*} 
\begin{split} 
& 
\left |
\int_{\mathbb{R}^d} 
\langle \overline{u}_{\tau_n}(t_0)
\nabla \overline{v}_{\tau_n}(t_0), \bm{\xi}\rangle \,dx
-\int_{\mathbb{R}^d}
\langle u(t_0)\nabla v(t_0), 
\bm{\xi}
\rangle \,dx \right |
\\
& = 
\left | 
\int_{\mathbb{R}^d}
\langle \overline{u}_{\tau_n}(t_0)
\left ( \nabla \overline{v}_{\tau_n}(t_0)
-\nabla v(t_0) \right ), \bm{\xi}\rangle \,dx \right |
+ 
\left | 
\int_{\mathbb{R}^d}
\langle \left ( \overline{u}_{\tau_n}(t_0)
-u(t_0) \right ) \nabla v(t_0), \bm{\xi} \rangle \,dx \right |
 \\
& \leq 
\|\overline{u}_{\tau_n}(t_0)\|_{L^2}\|\bm{\xi}\|_{L^{\infty}}
\int_{\supp{\bm{\xi}}}|\nabla \overline{v}_{\tau_n}(t_0)-\nabla v(t_0)|^2\,dx
+ \left |
\int_{\mathbb{R}^d}
\left ( \overline{u}_{\tau_n}(t_0)
-u(t_0) \right ) 
\langle \nabla v(t_0), \bm{\xi} \rangle \,dx \right |, 
\end{split}
\end{equation*}
where $\supp{\bm{\xi}}$ denotes the support of $\bm{\xi}$. 
The first factor of the fist term on the right-hand side is uniformly bounded 
with respect to $n$, and the third factor converges to $0$ as $n \to \infty$ by 
Proposition \ref{cpt} and Rellich's theorem. 
The second term on the right-hand side converges to $0$ 
since 
$\overline{u}_{\tau_n}(t_0)$ weakly converges to 
$\rightharpoonup u(t_0)$ in $L^2(\mathbb{R}^d)$. 
Thus, the first assertion holds. 
From Proposition \ref{cpt} and its corollary, 
the second and third assertions follow directly. 
The fourth assertion can be proven similarly to the first assertion.  
\end{proof}

\begin{prop}[$L^1$-convergence]\label{L^1}
Assuming conditions \eqref{uniform L^2} and \eqref{lbc}, 
and referring to the triplet 
$(\overline{u}_{\tau_n},\overline{v}_{\tau_n}, \overline{w}_{\tau_n})$ 
as well as $(u, v, w)$ from Lemma \ref{cpt}, 
the following convergence results hold for any test functions 
$\bm{\xi} \in C^{\infty}_c(\mathbb{R}^d\times (0,\infty);\mathbb{R}^d)$ 
and for any $\eta \in C^{\infty}_c(\mathbb{R}^d\times (0,\infty))$, and 
for any $T>0$, 
\begin{equation*} 
\begin{cases}\vspace{2mm}
\displaystyle 
\lim_{n \to \infty}
\stint{0}{T}{\mathbb{R}^d}
 \langle \nabla \overline{u}_{\tau_n}^m
-\overline{u}_{\tau_n}\nabla \overline{v}_{\tau_n}, \bm{\xi} \rangle\,dxdt
=\stint{0}{T}{\mathbb{R}^d} 
\langle \nabla u^m -u\nabla v, \bm{\xi} \rangle\,dxdt, \\
\vspace{2mm}
\displaystyle 
\lim_{n \to \infty} 
\stint{0}{T}{\mathbb{R}^d}( \kappa_1\Delta  \overline{v}_{\tau_n}
- \gamma_1 \overline{v}_{\tau_n} +  \overline{w}_{\tau_n})\eta\, dxdt
= 
\stint{0}{T}{\mathbb{R}^d}( \kappa_1\Delta  v
- \gamma_1 v +  w)\eta\, dxdt, 
 \\ \vspace{2mm}
\displaystyle 
\lim_{n \to \infty} 
\stint{0}{T}{\mathbb{R}^d}( \kappa_2\Delta  \overline{w}_{\tau_n}
- \gamma_2 \overline{w}_{\tau_n} +  \underline{u}_{\tau_n})\eta\, dxdt
= 
\stint{0}{T}{\mathbb{R}^d}( \kappa_1\Delta  w
- \gamma_1 w +  u)\eta\, dxdt, \\ \vspace{2mm}
\displaystyle 
\lim_{n \to \infty} 
\stint{0}{T}{\mathbb{R}^d}
\langle \nabla ( \kappa_1\Delta  \overline{v}_{\tau_n}
- \gamma_1 \overline{v}_{\tau_n} +  \overline{w}_{\tau_n}), \bm{\xi}\rangle 
\, dxdt
= 
\stint{0}{T}{\mathbb{R}^d}
\langle \nabla ( \kappa_1\Delta  v
- \gamma_1 v +  w), \bm{\xi} \rangle \, dxdt, \\
\displaystyle 
\lim_{n \to \infty}
\stint{0}{T}{\mathbb{R}^d}  \underline{u}_{\tau_n}\eta \,dx
=\stint{0}{T}{\mathbb{R}^d} u \eta\,dx. 
\end{cases}
\end{equation*}
\end{prop}

\begin{proof}
Initially, 
we establish that the integrand of the time integral 
on the right-hand side of the first equation in 
the proposition is an element of $L^{q_*}(0,T)$. 
Applying Fatou's Lemma, we obtain 
\begin{equation*} 
\begin{split} 
\int_0^T\liminf_{n \to \infty}\mathcal{B}_{\tau_n}^{q_*}(t)\,dt
\leq \liminf_{n \to \infty}\int_0^T\mathcal{B}_{\tau_n}^{q_*}(t)\,dt 
\leq \sup_{n \in \mathbb{N}}\int_0^T B_{\tau_n}^{q_*}(t)\,dt<\infty. 
\end{split}
\end{equation*}
Thus, there exists 
an $\mathscr{L}^1$-negligible subset such that
\begin{equation}\label{PWB}
\liminf_{n \to \infty}\mathcal{B}_{\tau_n}^{q_*}(t)<\infty, \quad 
 \text{ for all } t \in [0,T]\setminus \mathcal{N}.
\end{equation}
For every $t_0 \in [0,T]\setminus \mathcal{N}$, 
there exists a subsequence 
$\{\tau_n^0\} \subset \{\tau_n\}$ such that 
\begin{equation} \label{subsq}
\sup_{j \in \mathbb{N}}\mathcal{B}_{\tau_n^0}^{q_*}(t_0)<\infty, 
\quad \lim_{n \to \infty}\mathcal{B}_{\tau_n^0}^{q_*}(t_0) 
= \liminf_{n \to \infty}\mathcal{B}_{\tau_n}^{q_*}(t_0). 
\end{equation}
By Lemma \ref{pointwise}, we can deduce that for 
$\bm{\xi} \in C^{\infty}_c(\mathbb{R}^d;\mathbb{R}^d)$, 
\begin{equation*} 
\begin{split} 
\lim_{n \to \infty}\int_{\mathbb{R}^d} 
\langle \nabla \overline{u}_{\tau_{n}^0}^m(t_0)
-\overline{u}_{\tau_{n}^0}(t_0)
\nabla \overline{v}_{\tau_{n}^0}(t_0), \bm{\xi} \rangle\,dx
=\int_{\mathbb{R}^d} \langle \nabla u^m(t_0) -u(t_0)\nabla v(t_0), \bm{\xi} \rangle\,dx
\end{split}
\end{equation*}
holds. 
Furthermore, 
by H\"older's inequality, the triangle inequality and the inequality \eqref{ugradv}, 
the following estimate holds 
\begin{equation*} 
\begin{split} 
\int_{\mathbb{R}^d}\langle 
\nabla \overline{u}_{\tau_{n}^0}^m(t_0)
-\overline{u}_{\tau_{n}^0}(t_0)
\nabla \overline{v}_{\tau_{n}^0}(t_0), \bm{\xi} \rangle\,dx  
& 
\leq 
\|
\nabla \overline{u}_{\tau_{n}^0}^m(t_0)
-\overline{u}_{\tau_{n}^0}(t_0)
\nabla \overline{v}_{\tau_{n}^0}(t_0)
 \|_{L^{\frac{d-1}{d}}}
  \|\bm{\xi}\|_{L^d} \\ 
& \leq 
( \|\nabla \overline{u}_{\tau_{n}^0}^m(t_0)\|_{L^{\frac{d}{d-1}}}
+ C_5\|\overline{u}_{\tau_{n}^0}(t_0)\|_{L^2})\|\bm{\xi}\|_{L^d}
\\
& \leq C
\mathcal{B}_{\tau_n^0}(t_0)   \|\bm{\xi}\|_{L^d}
\end{split}
\end{equation*}
for some positive constant $C$. 
Thus, by duality and \eqref{subsq}, we deduce that 
\begin{equation*} 
\begin{split} 
\| \nabla u^m(t_0) -u(t_0)\nabla v(t_0)\|_{L^{\frac{d}{d-1}}}
\leq \liminf_{n \to \infty}C\mathcal{B}_{\tau_n}(t_0). 
\end{split}
\end{equation*}
Since this estimate holds for every $t_0 \in [0,T]\setminus \mathcal{N}$, 
we can conclude that
\begin{equation*} 
\begin{split} 
\int_0^T\| \nabla u^m(t) -u(t)\nabla v(t)\|_{L^{\frac{d}{d-1}}}^{q_*}\,dt
& \leq 
C
\int_0^T
\liminf_{n \to \infty}\mathcal{B}_{\tau_n}^{q_*}(t)\,dt
\leq C
\liminf_{n \to \infty}\int_0^T\mathcal{B}_{\tau_n}^{q_*}(t)\,dt<\infty.
\end{split}
\end{equation*}
For arbitrary $\bm{\xi} \in C^{\infty}_c(\mathbb{R}^d \times (0,\infty);\mathbb{R}^d)$, 
let us define 
\begin{equation*} 
\begin{split} 
\rho_n(t):=\left |  \int_{\mathbb{R}^d}
\langle \nabla \overline{u}_{\tau_n}^m
-\overline{u}_{\tau_n}\nabla \overline{v}_{\tau_n}, \bm{\xi}\rangle \,dx
-\int_{\mathbb{R}^d}
\langle \nabla u^m -u \nabla v, \bm{\xi} \rangle \,dx \right | 
\end{split}
\end{equation*}
and note that $\rho_n \in L^{q_*}(0,T)$ with $\sup_{n}\|\rho_n\|_{L^{q_*}}<\infty$.  
According to Lemma \ref{pointwise}, it holds that 
$\rho_n(t)\mathbbm{1}_{[0,T]\setminus S_{\tau_n}}(t) \to 0$ 
for all $t \in [0,T]$. Furthermore, 
by H\"older's inequality
\begin{equation*} 
\begin{split}
\int_{E}\rho_n(t)\mathbbm{1}_{[0,T]\setminus S_{\tau_n}}(t) \,dt 
\leq \left (\mathscr{L}^1(E) \right )^{\frac{q_*-1}{q_*}}
\|\rho_n\|_{L^{q_*}}
\leq (\mathscr{L}^1(E))^{\frac{q_*-1}{q_*}}
\sup_n \|\rho_n\|_{L^{q_*}}, 
\end{split}
\end{equation*}
thereby showing that $\rho_n\mathbbm{1}_{[0,T]\setminus S_{\tau_n}}$ 
is uniform integrable. 
By the Vitali convergence theorem, we have 
\begin{equation*} 
\begin{split} 
\lim_{n \to \infty}\int_{[0,T]\setminus S_{\tau_n}}\rho_n(t)\,dt=
\lim_{n \to \infty}\int_0^T \rho_n(t)\mathbbm{1}_{[0,T]\setminus  S_{\tau_n}}(t)\,dt
=0. 
\end{split}
\end{equation*}
On the other hand, we have 
\begin{equation*} 
\begin{split} 
 \int_{S_{\tau_n}}\rho_n(t)\,dt
 & \leq \|\rho_n\|_{L^{q_*}}
\bigl (
\mathscr{L}^1(S_{\tau_n})
\bigr )
 ^{\frac{q_*-1}{q_*}} \\
 & \leq 
 \bigl (
\mathscr{L}^1(S_{\tau_n})
\bigr )
 ^{\frac{q_*-1}{q_*}}
\sup_n \|\rho_n\|_{L^{q_*}}
  \to 0 \ (K \to \infty). 
\end{split}
\end{equation*}
Therefore, 
\begin{equation*} 
\begin{split} 
\int_0^T \rho_n (t)\,dt 
& = \int_{[0,T]\setminus S_{\tau_n}(K)}\rho_n(t)\,dt + \int_{S_{\tau_n}(K)}\rho_n(t)\,dt
\end{split}
\end{equation*}
where the first term on the right-hand side converges to $0$ as $n \to \infty$ 
for any $K$, and the second term on the right-hand side can be made arbitrarily small 
by choosing $K$ sufficiently large, independently of $n$. 
Thus, the left-hand side approaches $0$ as $n \to \infty$, 
establishing the first assertion of the proposition. 
Similarly, the second, third and fourth assertions can also be demonstrated.
\end{proof}

\section{Proof of Theorems \ref{subcritical} and \ref{critical}} 
From the Taylor expansion, we obtain
\begin{equation*} 
\begin{split} 
|\langle y-x, \nabla \varphi(y) \rangle
-(\varphi(y)-\varphi(x))|
\leq  \frac{\|D^2 \varphi\|_{L^{\infty}}}{2}
|x-y|^2. 
\end{split}
\end{equation*}
Integrating both sides over 
$p_k \in \Gamma_o(u_{\tau}^{k-1}, u_{\tau}^k)$, we obtain 
\begin{multline*}
\displaystyle 
\left |  \int_{\mathbb{R}^d \times \mathbb{R}^d}
\langle y-x, \nabla \varphi(y) \rangle \,Mdp_k(x,y)
 - \int_{\mathbb{R}^d} (u_{\tau}^k-u_{\tau}^{k-1})\varphi \,dx
 \right |
 \leq   
 \frac{\|D^2 \varphi\|_{L^{\infty}}}{2}
\mathcal{W}_2^2(u_{\tau}^k, u_{\tau}^{k-1}). 
\end{multline*}
Thus, for any function 
${\varphi} \in C^{\infty}_c(\mathbb{R}^d \times (0,\infty))$, defining 
\begin{equation*} 
\begin{cases}\vspace{2mm}
\varphi_{\tau}^k :=\varphi(\cdot, k\tau), \\
\dot{\varphi}_{\tau}^k:=\partial_t \varphi(\cdot, k\tau),  
\end{cases}
\begin{cases}  \vspace{2mm}
\overline{\varphi}_{\tau}(t) := \varphi_{\tau}^k, \quad 
 t \in ((k-1)\tau, k\tau], \\ \vspace{2mm}
\underline{\dot{\varphi}}_{\tau}(t):=\dot{\varphi}_{\tau}^{k-1}, 
\quad  t \in ((k-1)\tau, k\tau],  
\end{cases}
\bm{\xi}= \nabla \varphi_{\tau}^k, 
\end{equation*}
we establish from the first equation in \eqref{EL} that 
\begin{equation*} 
\begin{split} 
\left | \int_{\mathbb{R}^d} (u_{\tau}^k-u_{\tau}^{k-1})\varphi_{\tau}^k \,dx
+  \tau \int_{\mathbb{R}^d} \langle \nabla (u_{\tau}^k)^m
-  u_{\tau}^k \nabla v_{\tau}^{k}, \nabla \varphi_{\tau}^k
 \rangle\, dx \right |
\leq 
 \frac{\|D^2_x \varphi\|_{L^{\infty}}}{2}
\mathcal{W}_2^2(u_{\tau}^k, u_{\tau}^{k-1}). 
\end{split}
\end{equation*}
Considering the identity 
\[
(u_{\tau}^k-u_{\tau}^{k-1})\varphi_{\tau}^k
= u_{\tau}^k\varphi_{\tau}^k - u_{\tau}^{k-1}\varphi_{\tau}^{k-1}
- u_{\tau}^{k-1}(\varphi_{\tau}^{k}-\varphi_{\tau}^{k-1})
\]
and the inequality
\[
|(\varphi_{\tau}^{k}-\varphi_{\tau}^{k-1})
- \tau \dot{\varphi}_{\tau}^{k-1}|
\leq \frac{\|\partial_t^2 \varphi\|_{L^{\infty}}}{2}\tau^2, 
\]
summing from $k=1$ to $k=\lceil T/\tau \rceil$, we derive 
\begin{multline}\label{preweakform}
\left | 
\int_{\mathbb{R}^d}
(u_{\tau}^{\lceil T/\tau \rceil}\varphi_{\tau}^{\lceil T/\tau \rceil} 
- u_{\tau}^{0}\varphi_{\tau}^{0})\,dx
- 
\stint{0}{\lceil T/\tau \rceil \tau}{\mathbb{R}^d}
( \underline{u}_{\tau}\underline{\dot{\varphi}}_{\tau}  
-\langle \nabla \overline{u}_{\tau}^m
-  \overline{u}_{\tau} \nabla \overline{v}_{\tau}, \nabla \overline{\varphi}_{\tau}
 \rangle ) \,dxdt
  \right | \\
\leq  \frac{\|D^2_x \varphi\|_{L^{\infty}}}{2}\sum_{k=1}^{\lceil T/\tau \rceil}
\mathcal{W}_2^2(u_{\tau}^k, u_{\tau}^{k-1})
+ \lceil T/\tau \rceil \tau^2 \frac{M\|\partial_t^2 \varphi\|_{L^{\infty}}}{2}. 
\end{multline}
Considering $\{\tau_n\}$ as $\tau$ from Lemma \ref{cpt}, 
as $n \to \infty$, the right-hand side converges to $0$ 
due to Proposition \ref{uniform bounds}. 
Given that $\varphi \in C^{\infty}_c(\mathbb{R}^d) \times (0,\infty)$, 
the first term inside the absolute value on the left-hand side is $0$ 
for sufficiently large $T$, independent of $n$. 
As $n \to \infty$, 
$\underline{\dot{\varphi}}_{\tau_n}$ and $\nabla \overline{\varphi}_{\tau_n}$ 
converge uniformly to $\partial_t \varphi$ and $\nabla \varphi$ respectively. 
Therefore, considering Proposition \ref{L^1}, 
the second term inside that absolute value on the left-hand side converges to 
\[
\stint{0}{T}{\mathbb{R}^d}
[u\partial_t \varphi- \langle \nabla u^m - u \nabla v, \nabla \varphi\rangle  ]\,dxdt, 
\]
which equals to $0$. 
This holds for any sufficiently large $T>0$, 
and by letting $T \to \infty$, 
it follows that $(u,v)$ satisfies the first equation of 
Definition \ref{weak solutions}-(iv). 
From the second and the third equations in \eqref{EL}, 
by similarly summing over $k$ and considering the limit as $n \to \infty$, 
it becomes evident that the limit functions $(u,v,w)$ constitute a weak solution as defined in Definition \ref{weak solutions}.

\section{Proof of Theorem \ref{energy inequality}}

\begin{defn}[De Giorgi variational interpolation {\cite[Definition 3.3.1]{a-g-s}}]
\label{DefinitionDGVI}
The \textbf{De Giorgi variational interpolation} $\Tilde{U}_{\tau}$ 
of $\{u_{\tau}^k\}$ 
is defined by 
\begin{equation*}
\Tilde{U}_{\tau}(t) :=U_{\sigma}^k \   \text{ for }\ 
t= (k-1)\tau + \sigma. 
\end{equation*}
where 
$U_{\sigma}^k$ is defined by 
\begin{equation*}\label{DGVI}
\begin{split}
U_{\sigma}^k & 
\in \argmin_{u \in M\mathscr{P}_2(\mathbb{R}^d)\cap L^m(\mathbb{R}^d)}
\left \{ \mathcal{E}(u, v_{\tau}^k) 
+ \frac{1}{2\sigma}\mathcal{W}_2^2(u, u_{\tau}^{k-1}) \right \}
\end{split}
\end{equation*}
for $ \sigma \in (0,\tau]$ and for $k=1, 2, \cdots, N$. 
\end{defn}
$\overline{u}_{\tau}$ and $\Tilde{U}_{\tau}$ 
take the same values $u_{\tau}^k$ at 
$t \in \{1, 2, \cdots, N\}$ and differ in their interpolation 
for $t\in ((k-1)\tau, k\tau)$, 
but they belong to the same function space and have the same limit. 
\begin{lem}[slope estimate]\label{lem slope estimate2}
Let $U_{\sigma}^k$ be a solution in \eqref{DGVI}. 
Then, 
\begin{equation} \label{slope estimate2}
\begin{split} 
\left ( \int_{\mathbb{R}^d}\frac{|\nabla (U_{\sigma}^k)^m
-U_{\sigma}^k\nabla v_{\tau}^k|^2}{U_{\sigma}^k}\,dx \right )^{\frac{1}{2}}
\leq \frac{\mathcal{W}_2(U_{\sigma}^k,u_{\tau}^{k-1})}{\sigma}
\end{split}
\end{equation}
holds. 
\end{lem}
\begin{proof}
Similarly to the derivation of the first equation in 
Proposition \ref{EL}, we obtain  
\[\int_{\mathbb{R}^d \times \mathbb{R}^d}
\langle y-x, \bm{\xi}(y) \rangle \,MdP_k(x,y)+ 
 \sigma \int_{\mathbb{R}^d} \langle \nabla (U_{\sigma}^k)^m
- U_{\sigma}^k \nabla v_{\tau}^{k}, \bm{\xi} \rangle\, dx  = 
0. \]
Here, $P_k \in \Gamma_o(u_{\tau}^{k-1}, U_{\sigma}^k)$. 
From this, by the same method as the derivation of 
Lemma \ref{Lem slope estimate}, 
we obtain 
\[
\left ( \int_{\mathbb{R}^d}\frac{|\nabla (U_{\sigma}^k)^m
-U_{\sigma}^k\nabla v_{\tau}^k|^2}{U_{\sigma}^k}\,dx \right )^{\frac{1}{2}}
\leq \frac{\mathcal{W}_2(U_{\sigma}^k,u_{\tau}^{k-1})}{\sigma}. 
\]
\end{proof}

\begin{lem}\label{preEI}
Let $\overline{u}_{\tau}$ and $\Tilde{U}_{\tau}$ be the piecewise constant interpolation 
defined in Definition \ref{piecewise constant interpolation} 
and the De Giorgi variational interpolation defined in \ref{DefinitionDGVI}, 
respectively. Then, for any $T>0$ 
the following inequality holds. 
\begin{multline*}
\frac{1}{2}\stint{0}{T}{\mathbb{R}^d} 
\frac{|\nabla \overline{u}_{\tau}^m -  \overline{u}_{\tau} 
\nabla \overline{v}_{\tau}|^2}
{\overline{u}_{\tau}}\,dxdt
+ 
\frac{1}{2}
\stint{0}{T}{\mathbb{R}^d} 
\frac{|\nabla \Tilde{U}_{\tau}^m
-\Tilde{U}_{\tau}\nabla \overline{v}_{\tau}|^2}{\Tilde{U}_{\tau}}\,dxdt \\
+ 
\frac{\varepsilon_1 \kappa_2 + \varepsilon_2 \kappa_1}{\varepsilon_1^2}
\stint{0}{T}{\mathbb{R}^d}
|\nabla 
\left \{ \kappa_1
\Delta \overline{v}_{\tau}(t) -\gamma_1\overline{v}_{\tau}(t)
+ \overline{w}_{\tau}(t) \right \} |^2 \,dxdt \\
+ \frac{\gamma_1 \varepsilon_2 + \gamma_2 \varepsilon_1}{\varepsilon_1^2}
\stint{0}{T}{\mathbb{R}^d}|
\kappa_1 \Delta \overline{v}_{\tau}(t) -\gamma_1\overline{v}_{\tau}(t)
+ \overline{w}_{\tau}(t)|^2 \,dxdt \\
\leq \mathcal{L}(u_{\tau}^{0}, v_{\tau}^{0}, w_{\tau}^{0}) 
- \mathcal{L}
(u_{\tau}^{\lceil T/\tau \rceil }, 
v_{\tau}^{\lceil T/\tau \rceil }, 
w_{\tau}^{\lceil T/\tau \rceil }). 
\end{multline*}
\end{lem}
\begin{proof}
By applying the derivative of Moreaux-Yosida approximation 
\cite[Theorem 3.1.4]{a-g-s} 
to the third minimizing problem in \eqref{mm}, 
it holds for $k=1, 2, \ldots$ that 
\begin{equation*}
\begin{split}
\frac{\mathcal{W}_2^2(u_{\tau}^k, u_{\tau}^{k-1})}{2\tau}
+\int_0^{\tau}\frac{\mathcal{W}_2^2(U_{\sigma}^k, u_{\tau}^{k-1})}{2\sigma^2}\,d\sigma 
& = \mathcal{E}(u_{\tau}^{k-1},v_{\tau}^k) - \mathcal{E}(u_{\tau}^k,v_{\tau}^k) \\
& = \mathcal{L}(u_{\tau}^{k-1}, v_{\tau}^{k-1}, w_{\tau}^{k-1}) 
- \mathcal{L}(u_{\tau}^k, v_{\tau}^k, w_{\tau}^k)
-\mathcal{D}_{\tau}^k, 
\end{split}
\end{equation*}  
Summing over $k$, we have 
\begin{multline*}
\sum_{k=1}^{\lceil T/\tau \rceil }
\frac{\mathcal{W}_2^2(u_{\tau}^k, u_{\tau}^{k-1})}{2\tau}
+
\sum_{k=1}^{\lceil T/\tau \rceil }
\int_0^{\tau}\frac{\mathcal{W}_2^2(U_{\sigma}^k, u_{\tau}^{k-1})}{2\sigma^2}\,d\sigma 
+ \sum_{k=1}^{\lceil T/\tau \rceil }
\mathcal{D}_{\tau}^k \\
 = \mathcal{L}(u_{\tau}^{0}, v_{\tau}^{0}, w_{\tau}^{0}) 
- \mathcal{L}
(u_{\tau}^{\lceil T/\tau \rceil}, 
v_{\tau}^{\lceil T/\tau \rceil }, 
w_{\tau}^{\lceil T/\tau \rceil }). 
\end{multline*}
The three terms on the left-hand side are bounded below as follows, according to 
the inequalities \eqref{slope estimate} and \eqref{slope estimate2},  
and the representation of $\mathcal{D}_{\tau}^k$.
\begin{align*}
& \sum_{k=1}^{\lceil T/\tau \rceil}
\frac{\mathcal{W}_2^2(u_{\tau}^k, u_{\tau}^{k-1})}{2\tau}
\geq 
\sum_{k=1}^{\lceil T/\tau \rceil}
 \frac{\tau}{2} 
\int_{\mathbb{R}^d} 
\frac{| \nabla (u_{\tau}^k)^m -  u_{\tau}^k \nabla v_{\tau}^k|^2}{u_{\tau}^k}\,dx
\geq \frac{1}{2} 
\stint{0}{T}{\mathbb{R}^d} 
\frac{|\nabla \overline{u}_{\tau}^m -  \overline{u}_{\tau} 
\nabla \overline{v}_{\tau}|^2}
{\overline{u}_{\tau}}\,dxdt.  \\
&\sum_{k=1}^{\lceil T/\tau \rceil}\int_0^{\tau}
\frac{\mathcal{W}_2^2(U_{\sigma}^k, u_{\tau}^{k-1})}{2\sigma^2}\,d\sigma 
 \geq 
 \frac{1}{2}
\sum_{k=1}^{\lceil T/\tau \rceil}
\stint{0}{\tau}{\mathbb{R}^d}
\frac{|\nabla (U_{\sigma}^k)^m-U_{\sigma}^k\nabla v_{\tau}^k|^2}{U_{\sigma}^k}\,dxd\sigma  \\
& \hspace{42mm}=\frac{1}{2} \sum_{k=1}^{\lceil T/\tau \rceil}
\stint{(k-1)\tau}{k\tau}{\mathbb{R}^d} 
\frac{|\nabla \Tilde{U}_{\tau}^m
-\Tilde{U}_{\tau}\nabla \overline{v}_{\tau}|^2}{\Tilde{U}_{\tau}}\,dxdt 
\\ & 
\hspace{42mm} \geq \frac{1}{2}\stint{0}{T}{\mathbb{R}^d} 
\frac{|\nabla \Tilde{U}_{\tau}^m
-\Tilde{U}_{\tau}\nabla \overline{v}_{\tau}|^2}{\Tilde{U}_{\tau}}\,dxdt,  
\\
&\sum_{k=1}^{\lceil T/\tau \rceil}\mathcal{D}_{\tau}^k
 \geq 
(\varepsilon_1 \kappa_2 + \varepsilon_2 \kappa_1)
\sum_{k=1}^{\lceil T/\tau \rceil}\tau \|\partial_t \nabla v_{\tau}^k \|_{L^2}^2 
+ 
(\gamma_1 \varepsilon_2 + \gamma_2 \varepsilon_1)
\sum_{k=1}^{\lceil T/\tau \rceil}\tau \|\partial_t v_{\tau}^k \|_{L^2}^2 \\
& \hspace{12mm} \geq 
(\varepsilon_1 \kappa_2 + \varepsilon_2 \kappa_1)
\int_0^T\|\nabla \partial_t \overline{v}_{\tau}(t) \|_{L^2}^2 \,dt
+ (\gamma_1 \varepsilon_2 + \gamma_2 \varepsilon_1)
\int_0^T\|\partial_t \overline{v}_{\tau}(t) \|_{L^2}^2 \,dt
 \\
& \hspace{12mm} =
\frac{\varepsilon_1 \kappa_2 + \varepsilon_2 \kappa_1}{\varepsilon_1^2}
\int_0^T
\|\nabla 
\left \{ \kappa_1
\Delta \overline{v}_{\tau}(t) -\gamma_1\overline{v}_{\tau}(t)
+ \overline{w}_{\tau}(t) \right \} \|_{L^2}^2 \,dt \\
& \hspace{42mm}
+ \frac{\gamma_1 \varepsilon_2 + \gamma_2 \varepsilon_1}{\varepsilon_1^2}
\int_0^T\|
\kappa_1 \Delta \overline{v}_{\tau}(t) -\gamma_1\overline{v}_{\tau}(t)
+ \overline{w}_{\tau}(t)\|_{L^2}^2 \,dt. 
\end{align*}
By organizing the above, we obtain the inequality stated in the lemma.
\end{proof}

\begin{lem}\label{properties of DGVI} 
Assume \eqref{uniform L^2} and \eqref{lbc}. 
Let $\{\tau_n\}$ 
and $u$ be that in Lemma \ref{cpt}. 
Then, the following holds: 
\begin{equation*} 
\begin{cases}\vspace{5mm}
\displaystyle
\Tilde{U}_{\tau_n}(t) 
\rightharpoonup u(t) \ \text{weakly in } (L^1\cap L^m)(\mathbb{R}^d) 
\   \text{ for any }t >0, 
\\
\displaystyle
\sup_{n} 
\left (
\int_0^T\|\Tilde{U}_{\tau_n}(t)\|_{L^2}^2
\,dt
+ \int_0^T
\|\nabla \Tilde{U}_{\tau_n}^{m}(t)\|_{L^{\frac{d}{d-1}}}^{p_*}
\,dt
\right )<\infty,  \text{ for any }T>0.
\end{cases}
\end{equation*}
\end{lem}

\begin{proof}
The proof is divided into two parts.

\vspace{2mm}

\noindent
\textbf{(i) convergence} 

\vspace{2mm}

From the definition of 
$U_{\sigma}^k$, we have
\begin{equation*}
\begin{split}
&\mathcal{E}(U_{\sigma}^k, v_{\tau}^k) + 
\frac{1}{2\sigma}\mathcal{W}_2^2(U_{\sigma}^k, u_{\tau}^{k-1}) \\
&\leq  \mathcal{E}(u_{\tau}^k, v_{\tau}^k) 
+ \frac{1}{2\sigma}\mathcal{W}_2^2(u_{\tau}^k, u_{\tau}^{k-1}) \\
&= \left \{
\mathcal{E}(u_{\tau}^k, v_{\tau}^k) 
+ \frac{1}{2\tau}\mathcal{W}_2^2(u_{\tau}^k, u_{\tau}^{k-1}) 
\right \} + \frac{1}{2}
\left( \frac{1}{\sigma} - \frac{1}{\tau} \right )\mathcal{W}_2^2(u_{\tau}^k, u_{\tau}^{k-1}) 
\\ 
& \leq \mathcal{E}(U_{\sigma}^k, v_{\tau}^k) 
+ \frac{1}{2\tau}\mathcal{W}_2^2(U_{\sigma}^k, u_{\tau}^{k-1})
+\frac{1}{2}\left ( \frac{1}{\sigma} - \frac{1}{\tau} \right ) 
\mathcal{W}_2^2(u_{\tau}^k, u_{\tau}^{k-1}) 
\end{split}
\end{equation*}
from which we obtain
$\mathcal{W}_2(U_{\sigma}^k, u_{\tau}^{k-1}) 
\leq \mathcal{W}_2(u_{\tau}^k, u_{\tau}^{k-1})$. 
Consequently, for $t =(k-1)\tau + \sigma$, $\sigma \in (0,\tau]$, 
using the triangle inequality, the above inequality, Young's inequality, 
and Proposition \ref{ED}, we have 
\begin{equation}\label{CSI/DGVI}
\begin{split}
\mathcal{W}_2^2(\Tilde{U}_{\tau}(t), \overline{u}_{\tau}(t)) 
& = \mathcal{W}_2^2(U_{\sigma}^k, u_{\tau}^k) \\
& \leq 
\bigl ( \mathcal{W}_2(U_{\sigma}^k, u_{\tau}^{k-1}) 
+ \mathcal{W}_2(u_{\tau}^k, u_{\tau}^{k-1}) 
\bigr )^2 \\
& \leq 4 \mathcal{W}_2^2(u_{\tau}^k, u_{\tau}^{k-1}) \leq 
8 \tau \left (\frac{1}{2\tau}\sum_{k=1}^N
\mathcal{W}_2^2(u_{\tau}^k, u_{\tau}^{k-1}) \right )
\\
& \leq 8 \tau 
\left [ \mathcal{L}(u_{\tau}^0, v_{\tau}^0, w_{\tau}^0) 
-\mathcal{L}(u_{\tau}^N, v_{\tau}^N, w_{\tau}^N) \right  ] 
\to 0\ (\tau \to 0). 
\end{split}
\end{equation}
Again, from the definition of $U_{\sigma}^k$, we have 
\begin{equation*} 
\begin{split} 
\mathcal{E}(U_{\sigma}^k, v_{\tau}^k) 
+ \frac{1}{2\sigma}\mathcal{W}_2^2(U_{\sigma}^k, u_{\tau}^{k-1})
& \leq \mathcal{E}(u_{\tau}^{k-1}, v_{\tau}^k) 
= \frac{1}{m-1}\int_{\mathbb{R}^d}(u_{\tau}^{k-1})^m\,dx
- \int_{\mathbb{R}^d}u_{\tau}^{k-1}v_{\tau}^k\,dx \\
& \leq 
\frac{1}{m-1}\int_{\mathbb{R}^d}(u_{\tau}^{k-1})^m\,dx 
+ C_*M^{1-\theta}\|u_{\tau}^{k-1}\|_{L^m}^{\theta}\|\Delta v_{\tau}^k\|_{L^m},
\end{split}
\end{equation*}
where 
$C_*$ is defined by \eqref{CM}. 
By Proposition \ref{uniform bounds}, the right-hand side is bounded 
independently of $\tau$, $\sigma$, and $k$, 
so $\mathcal{E}(U_{\sigma}^k, v_{\tau}^k)$ is also bounded.  

From \eqref{coercivity of L}, we have 
\begin{equation*} 
\begin{split} 
\mathcal{L}(U_{\sigma}^k,v_{\tau}^k,w_{\tau}^k)
& =\mathcal{E}(U_{\sigma}^k,v_{\tau}^k)
+ 
\frac{\kappa_1 \kappa_2}{2}\|\Delta v_{\tau}^k\|_{L^2}^2 
+ \frac{\gamma_1 \kappa_2 + \gamma_2 \kappa_1}{2}
\|\nabla v_{\tau}^k\|_{L^2}^2
+ \frac{\gamma_1 \gamma_2}{2}
\|v_{\tau}^k\|_{L^2}^2 \\
& \hspace{5mm}+ 
\displaystyle 
\frac{\varepsilon_2}{2\varepsilon_1}
\|\kappa_1 \Delta v_{\tau}^k - \gamma_1 v_{\tau}^k + w_{\tau}^k\|_{L^2}^2 \\
& \geq 
C_1\|U_{\sigma}^k\|_{L^m}^m
\end{split}
\end{equation*}
for some positive constant $C_1$. 
Therefore, considering \eqref{coercivity of L}, 
we see that $\|U_{\sigma}^k\|_{L^m}$ is bounded independently of 
$\tau$, $\sigma$, and $k$. 
Hence, 
\[
\sup_{\substack{t\in [0,T] \\ \tau \in (0, \tau_*)}} \|\Tilde{U}_{\tau}(t)\|_{L^m}<\infty. 
\]
Now, if $\{ \tau_n \}$ is the sequence from Proposition \ref{cpt}, 
then for any fixed $t_0 \in [0,T]$, 
there exists a subsequence $\{\tau_n'\} \subset \{\tau_n\}$ and 
$\Tilde{U}(t_0)$ such that 
\[
\Tilde{U}_{\tau_n'}(t_0) \rightharpoonup \Tilde{U}(t_0)
\quad \text{weakly in }(L^1\cap L^m)(\mathbb{R}^d). 
\]
By inequality \eqref{CSI/DGVI} and 
the weak lower semicontinuity of the Wasserstein distance \cite[Lemma 7.1.4]{a-g-s}, 
we have 
\[
\mathcal{W}_2(\Tilde{U}(t_0), u(t_0))
\leq \liminf_{n \to \infty}
\mathcal{W}_2(\Tilde{U}_{\tau_n'}(t_0), u_{\tau_n'}(t_0))=0. 
\]
Thus, $\Tilde{U}(t_0)= u(t_0)$. 
This holds for any subsequence of $\{\Tilde{U}_{\tau_n}(t_0)\}$, 
so the result holds without extracting subsequences. 
Therefore, for any $t \in [0,T]$, 
\[
\Tilde{U}_{\tau_n}(t) \rightharpoonup u(t)
\quad \text{weakly in }(L^1\cap L^m)(\mathbb{R}^d). 
\]

\vspace{2mm}

\noindent 
\textbf{(ii) uniform boundedness }

\vspace{2mm}

Let $1<m<2$. 
Following the same argument as in the derivation of inequalities 
\eqref{a priori} and \eqref{uL2PW}, we obtain 
\begin{equation*} 
\frac{C_3^2}{2}  \|U_{\sigma}^k\|_{L^2}^{\frac{4m+2(m-1)d}{d(2-m)}}
\leq 2C_4^2\int_{\mathbb{R}^d}\frac{|\nabla (U_{\sigma}^k)^m
-U_{\sigma}^k\nabla v_{\tau}^k|^2}{U_{\sigma}^k}\,dx
+C_6. 
\end{equation*}
By integrating both sides with respect to $\sigma$ over $(0,\tau]$ 
and summing from $k=1$ to ${\lceil T/\tau \rceil}$, we obtain 
\begin{equation*} 
\frac{C_3^2}{2}  
\sum_{k=1}^{\lceil T/\tau \rceil}
\int_0^{\tau}\|U_{\sigma}^k\|_{L^2}^{\frac{4m+2(m-1)d}{d(2-m)}}\,
d\sigma
\leq 2C_4^2
\sum_{k=1}^{\lceil T/\tau \rceil}
\stint{0}{\tau}{\mathbb{R}^d}\frac{|\nabla (U_{\sigma}^k)^m
-U_{\sigma}^k\nabla v_{\tau}^k|^2}{U_{\sigma}^k}\,dxd\sigma
+C_6(T+\tau_*). 
\end{equation*}
This implies 
\begin{equation*} 
\frac{C_3^2}{2}  
\int_0^T \|\Tilde{U}_{\tau_n}\|_{L^2}^{\frac{4m+2(m-1)d}{d(2-m)}}\,
dt
 \leq 2C_4^2
\stint{0}{T+\tau_*}{\mathbb{R}^d} 
\frac{|\nabla \Tilde{U}_{\tau_n}^m
-\Tilde{U}_{\tau_n}\nabla \overline{v}_{\tau_n}|^2}{\Tilde{U}_{\tau_n}}\,dxdt
+C_6(T+\tau_*). 
\end{equation*}
The right-hand side is bounded independently of $n$ 
by Lemma \ref{preEI}. 
Due to \eqref{uniform L^2}, 
the exponent of the power of the integrand on the left-hand side 
is at least 2. 
Consequently, the result of lemma follows 
in a similar manner to Corollary \ref{umW1p}.  
\end{proof}

\begin{lem}\label{quadratic lsc}
Let $(\overline{u}_{\tau_n},\overline{v}_{\tau_n}, \overline{w}_{\tau_n})$ 
and $(u, v, w)$ be that in Lemma \ref{cpt}. 
Then, for any $T>0$ it holds that 
\begin{equation*} 
\begin{cases} \vspace{2mm}
\displaystyle 
\stint{0}{T}{\mathbb{R}^d} 
\frac{|\nabla u^m -  u 
\nabla v|^2}
{u}\,dxdt
\leq \liminf_{n \to \infty}
\stint{0}{T}{\mathbb{R}^d} 
\frac{|\nabla \overline{u}_{\tau_n}^m -  \overline{u}_{\tau_n} 
\nabla \overline{v}_{\tau_n}|^2}
{\overline{u}_{\tau_n}}\,dxdt, 
\\ \vspace{2mm}
\displaystyle 
\stint{0}{T}{\mathbb{R}^d} 
\frac{|\nabla u^m -  u 
\nabla v|^2}
{u}\,dxdt
\leq \liminf_{n\to \infty}
\stint{0}{T}{\mathbb{R}^d} 
\frac{|\nabla \Tilde{U}_{\tau}^m
-\Tilde{U}_{\tau}\nabla \overline{v}_{\tau}|^2}{\Tilde{U}_{\tau}}\,dxdt, 
\\ \vspace{2mm}
\displaystyle
\stint{0}{T}{\mathbb{R}^d}
|\kappa_1 \Delta v -\gamma_1v
+ w|^2 \,dxdt
\leq \liminf_{n \to \infty}
\stint{0}{T}{\mathbb{R}^d}
|\kappa_1 \Delta \overline{v}_{\tau_n} -\gamma_1\overline{v}_{\tau_n}
+ \overline{w}_{\tau_n}|^2 \,dxdt, \\
\displaystyle
\stint{0}{T}{\mathbb{R}^d}
|\nabla 
\left (
\kappa_1 \Delta v -\gamma_1v
+ w \right ) |^2 \,dxdt
\leq 
\liminf_{n \to \infty}\stint{0}{T}{\mathbb{R}^d}
|\nabla 
\left (
\kappa_1 \Delta \overline{v}_{\tau_n} -\gamma_1\overline{v}_{\tau_n}
+ \overline{w}_{\tau_n} \right ) |^2 \,dxdt. 
\end{cases}
\end{equation*}
\end{lem}

\begin{proof}
Since they can all be proven in a similar manner, 
we will only demonstrate the second inequality. 
From Lemma \ref{properties of DGVI}, 
we know that the first equation of  
Proposition \ref{L^1} holds even when 
$\{u_{\tau_n}\}$ is replaced by $\Tilde{U}_{\tau_n}$. 
Therefore, 
\begin{equation*} 
\begin{split} 
& \stint{0}{T}{\mathbb{R}^d} 
\langle \nabla u^m -u\nabla v, \bm{\xi} \rangle\,dxdt 
 = 
\liminf_{n \to \infty}
\stint{0}{T}{\mathbb{R}^d}
 \langle \nabla \Tilde{U}_{\tau_n}^m
-\Tilde{U}_{\tau_n}\nabla \overline{v}_{\tau_n}, \bm{\xi} \rangle\,dxdt \\
& \leq \liminf_{n \to \infty}
\left ( \stint{0}{T}{\mathbb{R}^d} 
\frac{|\nabla \Tilde{U}_{\tau_n}^m
-\Tilde{U}_{\tau_n}\nabla \overline{v}_{\tau}|^2}{\Tilde{U}_{\tau_n}}\,dxdt
\right )^{\frac{1}{2}}
\left ( \stint{0}{T}{\mathbb{R}^d} 
|\bm{\xi}|^2\Tilde{U}_{\tau_n}\,dxdt 
\right )^{\frac{1}{2}} \\
& = \liminf_{n \to \infty}
\left ( \stint{0}{T}{\mathbb{R}^d} 
\frac{|\nabla \Tilde{U}_{\tau_n}^m
-\Tilde{U}_{\tau_n}\nabla \overline{v}_{\tau}|^2}{\Tilde{U}_{\tau_n}}\,dxdt
\right )^{\frac{1}{2}}
\left ( \stint{0}{T}{\mathbb{R}^d} 
|\bm{\xi}|^2u\,dxdt 
\right )^{\frac{1}{2}}. 
\end{split}
\end{equation*}
By duality, the second inequality of the lemma holds. 
\end{proof}

\begin{proof}[Proof of Theorem \ref{energy inequality}]
Substituting $\tau_n$ from Lemma \ref{cpt} into $\tau$ in 
Lemma \ref{preEI} and taking the limit as $n \to \infty$, 
we see from Lemma \ref{quadratic lsc} that 
the desired inequality holds. 
\end{proof}

\bibliography{chimera}

\end{document}